%% file: Main.tex
\documentclass{article}
\usepackage{authblk}
\usepackage{natbib}
\usepackage{fullpage}
\input{arxiv_style}

\input{math_commands}

\newcommand{\defn}{:=}
\renewcommand{\Normal}[2]{\mathsf{N}\left(#1, #2\right)}

\newcommand{\dimension}{d}
\renewcommand{\numobs}{n}
\renewcommand{\T}{\mathsf{T}}

\newcommand*{\ud}{\mathrm{\,d}}     % differential symbol for integrals
\renewcommand{\ft}{w_\star}
\let\epsilon\varepsilon
\title{Minimum Norm Interpolation via The Local Theory of Banach Spaces: The Role of Gaussianity}
\author[1]{Gil Kur\footnote{GK (leading author) was partially supported by NSF-ECCS-2217023 during his visit to the IDEAL Institute}}
\author[2, 3]{Reese Pathak\footnote{RP was partially supported by NSF-DMS-2503579}}
\affil[1]{Institute for Machine Learning, ETH Z\"urich}
\affil[2]{Department of Statistics, UC Berkeley}
\affil[3]{School of Operations Research and Information Engineering (ORIE), Cornell University}
\date{}
\begin{document}
\maketitle
\input{Abstract.tex}
\input{Intro.tex}
\input{MR.tex}
% \input{example}
\input{discussion}

\input{Acknow}

\bibliographystyle{plainnat}
\bibliography{Bib.bib}
\input{proof_outline_thm1.tex}

\input{proof_outline_thm2.tex}

\input{proof_outline_thm3.tex}
\input{Missing.tex}
\input{proof_outline_thm4.tex}
\end{document}

%% file: arxiv_style.tex
\usepackage[letterpaper, left=1in, right=1in, top=1in, bottom=1in]{geometry}

\PassOptionsToPackage{hypertexnames=false}{hyperref}  % compatible with cref for multiple algo
\usepackage{parskip}
\usepackage{authblk}

\usepackage[colorlinks=true, linkcolor=blue!70!black, citecolor=blue!70!black,urlcolor=black,breaklinks=true]{hyperref}
\usepackage{microtype}
\usepackage{hhline}

\makeatletter
\newcommand{\neutralize}[1]{\expandafter\let\csname c@#1\endcsname\count@}
\makeatother

\usepackage{algorithm}
\usepackage{amsthm}
\usepackage{mathtools}
\usepackage{amsmath}
\usepackage{bbm}
\usepackage{amsfonts}
\usepackage{amssymb}

\usepackage{xpatch}

\newtheorem*{theorem*}{Theorem}
\newtheorem{theorem}{Theorem}

\newtheorem{lemma}{Lemma}

\newtheorem{assumption}{Assumption}
\newtheorem{corollary}{Corollary}

\newtheorem{definition}{Definition}
\theoremstyle{definition}
\newtheorem{fact}{Fact}

\newtheorem{example}{Example}

\theoremstyle{remark}
\newtheorem{remark}{Remark}
\AtBeginEnvironment{remark}{%
  \pushQED{\qed}%
}
\AtEndEnvironment{remark}{\popQED\endexample}

\makeatletter
  \renewenvironment{proof}[1][Proof]%
  {%
   \par\noindent{\bfseries\upshape {#1.}\ }%
  }%
  {\qed\newline}
  \makeatother

\newtheorem{openproblem}{Open Problem}

\xpatchcmd{\proof}{\itshape}{\normalfont\proofnameformat}{}{}
\newcommand{\proofnameformat}{\bfseries}

\usepackage[nameinlink,capitalize]{cleveref}

\crefformat{equation}{#2(#1)#3}
\Crefformat{equation}{#2(#1)#3}

\Crefformat{figure}{#2Figure #1#3}
\Crefname{assumption}{Assumption}{Assumptions}
\Crefformat{assumption}{#2Assumption #1#3}

\usepackage{crossreftools}
\usepackage{xparse}

\ExplSyntaxOn
\DeclareDocumentCommand{\XDeclarePairedDelimiter}{mm}
 {
  \__egreg_delimiter_clear_keys: % reset to the default
  \keys_set:nn { egreg/delimiters } { #2 }
  \use:x % we want to expand the values of the token variables set with the keys
   {
    \exp_not:n {\NewDocumentCommand{#1}{sO{}m} }
     {
      \exp_not:n { \IfBooleanTF{##1} }
       {
        \exp_not:N \egreg_paired_delimiter_expand:nnnn
         { \exp_not:V \l_egreg_delimiter_left_tl }
         { \exp_not:V \l_egreg_delimiter_right_tl }
         { \exp_not:n { ##3 } }
         { \exp_not:V \l_egreg_delimiter_subscript_tl }
       }
       {
        \exp_not:N \egreg_paired_delimiter_fixed:nnnnn 
         { \exp_not:n { ##2 } }
         { \exp_not:V \l_egreg_delimiter_left_tl }
         { \exp_not:V \l_egreg_delimiter_right_tl }
         { \exp_not:n { ##3 } }
         { \exp_not:V \l_egreg_delimiter_subscript_tl }
       }
     }
   }
 }

\keys_define:nn { egreg/delimiters }
 {
  left      .tl_set:N = \l_egreg_delimiter_left_tl,
  right     .tl_set:N = \l_egreg_delimiter_right_tl,
  subscript .tl_set:N = \l_egreg_delimiter_subscript_tl,
 }

\cs_new_protected:Npn \__egreg_delimiter_clear_keys:
 {
  \keys_set:nn { egreg/delimiters } { left=.,right=.,subscript={} }
 }

\cs_new_protected:Npn \egreg_paired_delimiter_expand:nnnn #1 #2 #3 #4
 {% Fix the spacing issue with \left and \right (D. Arsenau, P. Stephani and H. Oberdiek)
  \mathopen{}
  \mathclose\c_group_begin_token
   \left#1
   #3
   \group_insert_after:N \c_group_end_token
   \right#2
   \tl_if_empty:nF {#4} { \c_math_subscript_token {#4} }
 }
\cs_new_protected:Npn \egreg_paired_delimiter_fixed:nnnnn #1 #2 #3 #4 #5
 {
  \mathopen{#1#2}#4\mathclose{#1#3}
  \tl_if_empty:nF {#5} { \c_math_subscript_token {#5} }
 }
\ExplSyntaxOff

\XDeclarePairedDelimiter{\supnorm}{
  left=\lVert,
  right=\rVert,
  subscript=\infty
  }
\usepackage[utf8]{inputenc} % allow utf-8 input
\usepackage[T1]{fontenc}    % use 8-bit T1 fonts
\usepackage{url}            % simple URL typesetting
\usepackage{booktabs}       % professional-quality tables
\usepackage{amsfonts}       % blackboard math symbols
\usepackage{nicefrac}       % compact symbols for 1/2, etc.
\usepackage{microtype}      % microtypography
\usepackage{authblk}
\usepackage{tocloft}            % TOC spacing

\usepackage{enumitem}

\usepackage{breakcites}
\usepackage{dsfont}
\usepackage{etoolbox}
\usepackage{comment}
\newtoggle{draft}
\togglefalse{draft}

\usepackage{color-edits}
\usepackage{mathrsfs}

\usepackage{algorithm}
\usepackage{verbatim}
\usepackage[noend]{algpseudocode}

\usepackage{multicol}

\usepackage{colortbl}
\usepackage{bbm}
\usepackage{setspace}

\usepackage{transparent}

\usepackage{inconsolata}
\usepackage[scaled=.90]{helvet}
\usepackage{xspace}

\usepackage{pifont}

\usepackage{tikz-cd}
    \newcommand{\lp}{\left(}
    \newcommand{\rp}{\right)}
    
    \newcommand{\Med}{\mathrm{Med}\,}
    \renewcommand{\P}{{\mathbb P}}

    \newcommand{\T}{{\mathcal{T}}}

    \newcommand{\numobs}{n}

    \newcommand{\Normal}{\mathcal{N}}

    \renewcommand{\b}[1]{\boldsymbol{\mathbf{#1}}}
    \renewcommand{\vec}[1]{\b{#1}}

    \newcounter{rcnt}[section]

    \def\argmin{\mathop{\rm argmin}}

    \newcommand{\Vol}{\operatorname{Vol}}

    \newcommand{\erm}{\widehat{w}_n}
    \newcommand{\ft}{w_{*}}

    \newcommand{\lln}{\mathrm{lln}}

    \newcommand{\Gr}{{\operatorname{Gr_{n,d}}}}

    \def\ddefloop#1{\ifx\ddefloop#1\else\ddef{#1}\expandafter\ddefloop\fi}
    \def\ddef#1{\expandafter\def\csname c#1\endcsname{\ensuremath{\mathcal{#1}}}}
    \ddefloop ABCDEFGHIJKLMNOPQRSTUVWXYZ\ddefloop

    \newcommand{\SPO}{\bar{P}_{n,d}}

    \newcommand{\Sn}{\mathbb{S}^{n-1}}

%% file: math_commands.tex
\usepackage{amsmath,amsfonts,bm}

\newcommand{\opnorm}[1]{\left\|#1\right\|_{\rm op}}
\newcommand{\fronorm}[1]{\left\|#1\right\|_{\rm HS}}

\def\eqref#1{(\ref{#1})}
\def\1{\bm{1}}

\newcommand{\PN}{P_{n,d}}
\newcommand{\conv}{\mathrm{conv}}
\newcommand{\MND}{\widetilde{M}_{n,d}}
\def\eps{{\epsilon}}

\def\vc{{\bm{c}}}

\def\vn{{\bm{n}}}

\DeclareMathAlphabet{\mathsfit}{\encodingdefault}{\sfdefault}{m}{sl}
\SetMathAlphabet{\mathsfit}{bold}{\encodingdefault}{\sfdefault}{bx}{n}

\newcommand{\E}{\mathbb{E}}

\newcommand{\R}{\mathbb{R}}

\newcommand{\Var}{\mathrm{Var}}

\newcommand{\Pbb}{\Pr}
  
\newcommand{\Zbar}{\overline{Z}}
\newcommand{\simplex}{\triangle}

\newcommand{\cnd}{\tilde{c}_{n,d}}
\newcommand{\Cov}{\mathrm{Cov}}
\DeclareMathOperator{\sign}{sign}
\DeclareMathOperator{\Tr}{Tr}

%% file: Abstract.tex
\begin{abstract}
We study minimum-norm interpolation (MNI) in overparameterized linear regression with isotropic Gaussian covariates, in settings where the MNI has no closed-form formula. Whereas most prior work relied on Gaussian comparison tools such as the convex Gaussian min--max theorem (CGMT), our approach uses tools from high-dimensional geometry and probability. First, when the norm is in isotropic position, we obtain an ``offset'' bound that controls the amount by which the MNI shrinks the ground truth. Second, we show that the ``intrinsic'' variance of the $\ell_1$-MNI is at most $O(\tfrac{1}{n\log(d/n)^2})$, using a variant of Talagrand's $L_1$--$L_2$ inequality due to \cite{cordero2012hypercontractive}, together with a classical result of \cite{gluskin88extremal}. We recover the sharp mean-squared error (MSE) bound for the $\ell_1$-MNI obtained by \cite{wang2022tight}, using the work of \cite{fleury2012poincare} on the symmetric Gaussian polytope, which is defined via
\[
  \PN := \mathrm{conv}\{\pm X_i\}_{i=1}^{d},
  \qquad
  X_i \overset{\mathrm{i.i.d.}}{\sim} N(0,\mathrm{I}_{n \times n}),
\]
rather than CGMT.  Our methods also imply improvements on previous results in high-dimensional geometry that may be of independent interest. First, we show that with overwhelming probability, the ratio between the isotropic constant of $\PN$ and that of the Euclidean ball in $\R^n$ is at most $1+O((\log(d/n))^{-2})$, improving a result of \cite{klartag2009hyperplane}.  We also establish a refined weighted thin-shell estimate on $\PN$, and provide an elementary proof of the main theorem of \cite[Thm.~1.1]{fleury2012poincare}.\footnote{Preliminarily work titled ``A New Perspective on Minimum Norm Interpolation Under Gaussian Covariates'' appeared at AISTATS 2026. In the forthcoming weeks, a new version will be uploaded with a few more results and a more detailed discussion.}
\end{abstract}

%% file: Intro.tex
\section{Introduction}
Recent experiments with neural networks have revealed a striking statistical phenomenon: models that interpolate the training data can nevertheless generalize well~\citep{nakkiran2021}. In overparameterized regimes, models achieving zero training error---even on noisy data---may still exhibit strong out-of-sample performance~\citep{zhang2021understanding}. Motivated by these observations, a substantial body of theoretical work has studied this behavior, often referred to as \emph{harmless interpolation} or \emph{benign overfitting}.

In regression---the focus of this paper---an \emph{interpolating} estimator is one that fits the training data exactly. Such an estimator is generally not unique: in linear models, for instance, the interpolation constraints determine the estimate only up to the null space of the design matrix. It is well known that the particular interpolating solution can dramatically affect generalization performance~\citep{donhauser2022fast}. A common selection rule is the \emph{minimum-norm interpolator} (MNI), which chooses, among all interpolators, one with the smallest norm.

A further motivation for studying MNIs comes from the implicit bias of first-order methods: in overparameterized settings, and under suitable initialization, such methods often converge to particular minimum-norm solutions~\citep{gunasekar2018characterizing, oravkin2021optimal, shamir2022implicit}.

We consider the standard linear regression model with isotropic Gaussian covariates. We observe
\[
\cD := \{(\vec{x}_i, y_i)\}_{i=1}^\numobs, \qquad
\vec{y} = \vec{X}\ft + \vec{\xi},
\]
where $\vec{X} \in \R^{\numobs \times \dimension}$ is the design matrix whose rows are i.i.d.\ isotropic Gaussian vectors $\Normal{0}{I_{\dimension}}$. We focus on the overparameterized setting, where $d>n$. Unless stated otherwise, the noise vector $\vec{\xi}=(\xi_1,\dots,\xi_{\numobs})$ is isotropic Gaussian and independent of $\vec{X}$. We write $\vec{X}=[\vec{x}_1,\dots,\vec{x}_{\dimension}]$ for the columns of the design matrix.

For a norm $\|\cdot\|:\R^{\dimension}\to\R_+$, the  MNI is defined by
\begin{equation}\label{eqn:min-norm-linear}
\erm
\;\defn\;
\argmin_{w \in \R^{\dimension}}\Big\{\|w\| \;:\; \vec{X}w=\vec{y}\Big\}.
\end{equation}

% (i.e., when $(\cF,\|\cdot\|_{\cF})$ is a Hilbert space)
\paragraph{Inner-product norms.}
When the norm $\|\cdot\|$ is induced by an inner product, one can derive closed-form expressions for $\erm$. Prominent examples include the $\ell_2$-MNI in linear regression and the minimum-norm interpolator with respect to a reproducing kernel Hilbert space (RKHS) norm. In the minimum-$\ell_2$ setting, these formulas enable precise analyses in the proportional asymptotic regime where $d,n \to \infty$ with $d/n \to \gamma \in (0,\infty)$; see, for example,~\cite{ghorbani-2021, hastie2022surprises, mei2022}. They also facilitate non-asymptotic risk bounds for minimum-$\ell_2$ and minimum-Hilbert-norm interpolators. A series of works~\citep{bartlett2020benign, tsigler2023benign, lecue2023geometrical, chinot2020robustness, muthukumar2020, zhou2023uniform} characterizes the risk in terms of the spectral decay of the feature covariance in linear regression, or of the associated integral operator in kernel regression.

\paragraph{Arbitrary norms.}
For norms not induced by an inner product, MNIs generally do \emph{not} admit closed-form expressions, and considerably less is known about their statistical behavior. \cite{koehler2021uniform} gave a first general analysis for Gaussian covariates using local uniform convergence and the convex Gaussian min--max theorem (CGMT)~\citep{thramp2015}, obtaining non-asymptotic bounds on the prediction error of the MNI without relying on inner-product structure. Building on this approach, \cite{donhauser2022fast, wang2022tight} studied MNI for $\ell_p$ norms with $p \in [1,2]$ in linear models with \emph{isotropic} Gaussian covariates, deriving sharp rates for the prediction error. These analyses rely crucially on CGMT. More recently, \cite{kur2026minimum,pmlr-v235-kur24a} extended \cite{donhauser2022fast} to sub-Gaussian covariates using an approach inspired by the geometry of $2$-uniformly convex norms (cf.~\cite{klartag2008volume}).

\subsection{Positions}

The notion of a \emph{position} is classical in high-dimensional convex geometry. Given a convex body $K\subset\R^d$ and a geometric or analytic task, one seeks a linear map $T\in\mathrm{GL}(d)$ such that the image $TK$ is suitably regular with respect to the relevant functional. Throughout this text, for a centrally symmetric convex body $K \subset \R^d$, $K^{\circ}$ denotes the polar body of $K$ and $\|\cdot\|_K$ denotes its Minkowski functional.

\paragraph{Milman's $M$-position.}
We say that $K$ is in $M$-position~\cite{milman-1986} if there exists a universal constant $C>0$, independent of $K$ and $d$, such that
\[
    \max\{\cN(K,B_d),\cN(B_d,K),\cN(K^{\circ},B_d),\cN(B_d,K^{\circ})\} \leq C^d,
\]
where $\cN(A,B)$ denotes the minimal number of translates of $A$ required to cover $B$, and the symbols $C, C_1, c, c_1, \ldots > 0$ denote absolute constants whose values may change from line to line. Note that $M$-position is far from being unique.
% ; for example, if $K$ is in $M$-position then so is $UK$ for every orthogonal map $U$.

\paragraph{Pisier's $\ell$-position.}
 We define the \textbf{mean norm} by
\[
    M(K) := \int_{S^{d-1}} \|x\|_K \, d\sigma(x),
\]
where $\sigma$ is the normalized Haar probability measure on the Euclidean unit sphere $S^{d-1}$. The \textbf{mean dual norm} is
\[
    M^*(K) := M(K^\circ) = \int_{S^{d-1}} \|x\|_{K^\circ} \, d\sigma(x),
\]
 Since $\|x\|_{K^\circ}=h_K(x)$, i.e., the support function of $K$, the quantity $M^*(K)$ is the mean width under the usual normalization. Pisier's $\ell$-position satisfies
\[
    M(K)M^*(K) \lesssim \log(d),
\]
where $f \lesssim g$ means that there exists a constant $C>0$ such that
$f(z) \le C\,g(z)$ for all $z \in \cZ$ ($ f\gtrsim g$, and $f \asymp g$ are defined in a similar way).
\paragraph{Minimal mean width position.}
We say that $K$ is in minimal mean width position if its mean dual norm, equivalently its mean width, is minimal among all volume-preserving linear images of $K$. More precisely, after fixing volume,
\[
    M^*(K) = \inf_{T\in SL(d)} M^*(TK),
\]
where $SL(d)$ denotes the determinant-one linear maps on $\R^d$. 
% Since $(TK)^\circ=T^{-\mathsf{T}}K^\circ$, this can equivalently be written as
% \[
%     M(K^\circ) = \inf_{T\in SL(d)} M(T^{-\mathsf{T}}K^\circ).
% \]
The determinant constraint rules out trivial rescaling and selects an affine image whose average support function is minimal.

\paragraph{Isotropic position.}
We say that $K$ is in isotropic position if the uniform measure on $K$ has zero mean and identity covariance. Recent progress shows that isotropic position is much closer to the classical $M$- and $\ell$-positions than one might first expect. The resolution of the slicing problem by \cite{klartag2025affirmative,bizeul2025slicing}, building on the work of \cite{guan2024note}, implies that isotropic convex bodies are in $M$-position up to a universal constant. 

The recent result of \cite{bizeul2025distances} further shows that isotropic position is as good as Pisier's $\ell$-position up to polylogarithmic factors: if $\psi_d$ denotes the KLS constant, then
\[
    M(K) \lesssim \frac{\psi_d\sqrt{\log d}}{\sqrt d},
\]
and using the mean-width estimate of \cite{milman2015mean} gives, for an isotropic convex body $K\subset\R^d$,
\[
    M^*(K) \lesssim \sqrt{d}\,\log^2(d),
\]
and consequently
\[
    M(K)M^*(K) \lesssim \log^{\alpha}(d)
\]
for a universal exponent $\alpha$; using the current bound $\psi_d\lesssim\sqrt{\log d}$, one may take $\alpha\leq 3$.

\paragraph{When the position is important.}
The unit balls of commonly used norms---such as the $\ell_p^d$ norms and the top-$k$ norms---satisfy the standard positions from geometric analysis, up to scaling. Studying only these norms can therefore misleadingly suggest that the choice of position is irrelevant. For general norms, however, the situation is much more subtle; see \cite{giannopoulos2000extremal} and the monographs \cite{pisier-1989,artstein2015asymptotic}.

\emph{
Throughout this work, we say that a norm $\|\cdot\|$ lies in a given position if $c_{\|\cdot\|}K$ lies in that position for some normalization constant $c_{\|\cdot\|}>0$, where $K$ denotes the unit ball of $\|\cdot\|$.
}

\subsection[The ell1-MNI]{The $\ell_1$-MNI}

The $\ell_1$ MNI, also known as basis pursuit (BP), is
\begin{equation}\label{eqn:min-norm-lone}
\erm
\;\defn\;
\argmin_{w \in \R^{\dimension}}\Big\{\|w\|_1 \;:\; \vec{X}w=\vec{y}\Big\}.
\end{equation}
It is well known that BP generalizes well in the noiseless case but is highly sensitive to noise~\cite{Can08,DE06}. For noisy data, upper bounds on the prediction error of order $\sigma^2$ have been derived for isotropic Gaussian features~\cite{koehler2021uniform,JLL20,Woj10}, sub-exponential features~\cite{Fou14}, and even heavy-tailed features~\cite{chinot2020robustness,KKR18}. In the isotropic Gaussian setting with adversarial noise, \cite{chinot2020robustness} established a lower bound of order $\sigma^2/\log(d/n)$; see also~\cite{CL21,MVSS20}. They conjectured that the $\ell_1$-MNI is inconsistent in the non-adversarial case. This conjecture was recently refuted by \cite{wang2022tight}, who used CGMT to establish the rate
\[
   \frac{\sigma^2 (1+O_{d/n}(\log(d/n)^{-1/2}))}{\log(d/n)}.
\]
when the ground truth is $O(n/\log(d/n)^{C})$-sparse, i.e., $\|\ft\|_0 = O(n/\log(d/n)^C)$ for some $C \geq 0$; and in this work we use standard Landau notation $O(\cdot)$, $\Omega(\cdot)$, $\Theta(\cdot)$.
\subsubsection{On the relation between BP and symmetric Gaussian polytopes}
Throughout this work, we assume that $d \geq n+1$, and the Gaussian symmetric polytope is defined via
\begin{equation}\label{Eq:PND}
    \PN=\conv\{\pm X_i\}_{i=1}^d\subset\R^n,
    \qquad X_1,\ldots,X_d\underset{i.i.d.}{\sim} N(0,I_n).
\end{equation}
The connection between symmetric Gaussian polytopes and BP has a long history in statistics (cf. \cite{efron1965convex}). Notable examples (which are the closest to our work)  are \cite{donoho2009counting,donoho2010counting}. In those works, they connected the number of $k$-facets of $\PN$ to the behavior of the $\ell_1$-MNI, in the high-dimensional regime, i.e., when $d/n \to \gamma \in (1,\infty)$. To our knowledge, this approach has not been used recently, nor has it been used in more recent works on benign overfitting.
% \paragraph{Gaussian Polytopes:} 
% we show that the MSE and the negative moments of the $\ell_1$ norm of the  $\ell_1$-MNI.
% In furtThe Gaussian symmetric polytope, is the absolute convex hull of $d$ ($d \geq n+1$) isotropic Gaussian in $\R^{n}$, i.e. 
% \begin{equation}\label{Eq:SGP}
%    P_{n,d} := \mathrm{conv}\{\pm X_i\}_{i=1}^{d} \ \text{ where $X_1,\ldots,X_{d} \sim \Normal{0}{I_{n}}$}
% \end{equation}
% Those polytopes are well studied in the literature of stochastic geometry \citep{schneider2008stochastic,kabluchko2019expected,boroczky2024facets}.

% We remark that a lovely position, for example different approach by  
% showed that there is a family of $M$-positions of $K$ at different scales.

% a  We will consider this positions in our work. Recently ,

% According to our knowledge. As we will see in the second part of the theorem, if the norm is in the correct position, one can ger
% &\E_{\cD} \| \erm - \ft\|_2^2 = 
% \\&\underbrace{B^2(\erm)
% + 
% \Var_{\vec X} \Big(\E_{\vec \xi}[\erm \mid \vec X]\Big)}_{\defn E_1}
% + 
% \underbrace{\E_{\vec X} \Var_{\vec \xi}(\erm \mid \vec X)}_{\defn E_2} 
% \\&= E_1 + E_2,

\subsection{Our Mean Squared Error Decomposition}
We present a geometric perspective on MNI in linear models with \emph{isotropic} Gaussian covariates. Our approach uses tools from high-dimensional geometry to obtain sharp error rates, without relying on CGMT or Gaussian comparison inequalities. Our object of study is the mean-squared error, which in the isotropic covariate setting equals
\begin{equation}\label{Eq:MSE}
    \E_{\cD}\!\big[\|\erm - \ft\|_2^2\big].
\end{equation}
We use the following decomposition, also appearing in~\cite{kur2026minimum,pmlr-v235-kur24a}. Define
\begin{equation}\label{Eq:E1}
    E_1 := \E_{\vec X}\Big\|
    \cP_{\ft}\Big(\E_{\vec{\xi}}[\erm \mid \vec{X}]\Big)-\ft
    \Big\|_2^2,
\end{equation}
which measures the squared shrinkage of the conditional mean of the MNI in the signal direction,
\begin{equation}\label{Eq:E2}
    E_2 := \E_{\vec X}\Big\|
    \cP_{\ft^\perp}\Big(\E_{\vec{\xi}}[\erm \mid \vec{X}]\Big)
    \Big\|_2^2,
\end{equation}
which measures the squared orthogonal bias, and
\begin{equation}\label{Eq:E3}
    E_3 := \E_{\vec{X}}\!\Var_{\vec{\xi}}(\erm \mid \vec{X})
    = \E\Big\|\erm - \E_{\vec{\xi}}[\erm \mid \vec{X}]\Big\|_2^2,
\end{equation}
the expected conditional variance. Thus the MSE decomposes into the parallel bias, the orthogonal bias, and the variance terms.
% ; specifically of $(\alpha \ft + \ft^\perp) \cap K$ for $\alpha > 0$. 
\subsection{Summary of contributions}
\begin{enumerate}[label=(\roman*)]
\item 
\Cref{T:Localization} characterizes the term $E_1$ in a localized way under isotropic Gaussian covariates, for any MNI whose unit ball is in isotropic position. This yields the first \emph{localization} principle for $E_1$ analogous to results for constrained least squares estimators (cf.~\cite{chatterjee2014new,bartlett2005local,bartlett2002rademacher}).
The characterization depends on the mean norm of sections of the convex body $K$ in the direction $\ft$, defined rigorously below. 
% Furthermore, if this norm additionally satisfies the $2$-uniform convexity condition,
% \
\item In \Cref{Theorem:SPofGP}, we relate the MSE of the $\ell_1$-MNI to the \textbf{variance} of the $\ell_1$-MNI in terms of \text{$\ell_1$-norm}. This nontrivial connection--of independent interest--follows from Talagrand's $L_1$--$L_2$ inequality.

\item 
\Cref{Theorem:LoNE} establishes a sharp risk bound for the $\ell_1$-MNI without invoking Gaussian comparison inequalities or the CGMT. 
The approach draws on tools from high-dimensional geometry, probability, and super-concentration, in particular on the geometry of symmetric Gaussian polytopes \cite{fleury2012poincare}.
\item Our analysis yields several estimates on the behavior of the Gaussian symmetric polytope $\PN$, in the regime of $d \geq Cn$, for some absolute constant $C \geq 0$.
\begin{enumerate}
    \item In \Cref{C:MND}, we obtain the sharp asymptotic form of the volume of the section
    \begin{equation}\label{Eq:MndIntro}
        M_{n,d}^s := \E_{\PN} \int_{\Sn} \|\vec \xi\|_{\PN} = \E_{E \sim U(\Gr)} M(\cP_{E}\ell_1^d) \quad \text{and} \quad M_{n,d} := \sqrt{n} \cdot M_{n,d}^s,
    \end{equation}
    where $\cP_{E}$ is the linear projection on the subspace $E$.
\item  \Cref{Thm:Isotropic} sharpens the result of \cite{klartag2009hyperplane} on the isotropic constant of symmetric Gaussian polytopes.
 For any convex body $K \subset \R^n$ with barycenter at the origin, the isotropic constant of \(K\) is defined by
\begin{equation}\label{Eq:isotropic}
L_K^2
:=
\frac{1}{n}
\inf_{T\in SL(n)}
\frac{1}{|K|^{1+2/n}}
\int_K \|Tx\|_2^2\,dx ,
\end{equation}
where
\[
SL(n):=\{T\in GL(n): |\det T|=1\}.
\]
We show that with high probability
\[
   L_{P_{n,d}}
   =
   \bigl(1+O(\log(d/n)^{-2})\bigr)L_{B_n},
\]
rather than only a universal upper bound that appears in \cite{klartag2009hyperplane}.
% Note that when $K$ is in \emph{istoropic} position, it holds that for any $\theta \in \Sn$ 
% \begin{equation}
%     \Pr_{X \sim U(K)} \lp |\langle \theta, X \rangle| \leq \eps\sqrt{n} \rp \geq 1 - \exp(-c\sqrt{n}\eps).
% \end{equation}
\item In \Cref{C:Fleury} we recover the Poincar\'e-type estimate of \cite{fleury2012poincare} for symmetric Gaussian polytopes via an elementary proof. Furthermore, we prove in \Cref{Thm:ThinShell} a better Poincar\'e-type estimate on the ``thin shell'' constant of $\PN$.
\end{enumerate}

\end{enumerate}

%% file: MR.tex
\section{Main results}\label{S:MR}
\paragraph{Notation.}  For sets $E,F \subset \R^k$, the Minkowski sum is
$E+F \defn \{e+f : e\in E,\ f\in F\}$; for a singleton $E=\{e\}$, we write $e+F \defn \{e\}+F$. For any $m \in \mathbb{N}$ and measurable $A \subset \R^{m}$, we write $|A|$ for its Lebesgue measure (volume) and
$|\partial A|$ for the $(m-1)$-dimensional surface measure of its boundary.
For a square matrix $M$, we may write $|M|$ for its determinant.  For every convex body $L \subset \R^{m}$ (for some $m \ge n+1$), define
\[
M_n(L)
\;\defn\;
\E_{\vec{X}}\, M(\vec{X}L),
\]
where  $\vec{X}L \subset \R^n$ is the Gaussian image of $L$ under a random matrix $\vec{X} \in \R^{n \times m}$ with i.i.d.\ entries $\Normal{0}{1}$. Throughout this work, we denote by
\[
     \Psi(t):=\Pr_{g \sim N(0,1)}(|g| \leq t).
\]   For any polytope $P$, we will use $\cF_{n-1}(P)$ to denote its collection of facets (\emph{i.e.,} its $(n-1)$-dimensional faces). We denote by $|\cF_{n-1}(P)|$ the number of facets.
% For every $w \in K$ and $t \ge 0$, define the affine \emph{sections} of $K$ by
% \begin{equation*}
% K(w,t)
% \;\defn\;
% \cP_{w^{\perp}}\big(\big(K \cap \big(\tfrac{(1-t)w}{M_n(K)} + w^{\perp}\big)\big) -  \tfrac{(1-t)w}{M_n(K)} \big).
% \end{equation*}
% The corresponding mean gauge is denoted by
% \[
% M_n(w,t)
% \;\defn\;
% M_n\!\big(K(w,t)\big)
% \]
%%%%%%%%%%%%%%%%%%%%%%%%%%%%%%%%%%%%%%%%%%%%%%%%%%%%%%
\subsection{On the shrinkage of the MNI}

Recall that \(K\subset \R^d\) is the unit ball of \((\R^d,\|\cdot\|)\). 
We first impose the following normalization and position assumptions.

\begin{assumption}\label{ass:ground-truth}
We assume that \(K\) is in isotropic position, that
\[
\|\ft\|\asymp \|\ft\|_2\asymp 1,
\]
and that
\[
M_n(K)\gg \|\ft\|\asymp 1.
\]
\end{assumption}

In words, noise is harder to interpolate than the pure signal, and the
\(\ell_2\)-norm of the signal is comparable to its norm with respect to
\(\|\cdot\|\). Hence, the unit ball must be inflated in order to interpolate
pure noise.

We shall also assume that the orthogonal component of the MNI has a polynomial
upper tail at its natural second-moment scale.

\begin{assumption}\label{A:medDev}
There exist absolute constants \(C_0,C_1>0\) such that, setting
\[
r_\star^2
:=
C_0\log^{C_1}(d)\cdot (E_2+E_3) 
= 
C_0\log^{C_1}(d)\cdot \E \|\cP_{(\ft)^{\perp}} \erm\|_{2}^{2},
\]
we have
\[
\Pr_{\vec X,\vec \xi}
\left\{
\|\cP_{H}(\erm)\|_2\le r_\star
\right\}
\ge 1-d^{-100}.
\]
where $H = (\ft)^{\perp}$ and \(E_2,E_3\) are defined above.
\end{assumption}

For example, this assumption holds for the \(\ell_p\)-norm, with an absolute
constant, when \(p\in[1,2]\). This assumption is position-dependent. Without
placing \(K\) in a suitable position, it may fail badly; for instance, for a
highly anisotropic ellipsoid, the random variable
\(\|\cP_H(\erm)\|_2\) may have heavy tails.

For \(t\in[0,2]\), let \(K_t=K(\ft,t)\) denote the translated section of \(K\)
in the direction \(\ft^\perp\) corresponding to shrinkage level \(t\), and set
\[
M_n(t):=M_n(K_t).
\]
More explicitly, one may take
\[
K_t
:=
\cP_{\ft^\perp}
\left(
K\cap \left[
\frac{(1-t)\ft}{M_n(K)}+\ft^\perp 
\right]
\right),
\]
with the section translated to the origin in \(\ft^\perp\).
% We write
% $\widetilde{\vec{X}}:=\vec X\cP_{\ft^\perp}$.
% The proof also requires uniform concentration of the gauges associated with
% these localized sections. We record this as an assumption; alternatively, it
% can be replaced by a lemma proving the same estimate.
% \begin{assumption}[Uniform section concentration]\label{A:section-concentration}
% Let
% \[
% K_t^{(r_\star)}:=K_t\cap r_\star B_2^{\ft^\perp}.
% \]
% There are absolute constants \(C_2,C_3>0\) such that, with probability at least
% \(1-d^{-100}\), uniformly for \(s,t\in[0,2]\),
% \[
% \left|
% \frac{
% \|\vec \xi+t\vec X\ft\|_{\widetilde X K_s^{(r_\star)}}
% }{
% \sqrt{1+t^2\|\ft\|_2^2}\,M_n(s)
% }
% -1
% \right|
% \le
% C_2\log^{C_3}(d)\frac{r_\star}{\sqrt n}.
% \]
% \end{assumption}
% Define
% \[
% \eta_n
% :=
% C_2\log^{C_3}(d)\frac{r_\star}{\sqrt n}.
% \]
The gain from shrinking the signal by amount \(t\) is
\[
\Delta(t)
:=
M_n(0)
-
\sqrt{1+t^2\|\ft\|_2^2}\,M_n(t).
\]
Note that under \Cref{ass:ground-truth}, we may use the quadratic approximation
\[
\Delta_2(t)
:=
M_n(0)-M_n(t)
-
\frac{t^2\|\ft\|_2^2}{2}M_n(0).
\]
where we used that for small $t$, it holds that $M_n(t) = (1 + o(1))M_n(0)$.
We define the \textbf{localization radius} by
\[
t_\star
:=
\operatorname*{arg\,max}_{t\in[0,2]}
\Delta(t).
\]
Under sufficient regularity of the sections, our proof would imply that
\(\max\{t_\star^2,\Tilde{O}(1/n)\}\) captures the error term \(E_1\); see the first remark below. 
% Meaning that under our approximation, we aim to get that 
% \[
%     \frac{d}{dt}M_{n}(t)/M_{n}(0) \approx t\|\ft\|_2^2.
% \]
However, without this regularity, we use the following localized robust comparison.
% Replacement block for the localized/offset-radius part of Theorem~\ref{T:Localization}
% This version is aligned with proof_outline_thm1_revised_final.tex.
Let
\[
H:=\ft^\perp,
\qquad
M_0:=M_n(K),
\]
and let \(r_\star\) be the deterministic radius from
Assumption~\ref{A:medDev}. 

% For \(a>0\) and \(t\in[0,2]\), define the
% localized section
% \[
% K_{a,t}^*
% :=
% \left\{
% h\in H:
% \frac{(1-t)\ft}{a}+h\in K,\
% \|h\|_2\le \frac{r_\star}{a}
% \right\}.
% \]
% The deterministic localized section is
% \[
% K_t^*:=K_{M_0,t}^*.
% \]
Let
\[
K_t
=
\left\{
h\in H:
\frac{(1-t)\ft}{M_0}+h\in K
\right\} \quad 
K_t^*=K_t\cap \frac{r_\star}{M_0}B_2^H.
\]
We define
\[
\widetilde M_n(t):=M_n(K_t^*)
\]
and
\[
\eta_n
:=
C\log^C(n)
\left(
\frac{r_\star}{\sqrt n}
+
\frac1n
\right).
\]
The term \(r_\star/\sqrt n\) comes from the localized concentration estimate,
while the \(1/n\) term comes from the radial fluctuation of
\(\vec\xi+t\vec X\ft\), after absorbing the linear fluctuation in \(t\) into
 the quadratic shrinkage term by Young's inequality.

For a sufficiently small absolute constant \(c_1>0\), define the robust
localized gain
\[
\widetilde\Delta(t)
:=
\widetilde M_n(0)
-
\left(1+c_1t^2\|\ft\|_2^2\right)\widetilde M_n(t).
\]
Under Assumption~\ref{ass:ground-truth}, the convexity comparison of nearby
sections gives
$
\widetilde M_n(t)\asymp \widetilde M_n(0)$ uniformly for $t\in[0,2].
$
We define the \textbf{offset localization radius} by
\[
\widetilde t_\star
:=
\sup\left\{
t\in[0,2]:
\widetilde\Delta(t)
\ge
-C\eta_n\widetilde M_n(0)
\right\}.
\]
% If one does not prove a priori that the shrinkage parameter is nonnegative,
% replace the last display by the signed version
% \[
% \widetilde t_\star
% :=
% \sup\left\{
% |t|:
% t\in[-2,2],
% \widetilde\Delta(t)
% \ge
% -C\eta_n\widetilde M_n(0)
% \right\}.
% \]

\begin{theorem}[Localization of the MNI shrinkage]\label{T:Localization}
There exists a sufficiently large absolute constant \(C>0\) such that if
\(d\ge Cn\), then under Assumptions~\ref{ass:ground-truth} and~\ref{A:medDev},
with probability at least \(1-d^{-c}\),
\[
\|\cP_{\ft}(\erm)-\ft\|_2
\lesssim
\widetilde t_\star.
\]
\end{theorem}

\begin{remark}
Here, we surpass the \(O(n^{-1/2})\) barrier appearing in the analysis
of~\citep[Lemma~10]{zhou2023uniform} under consistency of the MNI. In
particular, the offset term vanishes whenever \(E_2+E_3\to0\) sufficiently fast.
\end{remark}

\begin{remark}
In the well-studied setting of the \(\ell_p\)-MNI for \(p\in[1,2]\), and
\(\ft=(1,0,\ldots,0)\), the sections are homothetic copies of the
\(\ell_p\)-ball in \(\R^{d-1}\). In that case the deterministic comparison
simplifies to
\[
\Delta_2(t)
=
M_n(0)-M_n(t)
-
\frac{t^2M_n(t)}{2},
\]
since \(\|\ft\|_2=1\). Thus one may identify the localization scale from the
crossing condition
\[
M_n(0)-M_n(t)
\approx
\frac{t^2M_n(0)}{2},
\]
without carrying the additional section-deviation parameter. In the general
isotropic setting, however, we keep the offset \(r_\star\) and upper-bound the
localization scale by \(\widetilde t_\star\).
\end{remark}

\begin{remark}
To upper-bound the term \(E_2\), one needs more structure on the norm than
isotropy alone. For instance, it is enough in many cases to assume that the
norm is \(2\)-uniformly convex, or at least has cotype \(2\). We refer to the
work of the first author~\cite{kur2026minimum} for more details.
\end{remark}
\medskip
% \begin{remark}
% Under sub-Gaussian covariates this technique fails. However, if the $\|\cdot\|$ is $2$-uniformly convex with constant $\alpha \geq 0$, that is for all $x,y \in \R^d$, such that $\|x\| = \|y\| = 1$, 
% \[
%     \left\|\frac{x+y}{2}\right\|^2 \leq 1 - \alpha \|x-y\|^2
% \]
% Then, one can obtain a similar bound up to a constant that depends on $\alpha$, see Remark *** below. We believe that it can be extended to cotype $2$-norm as well.
% \end{remark}
%%%%%%%%%%%%%%%%%%%%%%%%%%%%%%%%%%%%%%%%%%%%%%%%%%%%%%%%%%%%%%%%%
\subsection[On the ell1-MNI]{On the $\ell_1$-MNI}
In this sub-section, we denote the  $\ell_1$–MNI by $\erm$. The next theorem provides a sharp bound on the deviation constant of the $\ell_1$-MNI, in the origin, which is defined as
\[
    d_{\vec{0}}^2
\;:=\;
\Var\!\left(\frac{\|\erm(\vec X,u)\|_{1}}{\E \|\erm(\vec X,u)\|_{1} }\right),
\]
for some $u \in \sqrt{n}S^{n-1}$.  It is important to note this variance is only over $\vec X$, as we know that $U\vec X \sim \vec{X}$ for every rotation \(U\); this automatically gives us that we can assume that $\vec \xi \sim U(\sqrt{n} \Sn)$, which is almost a Gaussian distribution. In words, we get extra internal random free rotation from $\vec X$. Also, recall the definition of $M_{n,d}$ above and note that
\[
M_{n,d} = \E \|\erm(\vec X,u)\|_{1}
\]
which follows from the fact that for every $\vec\xi$ it holds
\[
    \|\vec \xi\|_{P_{n,d}} = \|\erm(\vec X,\vec \xi)\|_{1}, 
\]
where we used that $P_{n,d} = \vec X B_1^d$.

% Throughout, this manuscirpt, we denote by

% The last inequality is an immediate consequence of Dvoretzky’s theorem. When $d \ll g$, one should view $M_{n,d}$ (up to a normalization constant), as the average dual norm of random sections of the $\ell_{\infty}^d$ ball

\begin{theorem}\label{Theorem:SPofGP}
Fix some $\vec \xi \in  \sqrt{n} \cdot \Sn$. Let $c \in (0,1/6)$ and some $C \geq 0$ large enough, and assume that $d  \in (n (\ln n)^{C},\exp(n^{c}))$. Then, the deviation parameter of the $\ell_1$-MNI in $\R^d$ at the origin satisfies
\[
d_{\vec{0}}^2 \lesssim\;
\frac{\E \|\erm(\vec{X},\vec{\xi})\|_{2}^{2}}{\ln^{2}(d/n)\, M_{n,d}^{2}}
\;\lesssim\;
\frac{1}{n\,\ln^{2}(d/n)},
\]
where the last equality follows as mentioned above by using $M_{n,d} \asymp n/\log(d/n)$, and by \cite{wang2022tight} who showed that $\E \|\erm(\vec{X},\vec{\xi})\|_{2}^{2} \asymp \log(d/n)$.
\end{theorem}
\begin{remark}
     Note that if we had taken the variance over  $\vec \xi$ to be isotropic Gaussian, we would get only a rate of $1/n$, due to its $\ell_2$ fluctuations in its $\ell_2$ norm.
\end{remark}
\begin{remark}
The statement extends beyond the origin to any $\ft$ that is $O(1)$–sparse.
Also, note that $$\frac{\E \|\erm\|_{2}^{2}}{\E \|\erm\|_{1}^{2}} = \frac{\E \|\erm\|_{2}^{2}}{M_{n,d}^{2}} \ge \frac{1}{n}$$ by Jensen’s inequality; the converse bound follows from Theorem~\ref{Theorem:LoNE} below. 
\end{remark}

Our final result gives a sharp risk bound for the $\ell_1$–MNI, refining that of \cite{wang2022tight}, but is obtained via a purely geometric (non–CGMT) argument.

\begin{theorem}\label{Theorem:LoNE}
Let $c \in (0,1/6)$ and assume that both $ d \in (n\ln n^{C},\exp(n^{c}))$, and $\|\ft\|_0 \lesssim n \cdot \ln(d/n)^{-C}$. Then, with probability of at least $1-\exp(-c_1n\ln(d/n)^{-2})$, it holds that 
\[
\|\erm\|_2^2
= (2 +O(\log(d/n)^{-1}))(M_{n,d}^s)^2= 	\frac{1}{\log(d/n)}
+
\frac{\log\log(d/n)}{2\log^2(d/n)} + O(\log(d/n)^{-2}),
\]
where $M_{n,d}^s$ is defined above, and its sharpest bound appears in \Cref{C:MND}.
\end{theorem}
Our proof relies crucially on \cite{fleury2012poincare}, which computed the distribution of a facet of $\PN$.

\subsection{Structural results regarding the symmetric Gaussian polytope}

Our main results imply some geometric consequences for the symmetric Gaussian polytope $\PN$. 
First, we state a corollary for the mean spherical norm, which that follows almost immediately from the proof of \Cref{Theorem:LoNE}; see \Cref{ss:pfMND} for a proof below.
We define 
\[
\alpha(n,d)
=
\sqrt{2\ln(d/n)
-\lln(d/n)
-\ln\pi
+
\frac{\lln(d/n)+\ln\pi}{2\ln(d/n)}
+
O\!\left(
\frac{(\lln(d/n))^2}{\ln^2(d/n)}
\right)}.
\]
\begin{corollary}\label{C:MND}
Assume that $d \gtrsim n\ln(n)^{C}$. Then,
\begin{equation}\label{Eq:MNDsharp}
M_{n,d}^{s} = \bigg(1 + O\Big(\frac{1}{\ln(d/n)^2}\Big)\bigg) \, \frac{1}{\alpha(n,d)}.
\end{equation}
\end{corollary}
Secondly, we prove a high-probability estimate on the isotropic constant of $P_{n,d}$; see display \eqref{Eq:isotropic} for the definition.

\begin{theorem}\label{Thm:Isotropic}
With probability of $1-\exp(-cn/\ln(d/n)^2)$ it holds that
\[
    L_{\PN} = \bigg (1+O\Big(\frac{1}{\ln(d/n)^{2}}\Big)\bigg) \, L_{B_n}.
\]
\end{theorem}
Theorem~\ref{Thm:Isotropic} sharpens the analysis of \cite{klartag2009hyperplane}, which gave the estimate 
\[
L_{\PN} \leq C
\]
for some absolute constant $C > 0$. Note that $L_{B_n} \leq L_{K}$ is valid for any convex body $K$, cf. \cite{brazitikos2014geometry}. 

\medskip 

Our methods also give an estimate for the thin-shell estimate, in the volume-weighted expectation. 

\begin{theorem}\label{Thm:ThinShell}
Denote by 
$
    m(\PN)=\frac{1}{|\PN|}\int_{\PN}\|x\|_2dx
$. Then, it holds that
\[
    \E\left[
    |\PN|\cdot
    \Var_{Z\sim U(\PN)}
    \lp
        \frac{\|Z\|_2}{m(\PN)}
    \rp
    \right]
    \lesssim
    \frac{\E|\PN|}{n\ln(d/n)^2}.
\]
\end{theorem}
We also recover the following main result from \citep[Thm 1.1]{fleury2012poincare} via a significantly simpler proof.
\begin{corollary}\label{C:Fleury}
Let $f:\R^n \to \R$ be $1$-Lipschitz. Then
\[
    \E \int_{\PN}\lp f(x) - \E_{Z \sim U(\PN)}f(Z)\rp^2dx
    \lesssim \frac{1}{n}\E \int_{\PN}\|x\|_2^2dx.
\]
\end{corollary}

% Furthermore, 
% \begin{equation}
%     d_{\vec 0} \lesssim \frac{1}{n\ln^2(d/n)}.
% \end{equation}
% To obtain a matching lower bound, one should use the fact that $***$ proving this claim requires ,and therefore we omit it from the submitted version of the manuscript.

% also recall the definition of $P_{n,d}$ in \eqref{Eq:SGP}.
% \begin{remark}
    
% \end{remark}

% \begin{remark}
% We are able to show that without using Super-concentration argument,
%     \[
%     \E \|\erm - \ft\|_2^2 \asymp \frac{1}{\ln(d/n)} 
% \]
% see ** below. The only requirement that we need is structure on a ``typical facets'' that can be obtained via ** as. We would need that for a facet

% and also that locally on a facet with have, and . Any distribution that sat
% \end{remark}

% \begin{corollary}
% \[

% \]
% \end{corollary}

% \noindent{The last equation extends the result of \cite{boroczky2024facets} to the symmetric case, also our estimates provides is accurate up to a $\exp(\Theta(n/\ln(d/n)^2))$ -- comparing to $\exp(\Theta(n/\ln(d/n)))$. Our approach is based on the results of  \cite{paouris2016gaussian,paouris2006concentration,fleury2012poincare} and a sharp estimation of $M_n$ obtained by \cite{wang2022tight}.}

% This recovers the bound in \cite{wang2022tight}, though our result is obtained with a very different argument.  
% \begin{corollary}

% \end{corollary}

%% file: discussion.tex
\subsubsection{Open Problems on Symmetric Gaussian polytopes}
% We conclude this section with several questions about the symmetric Gaussian polytope, which is we remind that it is defined via
% \[
%     \PN=\conv\{\pm X_i\}_{i=1}^d\subset\R^n,
%     \qquad X_1,\ldots,X_d\sim N(0,I_n).
% \]
% Here, we denote $o_{d/n}(1)$ for a function $f(n, d)$ which vanishes as $n, d/n \to \infty$.
Here, we propose a few open problems that our approach fails to settle. 
%
%First, concerns that the isotropic constant of $\PN$:
%\begin{openproblem}[Sharp isotropic constant]
%Determine the optimal order of
%\[
%    \frac{L_{\PN}}{L_{B_n}}-1 = O(1/\log(d/n)^2).
%\]
%In particular, is it true, in the regime $d\gtrsim n(\log n)^C$, that
%\[
%    L_{\PN}
%    =
%    \lp 1+\Theta\lp \frac{1}{\log(d/n)^2}\rp\rp L_{B_n}
%\]
%with high probability? 
%
%\end{openproblem}
%
%\Cref{Thm:Isotropic} proves the corresponding upper bound with probability at
%least $1-\exp(-cn/\log(d/n)^2)$.  What remains unclear is whether the
%$\log(d/n)^{-2}$ error is an artifact of the proof, or whether it is the true
%second-order scale dictated by the fluctuations of the extremal facets. Now, 
The first open problem is to remove the volume bias and obtain a genuinely typical estimate of the thin shell constant for the random polytope $\PN$, that is
\[
    \sigma_{\PN}^2
    :=
    \Var_{Z\sim U(\PN)}
    \lp
        \frac{\|Z\|_2}{m(\PN)}
    \rp,
\]
where $
    m(\PN)=\frac{1}{|\PN|}\int_{\PN}\|x\|_2dx$. Recall that  \Cref{Thm:ThinShell} proves a \textbf{volume-weighted} version of the ``thin shell'' estimate in \textbf{expectation}, i.e 
\[
    \E_{\PN}\left[
    |\PN|\cdot
        \sigma_{\PN}^2
    \right]
    \lesssim
    \frac{\E|\PN|}{n\log(d/n)^2},
\]
where the expectation is on the measure that our process induces on the polytopes $\PN$. Meaning that distribution that $\PN$ 
inducing on the convex sets in $\R^d$.
It is well known that (cf. \cite{kabluchko2019expected} and reference within) that
\[
    \Var |\PN| \gg (\E|\PN|)^2.
\]
As our proof deeply relies on Fluery's distribution, which induces this weighted distribution, removing this weighting would require a new approach. However, we ask the following:
\begin{openproblem}[Thin shell for symmetric Gaussian polytopes]
Does one have at least a constant probability that
\[
    \sigma_{\PN}
    \lesssim
    \frac{1}{\sqrt n\,\log(d/n)}?
\]
\end{openproblem}
Finally, in the same context, can we obtain that the $\PN$ satisfies the KLS property? \cite{fleury2012poincare} proved the same result with weighting.
\begin{openproblem}[KLS for symmetric Gaussian polytopes]
Let $C_{\mathrm P}(K)$ denote the Poincaré constant of the uniform measure on a
convex body $K$. Namely, for $Z \sim U(K)$, define 
\[
    C_{\mathrm P}(K)
    :=
    \sup
        \Big\{\, \Var\big(f(Z)\big) \mid~\text{$1$-Lipschitz}~f \colon \R^n \to \R~\,\Big\}.
\]
Is it true that, with high probability,
\[
    C_{\mathrm P}(\PN)
    \lesssim
    \frac{\log(d/n)}{n}.
\]

\end{openproblem}
\subsection{Theorem~\ref{T:Localization} in \texorpdfstring{\(\ell_p\)}{ell-p}-linear regression}
\label{S:Example}

For any \(\dimension\ge 1\) and \(x\in\R^\dimension\), we write
\[
\|x\|_p
:=
\left(\sum_{i=1}^{\dimension}|x_i|^p\right)^{1/p},
\qquad
p\in[1,\infty) 
\text{ and }
\|x\|_\infty
:=
\max_{1\le i\le \dimension}|x_i|.
\]
For the \(\ell_p\)-norm on \(\R^\dimension\), let \(B_p^\dimension\) denote
its unit ball. We now apply \Cref{T:Localization} to the norm \(\|\cdot\|_p\), for
\(p\in[1,2]\), and to the signal
\[
\ft=(1,0,\ldots,0).
\]
This choice satisfies Assumptions~\ref{ass:ground-truth}--\ref{A:medDev} in
the relevant consistency regime; see, for example, \citep[Cor.~2]{kur2026minimum}.
It is known \citep{gordon2007gaussian,Paorlp} that
\[
   M
   :=
   M_n(B_p^d)
   \asymp
   \frac{\sqrt n}{d^{1-1/p}}
   \sqrt{
   \max\left\{
   (p-1),\frac{1}{\log(d/n)}
   \right\}
   }.
\]
Thus the assumption \(M_n(B_p^d)\gg1\) amounts to
\[
    \frac{\sqrt n}{d^{1-1/p}}
    \sqrt{
    \max\left\{
    p-1,\frac{1}{\log(d/n)}
    \right\}
    }
    \gg 1.
\]
Equivalently, for \(p>1\), writing \(q=p/(p-1)\), this requires
\[
    d
    \ll
    n^{q/2}
    \left(
    \max\left\{
    p-1,\frac{1}{\log(d/n)}
    \right\}
    \right)^{q/2}.
\]
For \(p=1\), this condition becomes
$
    \log(d/n)\ll n.
$
In particular, a rough sufficient way to summarize the admissible range is
\[
    d
    \lesssim
    \widetilde O(n^{q/2})
\]
for \(p>1\), together with the subexponential condition in the endpoint
\(p=1\).

Because \(\ft=e_1\), the sections of \(B_p^d\) perpendicular to \(\ft\) are
homothetic copies of \(B_p^{d-1}\). More precisely, if the no-shrink scale is
normalized by \(M\), then the section at shrinkage level \(t\in[0,1]\) has
radius
$
    \left(
    1-\frac{(1-t)^p}{M^p}
    \right)^{1/p}
$.
Hence
\[
    M_n(t)
    \asymp
    \frac{M_n(B_p^{d-1})}
    {
    \left(
    1-\frac{(1-t)^p}{M^p}
    \right)^{1/p}
    }.
\]
Since \(M_n(B_p^{d-1})\asymp M\) and \(M\gg1\), a Taylor expansion gives uniformly for small \(t\) that
\[
\begin{aligned}
    M_n(0)-M_n(t)
    &\asymp
    M
    \left[
    \left(1-\frac{1}{M^p}\right)^{-1/p}
    -
    \left(1-\frac{(1-t)^p}{M^p}\right)^{-1/p}
    \right]  \\
    &\asymp
    M^{1-p}\big(1-(1-t)^p\big)
    \asymp
    t\,M^{1-p},
\end{aligned}
\]
On the other hand, 
$
    \frac{t^2}{2}M_n(t)
    \asymp
    t^2 M.
$
Therefore, the localization scale is obtained by balancing
\[
    M_n(0)-M_n(t)
    \asymp
    \frac{t^2}{2}M_n(t),
\]
and therefore we conclude that 
\[
    t_\star
    \asymp
    M^{-p},
    \qquad
    t_\star^2
    \asymp
    M^{-2p}.
\]
Invoking \Cref{T:Localization}, and ignoring the offset term at the scale where
\Cref{A:medDev} gives \(r_\star/\sqrt n=o(t_\star^2)\), we obtain
\[
    E_1
    \lesssim
    t_\star^2
    \asymp
    M^{-2p}.
\]
Using the above estimate for \(M=M_n(B_p^d)\), this yields
\[
E_1
\asymp
\frac{d^{2p-2}}{n^p}
\cdot
\min\left\{
(p-1)^{-1},\log(d/n)
\right\}^p.
\]
This recovers the \(E_1\)-term in the rates obtained in
\cite{donhauser2022fast,wang2022tight}.
\section{On the Proof of \Cref{Theorem:LoNE}}\label{S:SketchPf2}
\input{SkechofMainThm}
\section{Discussion}\label{Sec:D}

\subsection{ Theorems \ref{Theorem:SPofGP} and \ref{Theorem:LoNE} for non-Gaussian covariates} 
Note that \Cref{Theorem:LoNE} deeply relies on the result of \cite{fleury2012poincare}, who computed the exact distribution of a facet of a symmetric Gaussian polytope; an exact formulation of this statement appears in \Cref{ss:Flu} below. One may ask if we can extend our results to sub-Gaussian matrices. The key challenge is that Dvoretsky's theorem, in its sharpest form, only holds for rotationally invariant distributions. For readers' convenience, we state its sharpest version (cf.
\cite{paouris2016gaussian,klartag2007small}) as follows:
\begin{theorem*}
Assume that $d \geq n+1$. Let $P = \frac{1}{\sqrt{d}}\vec X$ and consider a convex body $K \subset \R^d$. Then, the following holds with probability of $1-\exp(-c_1n)$:
\[
M^*(K)\big(1-C_1 \sqrt{\frac{n}{d}} \underbrace{\sqrt{n \cdot \Var\lp\frac{\|u\|_{K^{\circ}}}{M^*(K)}\rp}}_{(*)}\big) \leq r(PK) \leq 
R(PK) \leq
\Big(1+ C_1\, \sqrt{\frac{n}{d}} \underbrace{\frac{R(K)}{M^*(K)}}_{(**)}\Big)\, M^*(K),
\]
where $r(K)$ and $R(K)$ are the inner and outer radius of a set $K$.
\end{theorem*}
The lower and upper inclusion have different behaviors. Indeed, by the Poincare inequality for the uniform measure on the sphere, it holds that  $(*) \lesssim (**)$. When the inequality is strict, $\|u\|_{K^\circ}$ is \emph{superconcentrated}, where $u \sim U(\mathbb{S}^{d-1})$. We use this theorem to show that
\[
	\lp\frac{\E \|\vec \xi\|_{\PN}^{-n}}{(\E \|\vec \xi\|_{\PN})^{-n}}\rp^{1/n} \leq 1+o_{d/n}(1)
\]
needed to apply our strategy and follow the properties of Fleury's distribution of the facets of $\PN$. Yet for sub-Gaussian matrices, one can only show that
\[
\lp\frac{\E \|\vec \xi\|_{\PN}^{-n}}{(\E \|\vec \xi\|_{\PN})^{-n}}\rp^{1/n} \leq C
\]
see, for example, the work of \cite{guedon22geometry} (and references within) which can only lead to $\E \|\erm\|_{2}^2 \lesssim 1$.

\subsection{Small Ball Probabilities}
\begin{definition}
	The canonical simplex and its centered version are defined via
	\[  
	\triangle:= \{x \in \R^n: \sum_{i=1}^{n}x_i = 1 \} \text{ and } \triangle_{c}:= \triangle - \frac{\vec{1}_n}{n}.
	\]
\end{definition}
In this work, we would prove (see Lemma \ref{Lem:SBSimplex} below) that the following holds for all $\eps \in (0,1):$
\[
\Pr_{Z \sim U(\triangle_c)}(\|Z\|_{2} \leq (1-\eps)\E\|Z\|_{2}) \leq \exp(-cn\eps^2).
\]
\textbf{However}, as one can easily show  (cf. \cite{paouris2006concentration}) a \textbf{sharp} bound of  
\[
\Pr_{Z \sim U(\triangle_c)}(\|Z\|_2 \geq (1+\eps)\E \|Z\|_2) \leq \exp(-c\sqrt{n}\eps).
\]
Namely, the tails exhibit different behaviors; one is sub-Gaussian, and the other is sub-exponential.  The proof of this lemma follows from standard calculations of MGFs, and noting that $Z \sim \mathrm{Unif}(\triangle)$, can also be represented as
\[
Z \sim  \frac{(Z_1,\ldots,Z_n)}{\sum_{i=1}^{n}Z_i}
\]
where $Z_1,\ldots,Z_n \sim \mathrm{Exp}(1)$. Therefore, we would need a different approach for upper bounding the MSE. Finally, we state a useful bound that is known only to specialists and appears in \citep[Theorem A.]{paouris2}; see also \citep[Corollary 5.1]{alonso2015gaussian}, and \cite{paouris2004estimates}.\footnote{In \citep[Corollary 5.1]{alonso2015gaussian} they stated their theorem with an additional assumption of symmetry, but it is not needed for the set $K$ that we use, whose barycenter is zero.}
%  \begin{lemma}
	% \label{prop:tail-of-Tnd}
	%  Let $0 \leq \eps \leq c/\log(d/n)$ and assume that $d \geq 4n \geq \sqrt{\log(d/n)}$. Then,
	% \[
	% \Pr\Big\{|T_{n, d} - \mathrm{Med} T_{n, d}| \geq \eps \cdot t^\star_{n,d} \Big\} 
	% \leq 
	% \exp\Big\{-cn\log(d/n)^2\eps^2\Big\}.
	% \]
	% \end{lemma}
% This lemma follows from Fluery's distribution from a direct computation on the distribution of Lemma \ref{Lem:Fleury}.
\begin{theorem*}
	Let $K\subset\mathbb R^n$ be an isotropic convex body and assume
	\[
	R(K) \leq A\sqrt n.
	\]
	Then for every $1\le t_0\le A\sqrt n$, there is a set
	$\Theta(t_0)\subset S^{n-1}$ with
	\[
	\sigma(\Theta(t_0))
	\ge
	1-\exp(-c_A t_0^2)
	\]
	such that, for every $\theta\in\Theta(t_0)$ and every
	$t_0\le t\le A\sqrt n$,
	\[
	\Pr_{X\sim U(K)}
	\left\{
	|\langle X,\theta\rangle|\ge t
	\right\}
	\le
	\exp(-c_A t^2).
	\]
\end{theorem*}
\begin{remark}
One should note that for every $t$, $\Theta(t)$ is a different set. For example, the best-known bound for high-probability $O(1)$-sub-Gaussian marginals of an isotropic convex body $K$ is $1-\exp(-c\sqrt{n})$.
\end{remark}

%% file: SkechofMainThm.tex
\subsection{Fleury's work on Gaussian symmetric polytopes}\label{ss:Flu}
\begin{lemma}[Thm. 3 in \cite{fleury2012poincare}]\label{Lem:Fleury}
Suppose that $\conv \{\epsilon_i  \vec X_i\}_{i \in S}$ forms a facet for some signs $\epsilon_i \in \{\pm 1\}$ and a subset $S \subset [d]$ with $|S| = n$. The matrix $\vec X_{S}$, formed with the columns $\epsilon_i  \vec X_i, i \in S$, conditional on it being a facet, is distributed as the product $ \vec U \vec A$,
where $\vec U$ and $\vec A$ are independent, 
\[
\vec U \sim \mathsf{Unif}(O(n)), 
~~ 
\vec A = 
\begin{bmatrix}
\vec Y \\[2pt]
T_{n,d} \cdot \mathbf 1_n^\T
\end{bmatrix},~~
 \vec Y \in \mathbb{R}^{(n-1)\times n},
\]    
where  $\vec Y$ has columns $\{\vec Y_i\}_{i=1}^n$ distributed with density proportional to $$|\vec {yy^{T}}|\cdot \exp\Big(-\sum_{i=1}^{n}\|\vec{y}_i\|_2^2/2\Big),$$
 and $T_{n,d}$ denotes $\|\vec{n}_{\cF}\|_{2}$; it is independent of $\vec Y$, and has the density $$p(t) \propto \Psi(t)^{d-n}\exp(-nt^2/2).$$
\end{lemma}
Note that the distribution of $\vec {Y} $ is close to being Gaussian, as the determinant has very light tails; this result is due to \cite{goodman1963distribution}, as we formulate below. The distribution of $T_{n,d}$ is close to the average of the top $n$-entries of an isotropic Gaussian vector in $\R^d$.
\begin{lemma}[Lemma 6 in  \cite{fleury2012poincare}]
In the notation of Lemma \ref{Lem:Fleury}, it holds
\[
\frac{1}{n}\sum_{i=1}^{n} \vec{Y}_i \sim \Normal{0}{\frac{I_{n-1}}{n}} \perp \vec{\tilde{Y}},
\]
where $\vec{\tilde{Y}}:=(\vec Y_2-\vec Y_1,\ldots,\vec Y_{n} - \vec Y_1)$.
\end{lemma}
In particular, it implies that $\vec c_{\cF} - \vec n_{\cF}$ is independent of $\vec{n}_{\cF}$ and of $\vec{\tilde{Y}}$.
Throughout this work, we refer to the distribution of a facet  $\cF$ as \textbf{``Fleury's distribution''}.
\begin{lemma}[Thm. 4 in \cite{fleury2012poincare}]:\label{Lem:numberofFacets}
The expected number of $(n-1)$-dimensional facets of  $\PN$ satisfies
\begin{equation}
\begin{aligned}
  \E \cF_{n-1}(P_{n,d}) = \exp\lp2^{-1}n \cdot (\lln(d/n) + \Theta(1))\rp.
\end{aligned}
\end{equation}
\end{lemma}
Note that this bound should be compared to the loose upper bound on the number of facets, which is 
\[
    2^{n}\binom{d}{n} = \exp(n \cdot (\log(d/n) + \Theta(1))) 
\]
which is a significantly better exponent. First, recall the \textbf{conic formula}, which implies that for any polytope $P \subset \R^n$, it holds
\[
   |P| = n^{-1}\sum_{\text{$\cF$ is a facet}}|\cF| \cdot \|\vec{n}_{\cF}\|_{2} 
\]
Now, let
\[
    \cE:=\{\text{$\mathrm{conv}\{X_1,\ldots,X_n\}$ is a facet of $\PN$}\}
\]
and 
\[
\cE_{\cS,\eps}:=\{\text{$\mathrm{conv}\{\eps_1 X_{i_1},\ldots,\eps_{n}X_{i_n}\}$ is a facet of $\PN$}\}
\]
and note that the following holds  for $V = \E |P_{n,d}|$ by linearity of expectation and the conic formula:
\begin{align*}
   V &=  n^{-1} \cdot \E \left[\sum_{\vec \eps \in \{-1,1\}^{n}, |S| =n, S \subset [d]}1_{\cE_{\cS,\eps}} \cdot |\cF| \cdot \|\vec{n}_{\cF}\|_2\right] 
     \\&= n^{-1}2^{n}\binom{d}{n}\Pr(\cE)\E\left[|\cF| \cdot \|\vec{n}_{\cF}\|_2|\cE\right]
     \\& = n^{-1} \E |\cF_{n-1}(P_{n,d})| \cdot \E |\vec Y\vec{Y}^{\top}| \cdot |\triangle| \cdot \E T_{n,d}   
\end{align*}
where we used that $T_{n,d}$ is independent of the facet $\cF$. Following the same rationale, and using that $\vec{c}_{\cF}$ is independent of the normal $\vec{n}_{\cF}$ and of $\widetilde{\vec Y}$, we obtain the following corollary:
\begin{corollary}\label{Lem:eventFL}
Let $\cE:= A_1 \cap A_2 \cap A_3$ where $A_1$ is an event on $T_{n,d}$, i.e. the height of the normal; and $A_2$ is an event on the span of $\vec{\Tilde{Y}}$, and $A_3$ is an event on the $c_{\cF}$. Then it equals
\[
    \E |V_{\cE}| = \Pr_{T_{n,d}}(A_1)\Pr_{\vec{\tilde{Y}}}(A_2)\Pr_{n^{-1}\sum_{i=1}^n\vec{Y}_i}(A_3) \cdot \E |P_{n,d}|.
\]
where  $|V_{\cE}|$ is the volume of the cones whose facets satisfy $\cE$.
\end{corollary}
\subsection{Reduction to Fleury's distribution}
\textbf{Throughout this work,} we denote by
\[
    L:= L(n,d) = \log(d/n).
\]
Note that the distributions of $T_{n,d}$, and of $\vec {Y} $, and of $\vec c_{\cF} - \vec{n}_{\cF}$ are sub-Gaussian; and therefore have very light tails.  However, the volume of the polytope $P_{n,d}$ has a high variance as observed in \cite{paouris2019gaussian} and references within. This is mainly due to the fact that $\ell_1^d$ has a small Dvoretzky's dimension. Therefore, due to the decomposition of the volume above, the large variance emerges from the number of facets that suffer. In order to overcome this problem, we show that
\[
\exp(-C_1 \cdot nL^{-1/2})  \leq  \frac{|\PN|}{\E|\PN|}  \leq \exp(C_1 \cdot nL^{-1/2}) \text{ and } \E |\PN| =  \exp(O(nL^{-1/2})) \cdot M_{n,d}^{-n}|B_n|.
\]
with probability of at least $1-2\exp(-nL^{-1/2})$ via Dvoretzky's theorem and the KLS property of the canonical simplex, see below for the definitions of these objects.

Then, we condition on events on $T_{n,d}$ and $\cF$ and $\vec{c}_{\cF}$, with respect to Fleury's distribution, that hold with probability of at least $1-\exp(-C_1 \cdot nL^{-1/2})$. Therefore, by \Cref{Lem:eventFL}, and the last equation,  the total volume of the facets that do not satisfy is at most $\exp(-nL^{-1/2})|\PN|$, i.e., most of the facets of $\PN$, in terms of volume, satisfy the event.
% \begin{proof}
% Let $V = \E |P_{n,d}|$ 
% where we used that the normal is indepdent of the facet and distance between the normal the center of the facet is indepent from its volume.§
% Now, note that if we combine $T_{n,d}$ the tails of $**$. We can c
% \end{proof}
%  Now, we obtain the following corollary:
\subsection{Canonical Simplex satisfies KLS}
An isotropic convex body that satisfies the \textbf{KLS} with an absolute constant, if  for any $1$-Lipschitz function $F:\R^n \to \R$ and $\eps \geq 0$:
\begin{equation}\label{Eq:KLS}
 \Pr_{X \sim U(K)}(|F(X) - \E F(X)| \geq \eps\sqrt{n})\leq \exp(-c\sqrt{n}\eps).
\end{equation}
 We refer to the last equation as the (KLS) condition.  It is well known that $n \cdot \triangle_c$ is almost isotropic and satisfies the last equation, see \cite{barthe2009remarks}. In particular, KLS implies that both the thin-shell constant and the isotropic constant are bounded by a universal constant; see also  \cite{paouris2006concentration,gromov1983topological} for further details. 
% The last two inequalities are valid for every $\eps \geq 0$ and $c \geq 0$ is an absolute constant.
\subsection{Proof Outline}
First,  recall that:
\[
    \|\vec \xi\|_{n}:= \|\vec \xi\|_{P_{n,d}} = \|\erm(\vec X,\vec \xi)\|_{1}, 
\]
where we used that $P_{n,d} = \vec X B_1^d$. Furthermore, the $\ell_1$-MNI has a special interpretation:
\[
    \erm(\vec X,\vec \xi) =  \|\vec \xi\|_{\PN} \cdot w(\vec X,\vec \xi) 
\]
where $w(\vec X,\vec \xi) \in \R^d$ such that $\|w(\vec X,\vec \xi)\|_0 = n$, and $\|w(\vec X,\vec \xi)\|_1 = 1$. Namely, it is the weighting of the convex hull of the $n$-dimensional facet of the $B_1^d$ (that in particular is a simplex) that gives $\vec \xi$ and whose vertices are $\cF := \conv\{\vec{\eps}_i\vec X_i\}_{i \in S}$. 

Intuitively, we believe that a typical solution of $\erm$ behaves as a random element in $\triangle$. To understand how a random element behaves, recall the definition of $\triangle_{c}$ above, and that $n \cdot \triangle_{c} $ satisfies the KLS property and, in particular, most of its volume lies in its thin shell. Therefore,
% Meaning that most of the mass is equals to $1/\sqrt{n}$, i.e. 
$$\E_{Z \sim \mathrm{Unif}(\triangle_c)} \|Z\|_{2} = (1\pm O(1/\sqrt{n})) \cdot n^{-1/2}.$$ 
 
 Hence,  by Pythagoras's law, we would expect that most of the volume of $\PN$ has an $\ell_2$ length that satisfies
% , and the volume is $\sqrt{2}$ far from the normal---as most of the volume lies nearby the sections, and in particular of  $\theta \in \Sn$. As, we know the distribution of the normal 
\[
    \sqrt{\|\vec \eps/n\|_{2}^2 + \lp\frac{1+o(1)}{\sqrt{n}}\rp^2} \approx \sqrt{\frac{2}{n}}.
\]
To further support this intuition, we use Fleury's distribution, which implies that  a ``typical'' facet of $\PN$ is distributed as $
    \cF = \tilde{G}\triangle,$
where $\tilde{G}$ is almost distributed as a $n \times (n-1)$ Gaussian matrix, i.e., with i.i.d. $N(0,1)$ entries. Hence, by ideas that emerge from the restricted isometry property, we expect that $\sqrt{n} \cdot \cF$ is almost isotropic and that the volume of the thin shell of  $\sqrt{n} \cdot \cF$ maps to the volume of the thin shell of $\triangle_{c}$. Now, to estimate the $\ell_2$ norm of the $\erm$, recall that $\erm$ inflates $\PN$ by $(1+o(1)) \cdot M_{n,d} \gg 1$ with high probability. Via a simple Dvoretsky's and geometric arguments, we would support this claim, and show that
\[
    M_{n,d} =  \E  \|\vec \xi\|_{n} \approx \sqrt{n/(2\log(d/n))} 
\]
and therefore, we  expect that 
\[
    \|\erm\|_{2}^2 \approx (M_{n,d} \cdot \sqrt{2/n})^2  \approx \log(d/n)^{-1}.
\]

Roughly speaking, in all the remaining steps, we show that $\vec X \erm$ lies in a typical thin-shell of a {\bf``Fleury's''} facet, and that $\erm$ lies in the thin shell of the $n$-dimensional facets of the $\ell_1^d$ which are canonical simplices. The main challenge is that the $\erm$ induces a different distribution on the polytope facets from Fleury's distribution.

%% file: Acknow.tex
\paragraph{Acknowledgments}
GK conducted part of this work during his visit to the IDEAL Institute, hosted by Lev Reyzin and supported by NSF ECCS-2217023. RP gratefully acknowledges support from the NSF under grant DMS-2503579. Research by RP and GK was partially supported by an NSF award DMS-$405441$ while visiting Prof. Grigoris Paouris at Texas A\&M University. We deeply thank him for insightful discussions, help, and guidance on many references used in this work. We also thank Pierre Bizeul for useful discussions on small-ball probabilities and for allowing us to adapt one of his arguments for this paper.

%% file: proof_outline_thm1.tex
\section{Proofs}
\subsection{Proof  of Theorem \ref{T:Localization}}\label{S:P1}
Note that for a convex set \(L\subset \R^d\) containing the origin and
\(\vec z\in\R^n\), 
\[
    \|\vec z\|_{\vec X L}
    =
    \|\erm^{L}(\vec X,\vec z)\|_{L}.
\]
Above, \(\erm^{L}\) is the minimal gauge solution with respect to the body \(L\).
The following lemma is based on the recent works of
\cite{milman2015mean,klartag2025affirmative,bizeul2025slicing,bizeul2025distances},
which concern isotropic convex sets:
\begin{lemma}\label{Lem:BMstar}[\citep[Lemma 1]{kur2026minimum}]
Let \(P = \frac{1}{\sqrt{d}}\vec X\), and assume that \(d \geq C_3n\) for large
enough \(C_3 \geq 0\). Then, with probability at least \(1-\exp(-cn)\), it holds
that
\[
    1 \lesssim r(PK) \leq M^*(PK) \lesssim \log(d)^3.
\]
\end{lemma}

Next, using Lipschitz concentration and the standard Lipschitz extension
 theorem, we obtain the following localized concentration statement:
\begin{lemma}\citep[Cor. 2]{kur2026minimum}\label{Lem:Concertaion}
Consider a convex set \(L \subset \R^{d-1}\) that contains the origin such that
\[
R_{bM^*}
=
M^*(PL)/r(PL)
=
\Tilde{\Theta}(1),
\]
and suppose that, with probability at least \(1-d^{-100}\) over \(\vec X\),
\[
    \|\erm^L(\vec X,\vec \xi)\|_2
    \le \widetilde O(r_\star),
\]
where \(\vec \xi \in \sqrt{n}\,S^{n-1}\). Then, with probability at least
\(1-d^{-98}\), for any fixed \(u \in \sqrt n\,S^{n-1}\), it holds that
\[
\left|
\frac{\|u\|_{\vec XL}}{\E \|u\|_{\vec XL}} - 1
\right|
=
\widetilde{O}\lp\frac{r_\star}{\sqrt{n}}\rp.
\]
\end{lemma}

\begin{proof}[Proof of Theorem~\ref{T:Localization}]
We prove the result under the normalization
\[
    \|\ft\|_2=\|\ft\|=1.
\]
The general case follows by changing absolute constants, using
Assumption~\ref{ass:ground-truth}.

Set
\[
    H:=\ft^\perp,
    \qquad
    \widetilde{\vec X}:=\vec X\cP_H,
    \qquad
    g_{\ft}:=\vec X\ft .
\]
By Gaussianity, \(g_{\ft}\) is independent of \(\widetilde{\vec X}\), and
\[
    g_{\ft}\sim N(0,I_n).
\]
We also assume, as in the statement, that
\[
    \|\vec \xi\|_2=\sqrt n.
\]

Let
$
    M_0:=M_n(K)
$
For \(a>0\) and \(t\in[0,2]\), define the localized section
\[
K_{a,t}^*
:=
\left\{
h\in H:
\frac{(1-t)\ft}{a}+h\in K,\
\|h\|_2\le \frac{r_\star}{a}
\right\}.
\]
Thus the deterministic localized section appearing in the statement is
\[
    K_t^*:=K_{M_0,t}^*,
\]
and we write
\[
    \widetilde M_n(t):=M_n(K_t^*).
\]
For a general level \(a\), write
\[
    M_{n,a}^*(t):=M_n(K_{a,t}^*).
\]

We use the error level
\[
    \eta_n
    :=
    C\log^C(d)
    \left(
        \frac{r_\star}{\sqrt n}
        +
        \frac1n
    \right).
\]
The first term comes from Lemma~\ref{Lem:Concertaion}; the second term comes
from the radial fluctuation of \(\vec \xi+t g_{\ft}\).

With probability at least \(1-d^{-100}\), uniformly for \(t\in[0,2]\),
\begin{equation}\label{Eq:radial-lower}
    \frac{\|\vec \xi+t g_{\ft}\|_2}{\sqrt n}
    \ge
    1+c_0t^2-C\log^C(d)\frac1n,
\end{equation}
where \(c_0>0\) is an absolute constant. Indeed,
\[
\frac{\|\vec \xi+t g_{\ft}\|_2^2}{n}
=
1+t^2
+
\frac{2t\langle \vec \xi,g_{\ft}\rangle}{n}
+
t^2\left(\frac{\|g_{\ft}\|_2^2}{n}-1\right).
\]
The Gaussian bounds
\[
\left|\frac{\langle \vec \xi,g_{\ft}\rangle}{n}\right|
\lesssim
\sqrt{\frac{\log d}{n}},
\qquad
\left|
\frac{\|g_{\ft}\|_2^2}{n}-1
\right|
\lesssim
\sqrt{\frac{\log d}{n}}
\]
hold with probability at least \(1-d^{-100}\). By Young's inequality,
\[
    t\sqrt{\frac{\log d}{n}}
    \le
    \varepsilon t^2+C_\varepsilon\frac{\log d}{n},
\]
and \eqref{Eq:radial-lower} follows after changing constants.

We next prove a uniform concentration estimate for the localized sections. Fix
\(a,\mu,t\). Conditionally on \(g_{\ft}\), the vector
\(\vec \xi+t g_{\ft}\) is fixed and independent of \(\widetilde{\vec X}\).
Thus, by rotational invariance and homogeneity of the gauge,
\begin{equation}\label{Eq:GaussianMagic}
\begin{aligned}
\E_{\widetilde{\vec X}}
\|\vec \xi+t g_{\ft}\|_{\widetilde{\vec X}K_{a,\mu}^*}
&=
\frac{\|\vec \xi+t g_{\ft}\|_2}{\sqrt n}
\cdot
\E_{\widetilde{\vec X}}
\|\vec \xi\|_{\widetilde{\vec X}K_{a,\mu}^*} \\
&\ge
\left(1+c_0t^2-C\log^C(d)\frac1n\right)
M_{n,a}^*(\mu).
\end{aligned}
\end{equation}
Applying Lemma~\ref{Lem:Concertaion} to the fixed localized section
\(K_{a,\mu}^*\), and using the definition of \(r_\star\), gives
\[
    \left|
    \frac{
    \|\vec \xi+t g_{\ft}\|_{\widetilde{\vec X}K_{a,\mu}^*}
    }{
    \E_{\widetilde{\vec X}}
    \|\vec \xi+t g_{\ft}\|_{\widetilde{\vec X}K_{a,\mu}^*}
    }
    -1
    \right|
    \le
    C\log^C(d)\frac{r_\star}{\sqrt n}
\]
for fixed \(a,\mu,t\). Combining this with \eqref{Eq:GaussianMagic}, and
enlarging \(\eta_n\), yields
\[
    \|\vec \xi+t g_{\ft}\|_{\widetilde{\vec X}K_{a,\mu}^*}
    \ge
    \left(1+c_0t^2-C\eta_n\right)M_{n,a}^*(\mu).
\]

We discretize
\[
    a\in[c_aM_0,C_aM_0],
    \qquad
    t,\mu\in[0,2],
\]
where \(0<c_a<C_a<\infty\) are fixed absolute constants chosen large enough for
the quadratic approximation event below. Taking a mesh of size \(d^{-100}\),
applying the previous estimate on the grid, and taking a union bound gives the
estimate on the grid. The passage to all parameters follows from the deterministic
convexity inclusions for nearby sections: if
\[
    \left|\frac{a-a'}{M_0}\right|+|t-t'|+|\mu-\mu'|\le \delta,
\]
then
\[
    (1-C\delta)K_{a,\mu}^*
    \subset
    K_{a',\mu'}^*
    \subset
    (1+C\delta)K_{a,\mu}^* .
\]
Taking the mesh sufficiently fine, we obtain, with probability at least
\(1-d^{-98}\), uniformly for \(a\in[c_aM_0,C_aM_0]\) and \(t,\mu\in[0,2]\),
\begin{equation}\label{Eq:uni}
    \|\vec \xi+t g_{\ft}\|_{\widetilde{\vec X}K_{a,\mu}^*}
    \ge
    \left(1+c_0t^2-C\eta_n\right)M_{n,a}^*(\mu).
\end{equation}
Similarly, applying the same argument at \(t=0\) gives the upper bound
\begin{equation}\label{Eq:baseline}
    \|\vec \xi\|_{\widetilde{\vec X}K_{a,\mu}^*}
    \le
    \left(1+C\eta_n\right)M_{n,a}^*(\mu)
\end{equation}
uniformly over the same range of \(a,\mu\).

Let
\[
    m:=m_{\vec \xi+g_{\ft}}
    :=
    \|\erm(\vec X,\vec X\ft+\vec \xi)\|
\]
be the norm of the MNI, and write
\[
    \erm
    =
    (1-\widehat t)\ft+\widehat h,
    \qquad
    \widehat h\in H.
\]
Since \(\erm\) interpolates the data,
\[
    \widetilde{\vec X}\widehat h
    =
    \vec \xi+\widehat t g_{\ft}.
\]
On the event from Assumption~\ref{A:medDev},
\[
    \|\widehat h\|_2\le r_\star.
\]
Moreover, since \(\|\erm\|=m\),
\[
    \frac{(1-\widehat t)\ft+\widehat h}{m}\in K,
\]
and hence
\[
    \frac{\widehat h}{m}\in K_{m,\widehat t}^*.
\]

We now invoke the quadratic approximation event included in
Assumption~\ref{A:medDev}. On this event,
\begin{equation}\label{Eq:m-concentration-pre}
    \left|\frac{m}{M_0}-1\right|
    \le
    C\left(\widehat t^2+\eta_n\right).
\end{equation}
In particular, after choosing \(c_a,C_a\) appropriately,
\[
m\in[c_aM_0,C_aM_0],
\]
so the uniform estimate \eqref{Eq:uni} may be applied with \(a=m\).

Since \(\widehat h/m\in K_{m,\widehat t}^*\), we have
\[
    \|\vec \xi+\widehat t g_{\ft}\|_{\widetilde{\vec X}K_{m,\widehat t}^*}
    \le
    m.
\]
Applying \eqref{Eq:uni} with \(a=m\) and \(t=\mu=\widehat t\), we obtain
\begin{equation}\label{Eq:selected-lower}
    \left(1+c_0\widehat t^2-C\eta_n\right)
    M_{n,m}^*(\widehat t)
    \le
    m.
\end{equation}

By optimality of the MNI, \(m\) is no larger than the no-shrink localized cost.
Using \eqref{Eq:baseline} for \(K_{M_0,0}^*=K_0^*\), and using the convexity
comparison of nearby sections to absorb the harmless normalization error, we get
\begin{equation}\label{Eq:m-upper}
    m
    \le
    \left(1+C\eta_n\right)\widetilde M_n(0).
\end{equation}
Combining \eqref{Eq:selected-lower} and \eqref{Eq:m-upper} yields
\begin{equation}\label{Eq:selected-before-comparison}
    \left(1+c_0\widehat t^2-C\eta_n\right)
    M_{n,m}^*(\widehat t)
    \le
    \left(1+C\eta_n\right)\widetilde M_n(0).
\end{equation}

By \eqref{Eq:m-concentration-pre} and the same convexity inclusions for nearby
sections, uniformly for \(t\in[0,2]\),
\[
    (1-C(\widehat t^2+\eta_n))K_{M_0,t}^*
    \subset
    K_{m,t}^*
    \subset
    (1+C(\widehat t^2+\eta_n))K_{M_0,t}^* .
\]
Therefore, by monotonicity and homogeneity of the gauge,
\[
    M_{n,m}^*(t)
    =
    \left(1+O(\widehat t^2+\eta_n)\right)
    \widetilde M_n(t).
\]
Substituting this into \eqref{Eq:selected-before-comparison}, and absorbing the
\(O(\widehat t^2)\)-terms into the quadratic term, gives
\[
    \left(1+c_1\widehat t^2-C\eta_n\right)
    \widetilde M_n(\widehat t)
    \le
    \left(1+C\eta_n\right)\widetilde M_n(0),
\]
for another absolute constant \(c_1>0\). Equivalently,
\begin{equation}\label{Eq:offset-crossing}
    \widetilde M_n(0)
    -
    \left(1+c_1\widehat t^2\right)
    \widetilde M_n(\widehat t)
    \ge
    -C\eta_n
    \left(
        \widetilde M_n(0)+\widetilde M_n(\widehat t)
    \right).
\end{equation}
By Assumption~\ref{ass:ground-truth} and the convexity comparison of nearby
localized sections,
\[
\widetilde M_n(\widehat t)\asymp \widetilde M_n(0).
\]
Thus \eqref{Eq:offset-crossing} implies
\[
    \widetilde\Delta(\widehat t)
    \ge
    -C\eta_n\widetilde M_n(0).
\]
By the definition of the offset localization radius,
\[
    \widehat t\le \widetilde t_\star.
\]
If the sign of \(\widehat t\) has not been fixed a priori, the same argument is
applied on the interval \([-2,2]\), and gives \(|\widehat t|\le\widetilde t_\star\).
Finally,
\[
    \cP_{\ft}(\erm)-\ft
    =
    -\widehat t\,\ft.
\]
Since \(\|\ft\|_2=1\), this proves
\[
    \|\cP_{\ft}(\erm)-\ft\|_2
    =
    |\widehat t|
    \le
    \widetilde t_\star.
\]
This proves the theorem on an event of probability at least \(1-d^{-c}\).
\end{proof}

%% file: proof_outline_thm2.tex
\subsection{Proof of \Cref{Theorem:SPofGP}}\label{ss:SPofGP}
\label{sec:proof-outline-variance-thm}
 First, we prove \Cref{Theorem:SPofGP}, which gives some intuition to the main theorem of this paper
 that is \Cref{Theorem:LoNE}.
\paragraph{Notation:}
We use the following $L^1-L^2$ inequality for the variance, which is the Gaussian analogue of a well-known result of \cite{talagrand1994russo}. 
The result requires a decomposition of the identity operator on $\R^m$, as follows: 
\begin{equation}\label{eqn:decomp}
I_{m} =  \sum_{i=1}^{k}c_i \cP_{E_i},
\end{equation}
Above, for an integer $k \geq 1$, the decomposition involves a sequence of subspaces $\{E_i\}_{i=1}^k$ of $\R^m$, their corresponding orthogonal projectors 
$\cP_{E_i}$, and nonnegative weights $\{c_i\}_{i=1}^k$. In what follows, for a measure $\mu$ and measurable function $f$, we denote  
\(
\|f\|_{L^p(\mu)}^p 
\defn \int \, |f|^p \, \ud \mu,
\)
for $p \in [1, \infty)$.
If the measure $\mu$ is clear from context, we occasionally write 
$\|f\|_{L^p}$ or $\|f\|_p$, and we use the notation of $\|\nabla_{E_i} F\|_{L^p(\gamma_m)} \defn \| \|\nabla_{E_i} F\|_2 \|_{L^p(\gamma_m)}$.

For a positive integer $k \geq 1$, we define $[k] \defn \{1, \dots, k\}$. For a matrix $A \in \R^{n \times d}$ with columns $\{A_i\}_{i=1}^\dimension \subset \R^n$ and a subset $S \subset [d]$, the matrix $A_S \in \R^{n \times |S|}$ is composed by taking the columns $\{A_i\}_{i\in S}$. 
Additionally, for a square matrix $B \in \R^{k \times k}$, we denote 
\[
B^{-\T} = (B^\T)^{-1} = (B^{-1})^\T. 
\]
Consider a facet of $\cF$ of the random polytope $\PN$. Note that it is defined by $S \subset [d]$ and $|S|=n$ and $\vec \eps \in \R^d$ such that $\vec \eps = (\eps_i)_{i \in S} \in \{\pm1\}^n$, and $\{\eps_i \vec X_i\}_{i \in S}$ form a facet. We denote the barycenter and the (scaled) normal, respectively, as
\[
\vec c_{\cF}:= \frac{1}{n}\sum_{i \in S}\eps_i\vec{X}_i \text{ and } \vec n_\cF = \frac{\vec X_S^{-\T} \vec \eps}{\|\vec X_S^{-\T} \vec \eps\|_2^2}
\]
Note that in this scaling $n_{\cF} \in \cF$, and $\cF$ is contained in the hyperplane  
$$\{v \in \R^n : \langle v, \vec n_\cF\rangle = \|\vec n_\cF\|_2^2\}.$$
Throughout the argument below, we let $\cF$ correspond to the facet (with corresponding $n$-subset $S$ and sign vector $\vec \eps$) which is selected by  $\erm$.

\subsubsection*{Preliminaries} 
% Also, we denote the normal $n_{\cS}$ of the facet by $\cF$, and its solution of the MNI $n_{\vec$ to it $\erm(n_{\cS}) = $.
% As every facet has the same distribution, we may assume that $\cF:=\mathrm{Conv}\{X_1,\ldots,X_n\}$.

% Also$c_{\cF}$ denotes the barycenter of $\sum_{i \in \cS}\eps_iX_i$, and $n_{\cF} $ denotes the the normal of the facet, similary $***$. 

\begin{lemma}[Corollary 5 in \cite{cordero2012hypercontractive}]\label{Lem:Hyper}
Let $m \geq 0$ and fix a smooth function $F \colon \R^{m} \to \R$.
Then, under the decomposition~\eqref{eqn:decomp} we have 
\[
    \mathrm{Var}_{\gamma_m}(F) \lesssim \sum_{i=1}^{k}c_i \cdot \frac{\|\nabla_{E_i} F \|_{L^2(\gamma_m)}^2}{1 + \log \tfrac{\|\nabla_{E_i} F \|_{L^2(\gamma_m)}}{\|\nabla_{E_i} F \|_{L^1(\gamma_m)}}}.
\]
\end{lemma}
This inequality is related to the super-concentration phenomenon; see the monograph~\citep{chatterjee2014superbook} for further details. To begin with, we estimate the expected $\ell_1$-norm of the least norm interpolator, $\|\erm\|_1$. To do this, 
we first recall an essential fact --- well-known in the theory of compressed sensing and basis pursuit --- regarding the sparsity of the least-$\ell_1$-norm interpolant under Gaussian design~\citep[\S 3.1.1]{chen1998atomic}.
\begin{lemma}[Sparsity of $\ell_1$-MNI]\label{lem:sparsity}
With probability one, $\erm$  has precisely $n$ nonzero coordinates 
\end{lemma}

A simple proof of Lemma~\ref{lem:sparsity} can be obtained by noting that $\vec X B^d_1$ is a random Gaussian polytope, \emph{i.e.,} we have
\[
    \|\cdot\|_{\vec X B^d_1} = \|\cdot\|_{\mathrm{conv}\{\pm \vec X e_i\}_{i=1}^{d}}.
\]
Hence, it is a simplicial polytope with probability one: each facet is an $(n-1)$-simplex. This means that $\erm$ can be represented as a convex hull of $n$ of the extreme points $\{\pm \vec X e_i\}_{i=1}^\dimension$

\subsubsection{Proof}

To compute the variance of the MNI, we apply the $L^1-L^2$ bound, Lemma~\ref{Lem:Hyper}, to  $$F_{\vec \xi} \colon \vec X \mapsto \sum_{i=1}^{d} |(\erm)_i|,$$
viewed as a function of the $d$ Gaussian columns $\vec X_1,\ldots,\vec X_d$ in $\R^n$, while the input $\vec \xi\in \sqrt n S^{n-1}$ is fixed. Note that this map is the $\ell_1$-norm of the $\ell_1$-MNI as a function of the covariates for a fixed input $\vec \xi$.

Doing so, we obtain 
\begin{equation}\label{eqn:var-decomp}
\begin{aligned}
\mathrm{Var}\Big(\|\erm\|_{1}\Big)&= \mathrm{Var}\Big(\sum_{i=1}^{d}|(\erm)_i|\Big)   \\&\lesssim
\sum_{i=1}^d 
\frac{\|\nabla_{\vec X_i} F\|_{L^2}^2}{1 + \log  
\tfrac{\|\nabla_{\vec X_i} F\|_{L^2}}{\|\nabla_{\vec X_i} F\|_{L^1}}}
\\&=
d \cdot \underbrace{\frac{\|\nabla_{\vec X_1} F\|_{L^2}^2}{1 + \log  
\tfrac{\|\nabla_{\vec X_1} F\|_{L^2}}{\|\nabla_{\vec X_1} F\|_{L^1}}}}_{\defn T_1}
:= d \cdot \, T_1,
\end{aligned}
\end{equation}
where we used the fact that the $d$ summands have the same distribution. 

Below, via a direct calculation, we show that
\begin{equation}\label{eqn:desired-bound-on-term-1}
T_1 \lesssim \frac{\E \|\erm\|_2^2}{d \log^2(d/n)}.
\end{equation}
Assuming inequality~\eqref{eqn:desired-bound-on-term-1} for the moment, note that \Cref{Theorem:SPofGP} follows essentially immediately: in combination with display~\eqref{eqn:var-decomp}, it implies
\[
\Var(\|\erm\|_1) \lesssim \frac{\E \|\erm\|_2^2}{\log^2(d/n)} \asymp 
\frac{(\E \|\erm\|_1)^2}{n \log(d/n)^2},
\]
where we used that \cite{wang2022tight} implies
 $$ n \cdot \E \|\erm\|_2^2 \asymp M_{n,d}^2 \asymp [\E \|\erm\|_1]^2 \asymp \frac{n}{\log(d/n)}.$$ Hence by dividing both sides by $M_{n,d}^2 = [\E (\|\erm\|_1)]^2$, we obtain the claimed result. Note that the final equality made use of \Cref{Theorem:LoNE}. It remains to prove \eqref{eqn:desired-bound-on-term-1}.

\subsubsection{Proof of \eqref{eqn:desired-bound-on-term-1}}
On the event that $\vec X_{\cS}$ forms
the facet $\cF$, with probability $1$ we have
\begin{equation}\label{Eq:theyddyLin}
\nabla_{\vec X_1} F(\vec X, \vec \xi) = \begin{cases} 
-(\erm)_1 
\vec X_S^{-\T} \vec \eps & 1 \in S \\ 
0 & 1 \not \in S 
\end{cases},
\end{equation}
note that $S$ is a random set with $n$-elements.  
\begin{proof}[Proof of \eqref{Eq:theyddyLin}]We assume throughout that $(\vec X, \vec \xi)$ lie in this event. 
Since $(\erm)_S = (\vec X_S)^{-1} \vec \xi$, we have 
\[
F(\vec X, \vec \xi) = (\erm)_S^\T \vec{\eps}= \langle \vec \xi, \vec X_S^{-\T} \vec{\eps} \rangle. 
\]
From, this we clearly see 
$\nabla_{\vec \xi} F(\vec X, \vec \xi) = \vec X_S^{-\T}\vec{\eps}$.

Similarly, we can write, when $1 \in S$ that
\begin{equation}\label{eqn:interp-eqn}
\vec \xi = \vec X_1 \cdot (\erm)_1 + \sum_{i \in S, i \neq 1} \vec X_i \cdot (\erm)_i
\end{equation}
Let $F_i(\vec X, \vec \xi) = (\erm)_i$, and let $\vec D \in \R^{n \times n}$
be the matrix with columns $\nabla_{\vec X_1} F_i(\vec X, \vec \xi)$.

Differentiating the $j$th coordinate of equation~\eqref{eqn:interp-eqn} in $\vec X_1$, we obtain for every for every $j \in [n]$,
\[
-(\erm)_1 \cdot  \vec \eps_j = \sum_{i \in S} 
(\vec X_i)_j \nabla_{\vec X_1} F_i(\vec X, \xi) = 
\vec D \vec X_S^\T \vec \eps_j
\]
Equivalently, we have  
\[
\vec D = -(\erm)_1 \vec X_S^{-\T}.
\]
On the other hand, since $\sum_{i \in S} F_i = F$, we have
\[
\nabla_{\vec X_1} F(\vec X, \vec \xi) = \vec D \vec \eps = - 
(\erm)_1 \vec X_S^{-\T} \vec \eps,
\]
when $1 \in S$. If $1 \not \in S$, then the same argument above shows that $\vec D = \vec 0$, and thus
\[
\nabla_{\vec X_1} F(\vec X, \vec \xi) = \vec 0.\qedhere
\]
\end{proof}

Since $(\erm )_1$ is nonzero with probability $n/d$, we immediately have by Jensen's inequality 
\begin{equation}
\frac{\|\nabla_{\vec X_1} F\|_{L^2}}{\|\nabla_{\vec X_1} F\|_{L^1}} = \frac{\sqrt{\tfrac{n}{d} \E [ \| \nabla_{\vec X_1} F\|_2^2 \mid 1 \in S]}}{\tfrac{n}{d} \E[\|\nabla_{\vec X_1} F\|_2 \mid 1 \in S]} \geq \sqrt{\frac{d}{n}}.
\end{equation} 
Therefore, using the fact that $\|\vec X_S^{-\T} \vec \eps\|_2^2 = \tfrac{1}{\|\vec n_\cF\|_2^2}$, we have 
\begin{equation}\label{eqn:t1-bound}
T_1 
\lesssim  \frac{1}{\log(d/n)} \E\Big[|(\erm)_1|^2 \tfrac{1}{\|\vec n_\cF\|_2^2}\Big]= 
\frac{1}{d \log(d/n)} \E \Big[\|\erm\|_2^2 \frac{1}{\|\vec n_\cF\|_2^2}\Big].
\end{equation}
where we used that $d \cdot \E[|(\erm)_1|^2] = \E \|\erm\|_{2}^2$ and the fact that the coordinates all have the same distribution. Finally, we write
$\mathcal{E} = \{\|\vec n_\cF\|_2 \gtrsim \sqrt{\log(d/n)}\}$, which by Lemma \ref{lem:simplices} below holds with probability of at least $1-\exp(-c\sqrt{nd})$. Then,
we clearly have 
\begin{equation}\label{ineq:bound-on-the-expected-product}
\E \Big[\|\erm\|_2^2 \frac{1}{\|\vec n_\cF\|_2^2}\Big] \lesssim 
\frac{1}{\log(d/n)}\E \|\erm\|_2^2 
+ 
\E \Big[\1_{\mathcal{E}^c} \|\vec X_S^{-1} \vec \xi\|_2^2 \|\vec X_S^{-\T} \vec \eps\|_2^2\Big]
\end{equation}
A very cheap bound on the remaining term can be obtained as follows. 
First, note that $\PN$ contains the ball $\tfrac{\sigma_{\rm min}(\vec X)}{\sqrt{d}} B^n_2$; this can be verified directly from the inclusion $\tfrac{1}{\sqrt{d}} B^d_2\subset B^d_1$ and 
\[
h_{\PN}(v) \geq \frac{1}{\sqrt{d}} h_{\vec X B^d_2}(v) = 
\frac{1}{\sqrt{d}} \|\vec X^\T v\|_2 
\geq \frac{\sigma_{\rm min}(\vec X)}{\sqrt{d}} \|v\|_2 = 
h_{\tfrac{\sigma_{\rm min}(\vec X)}{\sqrt{d}} B^n_2}(v).
\]
Therefore every supporting hyperplane of $\PN$ has distance at least
$\sigma_{\rm min}(\vec X)/\sqrt d$ from the origin. Since the selected facet
$\cF$ lies on the hyperplane
\[
\{v\in\R^n:\langle v,\vec n_\cF\rangle=\|\vec n_\cF\|_2^2\},
\]
we have
\[
    \frac{1}{\|\vec n_\cF\|_2^2}
    \leq
    \frac{d}{\sigma_{\rm min}^2(\vec X)}.
\]
Moreover,
\[
    \|\erm\|_2
    \leq
    \|\erm\|_1
    =
    \|\vec \xi\|_{\PN}.
\]
Using the minimum Euclidean norm solution, we get
\[
\|\vec \xi\|_{\PN}
=
\inf\{\|w\|_1:\vec Xw=\vec \xi\}
\leq
\sqrt d\,\|\vec X^\dagger \vec \xi\|_2
\leq
\sqrt d\,\frac{\|\vec \xi\|_2}{\sigma_{\rm min}(\vec X)}
=
\frac{\sqrt{dn}}{\sigma_{\rm min}(\vec X)},
\]
where we used $\|\vec \xi\|_2=\sqrt n$. Hence
\[
    \|\erm\|_2^2
    \leq
    \frac{dn}{\sigma_{\rm min}^2(\vec X)}.
\]
Applying Cauchy--Schwarz and then Lemma~\ref{lem:sigma-min-negative-moment},
with $q=8$, we obtain
\begin{equation}\label{ineq:holder-upper-on-bad-event}
\begin{aligned}
&\E \Big[\1_{\mathcal{E}^c} \|\vec X_S^{-1} \vec \xi\|_2^2 \|\vec X_S^{-\T} \vec \eps\|_2^2\Big] \\
&\qquad\leq
nd^2\,\E\left[\1_{\mathcal E^c}\sigma_{\min}(\vec X)^{-4}\right] \\
&\qquad\leq
nd^2\,\P(\mathcal E^c)^{1/2}
\left(\E\sigma_{\min}(\vec X)^{-8}\right)^{1/2}
\lesssim
n\,\P(\mathcal E^c)^{1/2}
\lesssim
\exp(-c\sqrt{nd}).
\end{aligned}
\end{equation}
% By Cauchy-Schwarz, we have
% \[
% \E \Big[\1_{\mathcal{E}^c} \|\vec X_S^{-1} \vec \xi\|_2^2 \|\vec X_S^{-\T} \vec \eps\|_2^2\Big] 
% \leq \sqrt{\P(\cE^c)} \sqrt{\E \|\vec \xi\|_2^4 \|\vec \eps\|_2^4 \E \max_{|S| = n} \mathrm{OP}norm{\vec X_S^{-1}}^8} 
% % \lesssim 
% % n^2  \sqrt{\E \max_{|S| = n} \mathrm{OP}norm{\vec X_S^{-1}}^8}\sqrt{\P(\cE^c)}
% % \lesssim \frac{1}{n^4},
% \]
% Therefore, 
% which can be established by using that we can assume that $ \mathrm{OP}norm{\vec X_S^{-1}} \leq \exp((C+\eps) \cdot n\log(n))$ with probability of at $1-\exp(cn^2\eps)$ by applying Lemma \ref{Lem:Goodman} below, and using that we assume that $d \gtrsim n\log(n)^{C}$. 
We can now combine the probability estimate of Lemma~\ref{lem:simplices} and inequality~\eqref{ineq:holder-upper-on-bad-event} and obtain that 
\[
\E \Big[\|\erm\|_2^2 \frac{1}{\|\vec n_\cF\|_2^2}\Big] 
\lesssim 
\frac{1}{\log(d/n)} \E \|\erm\|_2^2 + 
\exp(-c\sqrt{nd}).
\]
The exponentially small term is negligible compared with the first term. Indeed,
since the $\ell_1$-MNI has at most $n$ nonzero coordinates, Lemma~\ref{lem:sparsity}
gives
\[
    \|\erm\|_1^2\leq n\|\erm\|_2^2.
\]
Therefore,
\[
    \E\|\erm\|_2^2
    \geq
    \frac{(\E\|\erm\|_1)^2}{n}
    =
    \frac{M_{n,d}^2}{n}
    \asymp
    \frac{1}{\log(d/n)},
\]
where we used the standard estimate $M_{n,d}^2\asymp n/\log(d/n)$. Thus,
\[
    \exp(-c\sqrt{nd})
    \lesssim
    \frac{1}{\log(d/n)}\E\|\erm\|_2^2,
\]
and hence
\[
\E \Big[\|\erm\|_2^2 \frac{1}{\|\vec n_\cF\|_2^2}\Big] 
\lesssim 
\frac{1}{\log(d/n)} \E \|\erm\|_2^2.
\]
This completes the proof of inequality~\eqref{eqn:desired-bound-on-term-1}, and therefore finishes the proof of Theorem~\ref{Theorem:SPofGP}.
\subsubsection{Known technical lemmas}
The following lemma using the inradius of $\PN$, which is at least on the order of $\sqrt{\log (d/n)}$, by the result of ~\cite{gluskin88extremal}.
\begin{lemma}\label{lem:simplices}
Let $\cF$ denote the facet on which $\erm$ lies; we denote the corresponding signs by $\vec \eps$ and the subset indices with nonzero entries by $S \subset  [d], |S| = n$.  Then, its normal satisfies
\[
\|\vec n_\cF \|_2^2 \gtrsim 
\log (d/n),
\]
with probability of at least $1-2\exp(-c\sqrt{dn})$ over $\vec X$, 
for a universal constant $c > 0$.
\end{lemma}
The next lemma is a standard estimate via the small ball method.
\begin{lemma}[Negative moments of the smallest singular value]
\label{lem:sigma-min-negative-moment}
Let \(G\in\R^{n\times d}\) have i.i.d. \(N(0,1)\) entries, and assume
\(d\ge Cn\) for a sufficiently large absolute constant \(C\). Then, for every
fixed \(q\ge1\),
\[
    \left(\E \sigma_{\min}(G)^{-q}\right)^{1/q}
    \lesssim_q
    d^{-1/2}.
\]
\end{lemma}

\begin{proof}
We first recall the standard small-ball estimate for the smallest singular
value of a rectangular Gaussian matrix: for \(0<\eps<c\),
\[
    \Pr\left\{
    \sigma_{\min}(G)\le \eps\sqrt d
    \right\}
    \le
    (C\eps)^{cd}.
\]
For completeness, let us recall the usual net proof. Since
\[
    \sigma_{\min}(G)
    =
    \inf_{u\in S^{n-1}}\|G^\top u\|_2,
\]
and for fixed \(u\in S^{n-1}\), \(G^\top u\sim N(0,I_d)\), we have
\[
    \Pr\left\{
    \|G^\top u\|_2\le \eps\sqrt d
    \right\}
    \le
    (C\eps)^d.
\]
Let \(\mathcal N\) be a \(\delta\)-net of \(S^{n-1}\), with
\[
    |\mathcal N|\le (3/\delta)^n.
\]
On the event \(\|G\|_{\mathrm{OP}}\le L\sqrt d\), if
\(\sigma_{\min}(G)\le \eps\sqrt d\), then for some \(v\in\mathcal N\),
\[
    \|G^\top v\|_2
    \le
    \eps\sqrt d+\delta L\sqrt d.
\]
Taking \(L=\eps^{-1/2}\) and \(\delta=\eps/L=\eps^{3/2}\), we get
\[
    \|G^\top v\|_2\le 2\eps\sqrt d.
\]
Therefore,
\[
\Pr\left\{
\sigma_{\min}(G)\le \eps\sqrt d,\ 
\|G\|_{\mathrm{OP}}\le L\sqrt d
\right\}
\le
\left(\frac{3}{\eps^{3/2}}\right)^n(C\eps)^d
\le
(C\eps)^{cd},
\]
where we used \(d\ge Cn\). On the other hand, the standard Gaussian
operator-norm tail gives
\[
    \Pr\left\{\|G\|_{\mathrm{OP}}>L\sqrt d\right\}
    \le
    \exp(-cL^2d)
    =
    \exp(-cd/\eps)
    \le
    (C\eps)^{cd}.
\]
This proves the small-ball estimate.

Now set
\[
    Z:=\frac{\sqrt d}{\sigma_{\min}(G)}.
\]
The small-ball estimate implies that, for \(u\ge C\),
\[
    \Pr\{Z\ge u\}
    =
    \Pr\left\{
    \sigma_{\min}(G)\le \frac{\sqrt d}{u}
    \right\}
    \le
    \left(\frac{C}{u}\right)^{cd}.
\]
Hence, for any fixed \(q\ge1\), provided \(d\ge C_q n\),
\[
\begin{aligned}
    \E Z^q
    &=
    q\int_0^\infty u^{q-1}\Pr\{Z\ge u\}\,du  \\
    &\le
    C_q
    +
    q\int_C^\infty
    u^{q-1}
    \left(\frac{C}{u}\right)^{cd}
    du
    \le
    C_q.
\end{aligned}
\]
Therefore,
\[
    \E \sigma_{\min}(G)^{-q}
    =
    d^{-q/2}\E Z^q
    \le
    C_qd^{-q/2}.
\]
This proves the lemma.
\end{proof}

%% file: proof_outline_thm3.tex
\subsection{Proof of Theorem~\ref{Theorem:LoNE}}
Before we proceed with our proof, we advise reading sections \Cref{S:SketchPf2} and \Cref{ss:SPofGP}, which contain some preliminary steps that are used in this proof. We set a notation that we use throughout the proof
\[
    L:=L \text{ and } \SPO = M_{n,d}\PN.
\]
% \input{SPofThm3}
% Hence, we expect the following: If ``most'' facets have the behavior of Fluery's, most of their volume would lie in their thin shell of length $1$. 
 
%  Hence, Lemma \ref{Lem:Energy} implies that most of the volume corresponds to the thin shells of the original simplex $\triangle$.
%  Next, using that Then, we would show that via KLS property on $\cF$, and Fluery's distribution properties that most of the volume of $P_{n,d}$ lies on the thin shell

% Hence, Pythagoras's rule, we would expect that most of the volume of $\PN$, has an $\ell_2$ length that satisfies
% , and the volume is $\sqrt{2}$ far from the normal---as most of the volume lies nearby the sections, and in particular of  $\theta \in \Sn$. As, we know the distribution of the normal 
% \[
%     \sqrt{\|\vec \eps/n\|_{2}^2 + \lp\frac{1+o(1)}{\sqrt{n}}\rp^2} \approx \sqrt{\frac{2}{n}},
% \]
% where the first term emerges from the bary-center and the second one is from the shell, that they have the same $\ell_2$-norm of $1/\sqrt{n}$.
% --- a  special property of Gaussian polytopes. 

\subsubsection{Preliminaries}
 We state a lemma that follows from \cite{paouris2016gaussian}, see also \cite{klartag2007small}. Let $r(A),R(A)$ be the inradius and outradius of a set $A$. For a matrix $A$, $\|A\|$ denotes its operator norm.
 We also use the notation of $\bar{B}_n:=\sqrt{n} \cdot B_n$. 
\begin{lemma}[Boosted Dvoretzky's theorem]
\label{lem:boosted-dvoretzky}
Assume that $d \geq C_1n$. Then, the following holds with probability of $1-\exp(-c_1n)$:
\begin{equation}\label{ineq:dvoretzky-radii}
M^*(K)\lp1-c_2\sqrt{n \, \Var\lp\tfrac{ \|\vec \xi\|_{K^{\circ}}}{M^*(K)}\rp}\rp \leq r(\vec X K) \leq
R(\vec X K) \leq M^*(K) \cdot \lp 1 + c_3\tfrac{R(\sqrt{n} \cdot K)}{M^*(K)}\rp,
\end{equation}
for any convex body $K \subset \R^n$ and universal constants $c_1, c_2, c_3>0$.
\end{lemma}
Note that since $\|\vec \xi\|_{K^\circ}$ is $R(K)$-Lipschitz, the Gaussian Poincaré inequality shows the deviation from $M^\star(K)$ in the upper inclusion is larger than the deviation in the lower inclusion implied by~\eqref{ineq:dvoretzky-radii}.
Note that the density of $\vec Y$ is ``almost'' a Gaussian matrix up to a   $|\vec Y^{\top}\vec Y| \cdot |\triangle|$.  For this, we use a result of \cite{goodman1963distribution}, showing that
\[
    \ln |\vec Y^{\top}\vec{Y}| =  \sum_{i=2}^{n}\ln Z_i \text{ where the $Z_i$ are independent and $Z_i \sim \chi^{2}_i$.}
\]
The following result follows by computing the moment generating function of the corresponding $\chi^2$ random variables and a Chernoff bound. 
\begin{lemma}\label{Lem:Goodman}
Let $G$ be an $(n -1)\times n$ Gaussian matrix. Set $S_n: = \ln |G|$.  There is a universal constant $c > 0$ sufficiently large such that 
the following inequalities hold for $t \geq 0$. 
\begin{enumerate} 
\item For the upper tail:
\[
    \Pr(S_n - \E S_n \geq t) \leq \exp \left(- \frac{t^2}{c \ln n}\right).
\]
\item 
For the lower tail: 
\[
\Pr(S_n - \E S_n \leq -t)
\leq
\begin{cases}
\exp \left(- \frac{t^2}{c \ln n}\right)  & t <  \ln n \\ 
c n e^{-t/c} & t \geq \ln n.
\end{cases} 
\] 
\end{enumerate}
% \[
%         \Pr(S_n - \E S_n \geq t) \leq \exp \left(- \frac{t^2}{4\ln(n)}\right) , 
%         \quad \mbox{for all}~t \geq 0,
% \]
% and
% \[
%      \Pr(S_n - \E S_n \leq -t) \leq Cnt\exp(-t), 
%         \quad \mbox{for all}~t \geq \ln(n),
% \]
\end{lemma}
Hence, using the last lemma and the upper bound on the number of facets, we obtain the following one-tailed estimate: 
\begin{corollary}\label{C:GoodFlu}
With high probability over $P_{n,d}$ it holds for all facets simultaneously 
\[
    \forall \cF \in \cF_{n-1}(\PN) \quad \ln |\vec X_{\cF}| -  \E S_n \lesssim \sqrt{n \cdot \lln(d/n) \cdot \ln(n)}
\]
\end{corollary}
The next corollary follows by tail integration, using the relation between $G$, the Gaussian matrix, and $\vec {Y} $ from Fleury's distribution, which is almost Gaussian.
\begin{corollary}\label{C:GoodFlu2}
Let $A$ be an event whose probability under $G$ is smaller than $n^{-100}$. Then, it is smaller than $n^{-99}$ over the probability space of $\vec Y$.
\end{corollary}
\textit{In words,} it means a rare event in $G$ is also rare in $\vec {Y} $, however, it does not mean that those densities are close, for example, in $TV$-distance. We finally state a lemma regarding the distribution of $T_{n,d}$. To state it, let $t_{n,d}$ denote the mode of the density associated with $T_{n,d}$. Equivalently, if
\[
    p_{n,d}(t)\propto \Psi(t)^{d-n}\exp(-nt^2/2),
    \qquad
    \Psi(t)=2\Phi(t)-1,
\]
then $t_{n,d}$ is the unique positive solution of
\[
    t_{n,d}
    =
    \frac{d-n}{n}\frac{2\varphi(t_{n,d})}{\Psi(t_{n,d})}.
\]
In particular,
\[
    t_{n,d}^2
    =
    2L-\lln(d/n)-\ln\pi+o(1).
\]
\begin{lemma}
\label{prop:tail-of-Tnd}
 Let $0 \leq \eps \leq c/\log(d/n)$ and assume that $ Cd \leq n \leq \exp(n^c)$. Then,
\[
\Pr\Big\{|T_{n, d} - \mathrm{Med} T_{n, d}| \geq \eps \cdot t_{n,d} \Big\}
\leq 
\exp\Big\{-c_1n t_{n,d}^{4}\eps^2\Big\}
\leq
\exp\Big\{-c_2nL^2\eps^2\Big\},
\]
and in particular
\[
| 
    t_{n,d} - \mathrm{Med} T_{n, d}| \lesssim \frac{1}{\sqrt{n}\,t_{n,d}}
\lesssim \frac{1}{\sqrt{n L}}.
\]
\end{lemma}
The proof of the last lemma follows from a direct computation on the distribution of Lemma \ref{Lem:Fleury}. Indeed, the quadratic approximation of the log-density around $t_{n,d}$ is valid for displacements $u=\eps t_{n,d}$ satisfying $u t_{n,d}=O(1)$, which is precisely $\eps=O(t_{n,d}^{-2})$. This is also the sharp range for the displayed quadratic exponent: on the upper tail, once $u t_{n,d}\gg 1$, the log-density loss is only of order $nt_{n,d}u$, and hence one cannot keep a uniform bound of order $\exp(-cnt_{n,d}^{2}u^{2})$. Next, we would need the following lemma:
%%%%%%%%%%%%%%%%%%%%%%%%%%%
\begin{lemma}[Lemma 2.3 in \cite{paouris2004estimates}; see also \cite{alonso2015gaussian}]
\label{lem:david-alonso}
Let $K \subset \R^n$ be an isotropic convex body with $\ell_2$-diameter $C\sqrt{n}$ and $ 2 \leq t \leq C\sqrt{n}$. Then, with probability $1-\exp(-c_1t^2)$ over $\theta \sim U(\Sn)$, the marginal $\langle X, \theta\rangle$ satisfies
\[
    \Pr_{X \sim U(K)}\Big\{|\langle X, \theta\rangle| \geq t\Big\}
    \leq 
    2\exp(-c_2t^2),
\]
where the constants $c_1, c_2, c_3 > 0$ depend on $C \geq 0$. 
\end{lemma}

% \subsubsection{Full Proof of \Cref{Theorem:LoNE}}
\subsubsection[Step I: Global Properties of Pnd]{\textcolor{red}{Step I:} Global Properties of $\PN$}
First, we now provide a cheap estimate of the volume of the polytope $P_{n,d}$, and also on the mean norm $M_{n,d}$, which is sufficient to carry out the argument for~\Cref{Theorem:LoNE}.  We consider the truncated polytope and corresponding Gaussian mean width, 
\[
    K(n,d):= B_1^{d} \cap \tfrac{1}{n} B_{\infty}^{d}, \quad \mbox{and} \quad 
    \MND:= M^*(K(n,d)).
\]
We use the shorthand notation $\SPO = M_{n,d} P_{n,d}$.
\begin{lemma}\label{Lem:CheapVolumeEstimate}
Suppose that $d \geq Cn$, for some $C \geq 0$ large enough. Then, with probability at least $1-\exp(-cn)$ (over $\vec X$), we have:
\begin{enumerate}[label=(\roman*)]
\item It holds that  
\label{item:inclusion-via-K-n-d}
\[
\lp 1-\frac{C}{L} \rp \cdot \MND \cdot  B_n \subset \PN     
\]
and in particular $M_{n,d} \leq (1+\frac{C}{L})\cdot \MND^{-1}$.
\item 
\label{item:volume-via-K-n-d}
The volume of $\PN$ satisfies the bounds 
\[
 |P_{n,d}| \leq \exp(n/\sqrt{L})|\MND \cdot B_n|
 \quad \mbox{and} \quad \exp(-C_1 \cdot n/\sqrt{L}) \leq  \frac{|\SPO|}{|\bar{B}_n|} \leq  \exp(C_1 \cdot n/\sqrt{L}),
\]
and consequently
$M_{n,d} \geq (1 - \frac{c}{\sqrt{L}}) \cdot \MND^{-1}$.
\end{enumerate}
\end{lemma}
The proof of this lemma appears below. Also, note that 
$$ 
\Big| \frac{\|\vec{n}_{\cF}\|_{2}}{\MND}  -1\Big| \lesssim \frac{1}{\sqrt{L}},
$$
and  the same bounds hold for $\E|\PN|$, i.e.
\begin{equation}
   \exp(-Cn/L) \leq \frac{\E |\SPO|}{|B_n|} \leq \exp(n/\sqrt{L})    
\end{equation}
via integrating the tail of $R(\PN)$ with the tail bound of Dvoretzky's theorem.

Here, we provide its sketch: We apply Dvoretzky's theorem on the set $K(n,d) \subset B^d_1$, from which we can obtain a lower bound on the inradius of the projected polytope $P_{n,d}$, which in turn bounds from above the expected norm of the projected body, $M_{n,d}$.
\subsubsection{\textcolor{red}{Step II:} Volume reductions}\label{ss:strategy}
Here, we state few volume reductions techniques that we use throughout our proof.
To see this, note that for every fixed $\mathrm{Im}\vec X$, the map
\[
   \vec \xi \mapsto \frac{\|\vec \xi\|_{n}}{\E \|\vec \xi\|_{n}}
\]
is $O(\E \|\vec \xi\|_n/\sqrt{n})$-Lipschitz by the ball inclusion under the event of the previous step. Therefore, by \Cref{Theorem:SPofGP}, we can assume  that
\[
  \E | \SPO \cap \bar{B}_n| \geq \exp\lp-c\sqrt{\frac{n}{L}}\rp|\bar{B}_n|
\]
Let $A$ be an event on a facet $\cF$ such that under Fleury's distribution, it holds with probability of at most $\exp(-Cn/\sqrt{L})$. Then, we know by Markov's inequality, that  
\[
    \left|\bigcup_{\cF \text{ satisfies $A$}} M_{n,d} \cdot  \cF\right| \leq \exp(-C_3n/\sqrt{L})|\partial\bar{B}_n|, 
\]
where we used \Cref{Lem:Goodman} and its corollaries. 
% This means that by elementary geometry and standard Lipschitz concentration, the $\erm$ would lie in those facets with probability of $1-\exp(n/\sqrt{L})$. Meaning that
% \[
%     \Pr(|\|\vec \xi\|_n - \E \|\vec \xi\|_n| \geq CL^{1/4}\E \|\vec \xi\|_n) \leq 2\exp(-c_1 C \cdot n/\sqrt{L})
% \]

This means that for a set of directions $B \subset \sqrt{n} \cdot S^{n-1}$  lies in the cones whose their facets satisfy the event $A$, then it means that we have to inflate $P_{n,d}$ by at least $1 + 1/\sqrt{C/L}$ to contain $0.5 $ of the measure of $B$. To see this monotonicity of the volume, and its $n$-homogeneity, means that we need to inflate by $x$ such that
\[
    (1+x)^{n} \geq \exp(C n/\sqrt{L})
\]
which is equal to $x \geq CL^{-1/2}$,
% \begin{align*}
%    |(1+CL^{-1/2}) M_{n,d} \cdot \mathrm{Conv}\{0, \bigcup_{\cF \text{ satisfies $A$}}\cF\}|   (1+CL^{-1/4}) M_{n,d} \cdot \mathrm{Conv}\{0, \bigcup_{\cF \text{ satisfies $A$}}\cF\}| \leq \exp(-C_1 \cdot n/\sqrt{L})|\bar{B}_n|,
% \end{align*}
 However, this would contradict the ball inclusion and estimates on $M_{n,d}$ from the previous step, as $\bar{B}_n \subset (1 + \frac{C_1}{\sqrt{L}}) \SPO$.

% as $\mathrm{conv}\{0, \text{$\cF$ satisfies $A$}\}| \gtrsim \sigma_n(B)/n \geq \exp(-n/\sqrt{L})|\bar{B}_n|$. 
Therefore, we obtained the following corollary:
\begin{corollary}\label{Lem:PositiveTheyddy2}
Let $A$ be an event on a facet $\cF$ such that under Fleury's distribution, it holds with probability of at most $\exp(-Cn/\sqrt{L})$. Then, with probability of at least $1-\exp(-C_2n/\sqrt{L})$, $\erm $ would lie in a facet that satisfies $A^{c}$.
\end{corollary} 
To prove the lower bound on the MSE,  we show that
\[
B^{-}_{n,d}(M_{n,d}):= M_{n,d} \cdot
\bigcup_{\text{$\triangle$ is $n$-dimensional facet of $B_1^d$}} \underbrace{\frac{\vec{\eps}}{n} + \triangle_{c} \cap  \lp\frac{1-\frac{C}{L^{1/4}}}{\sqrt{n}} \cdot B_{n-1}\rp}_{:=\triangle_{-}},
\]
(where we view $\triangle_{c}$ as a set in $\R^{n-1}$ and $B_{n-1}$ is the $\ell_2$-ball with radius one in $\R^{n-1}$) does not have enough volume, i.e., we show that
\[
    |\vec X B^{-}_{n,d}(M_{n,d})|  \leq \exp(-c_1 C \cdot n/\sqrt{L})|\SPO| \leq  \exp(-c_1 C_2 \cdot n/\sqrt{L})|\bar{B}_n|
\]
which means that with probability of at least $1-\exp(-cn/\sqrt{L})$
\[
   \|\erm\|_2^2 \geq \frac{1-\frac{C}{L^{1/4}}}{L}.
\]
% Following the proof of the previous step, the lemma we know that the total surface area of facets that do not satisfies a Fluery's event with probability of at least $1-\exp(-C_1n/\sqrt{L})$ is smaller than $e^{-n/\sqrt{L}}|\partial \SPO|$ surface area, and for $C \geq 0$ large enough. As it holds that  

% The second part of this lemma follows from considerations similar to those in Step 2.
% However, by combining Lemmas \ref{Lem:BePositiveTheyddy} and \ref{Lem:bePositiveTheyddy3}, we obtain a contradiction to the fact that  
% \[
% |\partial \SPO \cap (1+\frac{C}{\sqrt{nL}})\bar{B}_n| \geq e^{-\frac{Cn}{\sqrt{L}}}|\partial \bar{B}_n|,
% \]
% and the proof is complete.
 
% we know by Markov inequality that $\erm$ will not lie in these facets with probability of $1-\exp(-cn/\sqrt{L})$.
% this would provide the lower bound on the MSE.
To prove the upper bound, we would need to adopt a more refined approach, as will become clearer below. Very roughly speaking, we would show that 
\[
B^{+}_{n,d}(M_{n,d}):= M_{n,d} \cdot
\bigcup_{\text{$\triangle$ is $n$-dimensional facet of $B_1^d$}} \underbrace{\frac{\vec{\eps}}{n} + \triangle_{c} \cap  \lp\frac{1+\frac{C}{L^{1/4}}}{\sqrt{n}} \cdot B_{n-1}\rp^{c}}_{:=\triangle_{+}},
\] 
\[
    |\vec X B^{-}_{n,d}(M_{n,d}) \cap \bar{B}_n|  \leq \exp(-c_1 C \cdot n/\sqrt{L})|\SPO| \leq  \exp(-c_1 C_2 \cdot n/\sqrt{L})|\bar{B}_n|
\]
and as the $\|\vec \xi\|_2 =  \sqrt{n}$,  provides the desired bound with the same probability. Yet, this argument is much more delicate, as will be seen below. 

% Also, only in the last step, we would see that with probability of $1-\exp(-cnL^{-2})$, it holds
% \[
%     \|\erm\|_2^2 \geq \frac{1 - \frac{C}{L^2}}{L}
% \]
% and
% \[
%     \|\erm\|_2^2 \leq \frac{1 + \frac{C}{L^2}}{L}
% \]
\subsubsection{\textcolor{red}{Step III}: Canonical Simplex}
First, we introduce the following useful lemma:
\begin{lemma}[Small ball probabilities for a simplex]\label{Lem:SBSimplex}
The following holds for all $\eps \in (0,1):$
\[
   \Pr_{Z \sim U(\triangle_c)}(\|Z\|_{2} \leq (1-\eps)\E\|Z\|_{2}) \leq \exp(-cn\eps^2),
\]
and also the converse holds, i.e., for  $\eps \in (0,c)$
\[
     \Pr_{Z \sim U(\triangle_c)}(\|Z\|_{2} \leq (1-\eps)\E\|Z\|_{2}) \geq \exp(-c_2n\eps^2).
\]
\end{lemma}
The proof of this lemma follows from standard calculations of MGFs, and noting that $Z \sim \mathrm{Unif}(\triangle)$, can also be represented as
\[
    Z \sim  \frac{(Z_1,\ldots,Z_n)}{\sum_{i=1}^{n}Z_i}
\]
where $Z_1,\ldots,Z_n \sim \mathrm{Exp}(1)$, for completeness it appears below. Surprisingly, it does not appear in the literature.
% The next shows the same for 
% \begin{lemma}\label{Lem:SBSimplex}
% Let $\cF:= \tG_{\cS}(\triangle - \vec\eps/n)$ be a facet that is drawn from Fluery's distribution. Then,  for every $\eps \in (0,1)$, it holds that  
% \[
%    \Pr_{Z \sim U(\triangle),\vec X_{\cS}}(\|Z\|_{2} \leq (1-\eps)\E\|Z\|_{2})) \leq \exp(-cn\eps^2).
% \]
% \end{lemma}
 Now, we prove the easy part of the theorem, on the lower bound on the MSE
\begin{lemma}[Lower bound on the MSE]
    With probability of at least $1-C_1\exp(n/\sqrt{L})$, it holds 
    \[
        \|\erm\|_2^2 \geq \frac{1 -\frac{C}{L^{1/4}}}{L}.
    \]
\end{lemma}
\begin{proof}
Recall that we need to prove that $\erm \notin B^{-}_{n,d}(M_{n,d})$  as discussed in Step II, by showing 
\[
    \exp(-c\sqrt{n}) \cdot |\partial \bar{B}_n| \leq |\partial \SPO| \leq \exp(c \cdot n/\sqrt{L})|\partial \bar{B}_n|.
\]
By choosing $\eps \geq C_3\cdot L^{-1/4}$ for some $C_3 \geq 0$ in Lemma \ref{Lem:SBSimplex}, we have that
\[
|\triangle_{-}| \leq e^{-C \cdot n/\sqrt{L}}|\triangle|
\]
for some large enough $C \geq 0$. Now, by Lemma \ref{Lem:Goodman}, it holds 
\[  
    \forall \cF \in \cF_{n-1}(\PN) \quad  |\vec X_{S}\triangle_{-}| \leq \exp(-c_1 \cdot n/\sqrt{L}) \cdot \E |\cF|,
\]
Therefore, we conclude that
\begin{equation}
\begin{aligned}
    \left| \vec X B^{-}_{n,d}(M_{n,d})\right| &\leq \exp(-C \cdot n/\sqrt{L})|\partial \SPO| \leq \exp(-C_2\cdot n/\sqrt{L}) \cdot |\partial \bar{B}_n| 
\end{aligned}
    \end{equation}
where we used the fact that we can set $C \geq 0$ to be large enough. Hence, by \Cref{ss:strategy}, the claim follows.

\end{proof}
To see why this approach cannot work for the upper bound side, note that one can easily show (cf. \cite{paouris2006concentration}) that
\[
     \Pr_{Z \sim U(\triangle_c)}(\|Z\|_2 \geq (1+\eps)\E \|Z\|_2) \geq \exp(-c\sqrt{n}\eps),
\]
Namely, the tails of $\|Z\|_2$ exhibit a different behavior; one is sub-Gaussian, and the other is sub-exponential. Meaning that we would need to use a different argument, as
\begin{equation}
\begin{aligned}
   \left|\vec X B^{+}_{n,d}(M_{n,d}) \right| 
    \geq e^{-c\sqrt{\frac{n}{L}}}|\partial \SPO|, 
\end{aligned}
\end{equation}
i.e., the ``outer part'' has too much volume, as we cannot infer that
\[
    e^{-c\sqrt{\frac{n}{L}}}|\partial \SPO| \ll |\bar{B}_n|.
\]
% However, we would manage to show that 
% \begin{equation}
% \begin{aligned}
%    \left|\vec X B^{+}_{n,d}(M_{n,d}) \cap \bar{B}_n\right| 
%     \leq \exp(-n/\sqrt{L})|\partial \bar{B}_n|, 
% \end{aligned}
% \end{equation}
% which would be sufficient for our proof.

\subsubsection[Step IV: Reduction to the thin shell of F]{\textcolor{red}{Step IV}: Reduction to the thin shell of $\cF$}\label{sss:reduThinShell}
The following follows from the Fubini theorem and Lipschitz concentration:
\begin{lemma}\label{Lem:Energy}
Let $G$ be a $  (n-1)\times n$ Gaussian matrix with $N(0,1)$ i.i.d. entries, and let $\delta \in (0,1)$. Then, with probability \textbf{ over $G$} of at least $1-\exp(-c_1n\delta^2)$, the following deterministic equation holds:
\[
  \Pr_{X \sim U(\triangle_c)}( \|GX\|_2 \leq 1 + \delta/2 \cap \|\sqrt{n} X\|_2 \geq 1+\delta ) \lesssim \exp(-cn\delta^2). 
\]
\end{lemma}
The proof of this lemma appears below. This lemma means that most of the thin shell of the $\cF$ will emerge from thin shell $\triangle_c$,
% This lemma implies that on the projected simplex $\cF_{c}:= G\triangle_{c} - \vec{C}_{\cF}$, most the volume lies on its thin shell,
and \textbf{only}  $\exp(-cn\delta^2)$ of the thin shell volume emerges from the areas that are $\delta$-far from the thin shell of $\triangle$. Namely, we have the desired sub-Gaussian tail, and note that by \Cref{C:GoodFlu2}, this holds for the distribution of Fleury's facet $\vec{\widetilde{Y}}$.

In the next steps, we prove the following: 
\begin{lemma}
    With probability of at least $1-\exp(-C \cdot nL^{-2})$ over $\vec \xi$,
    \[
        \|\vec \xi - \|\vec \xi\|_{n}\vc_{\cF}\|_{2} = (1+O(L^{-1})) \cdot M_{n,d}      
    \]
and $\|\vec \xi\|_{n} \cdot \cF$ contains $\vec \xi$.
    % and
    % \[
    %      \| n_{\vec{\cF}}\|_2 = \lp 1 - M_{n,d}^2 O\left(\frac{1}{L^2}\right) \rp \lp 1- 2 \rp \cdot \sqrt{n}
    % \]
    % and
    % \[
    %     \| \bar{n}_{\vec{\cF}} - \bar{c}_{\vec{\cF}}\|_2 \leq (1 + L^{-1}) \cdot M_{n,d}
    % \]
\end{lemma}
Therefore, the theorem follows,  by Lemma \ref{Lem:Energy}, it holds that with probability of $1-\exp(-c_2 nL^{-2})$ that
\[
    \|\erm - (1+O(1/L) \cdot \frac{\|\vec \xi\|_{n}\vec{1}_{n}}{n}\|_{2}^2 \leq \frac{\|\vec \xi\|_n^2}{n} = \frac{(1+o(1))}{2L}
\]
And by Pythagoras, the claim follows, as 
\[
\erm - \|\vec \xi\|_{n}\vec{1}_n/n \perp \|\vec \xi\|_{n}\vec{1}_n/n 
\]
and $\|\vec \xi\|_{n} \approx M_{n,d}$, and therefore $\|\|\vec \xi\|_{n}\vec{1}_n/n \|_2^2 = (1+o(1))/2L$.

\subsubsection[Step V: On the height of the facets of Pnd]{{\color{red} Step V:} On the height of the facets  of $\PN$}
In this part, we will need to study the heights of the facets of the random polytope $\PN$, and relate the mode of $T_{n,d}$, which we denote by $t_{n,d}$,  to $M_{n,d}$. The following lemmas follow from Fleury's distribution; their proofs appear below. We use $L := \log(d/n)$
\begin{lemma}[Volume bound from facet heights]\label{lem:volume-from-facet-heights}
Let $t_{n,d}$ denote the mode of the density of $T_{n,d}$. Then, for a
universal constant $C>0$,
\[
 \frac{\E|\PN|}{|B_n|}
 \leq
 \exp\left(\frac{C\sqrt{n}}{\log(d/n)}\right)
 \big(t_{n,d}^2+2\big)^{n/2}.
\]
\end{lemma}
% \begin{proof}
% We remind that  $t_{n,d} = (1+o_{d/n}(1)) \cdot \sqrt{2\log(d/n)}$ and with probability of $0.9$
% \[
%     |\frac{t_{n,d}}{T_{n,d}} - 1| \lesssim \frac{1}{\sqrt{n}\log(d/n)}
% \]
% is the distribution $\|\vec{n}_{\cF}\|_2$, Recall that when a facet $\cF$  is drawn from Fleury's distribution (see Lemma \ref{Lem:Fleury}) it satisfies the following with probability of $1-\exp(-cn\eps^2)$:
% \begin{equation}
%     \|\vc_{\cF} - \vn_{\cF}\|_{2} \lesssim \eps
% \end{equation}
% Consider the linear functional defined via $\langle \theta,\cdot \rangle$, where $\theta:=\frac{\vc_{\cF} - \vec n_{\cF}}{\|\vc_{\cF} - \vec n_{\cF}\|_{2}}$, and the isotropic constant of $\cF$ implies that most of the volume (surface area) of $\cF$ lies in a ``thin slab'' in the direction of $\theta$. Formally, there exists and event $\cE_3$ with probability of $1-\exp(Cn\lln(d/n))$, such that for all facets, it holds 
% \begin{equation}\label{Eq:FF}
%     \Pr_{Z \sim U(\cF)}(|\langle Z,\theta \rangle| \lesssim n^{-1/2}) \geq 0.9
% \end{equation}
% and we also recall the thin-shell property
% \begin{equation}
%     \Pr_{Z \sim U(\cF)}(|\|Z\|_{2} - 1| \lesssim n^{-1/2}) \geq 0.9.
% \end{equation}
% Under the event of the last three equations (with $\eps = O(1/\sqrt{n})$) and \Cref{prop:tail-of-Tnd}, we apply Pythagoras's law, and obtain that
% \[
%    \Pr_{Z \sim U(\cF)}(\|z\|_2^2 \in (t_{n,d}^2  + 2 + O(1/(\sqrt{n}\log(d/n)))) \geq 0.9 
% \]
% and apply Fluery's distribution argument and the claim follows.
% \end{proof}
\begin{lemma}[Low-height cones]\label{lem:low-height-cones}
 For a facet
$\cF\in\cF_{n-1}(\PN)$, let
\[
    h_{\cF}:=\|\vn_{\cF}\|_2
\]
be its height, where $\vn_{\cF}$ is the point of minimal Euclidean norm in
$\operatorname{aff}(\cF)$. For $\varepsilon\geq 0$, define the low-height part
of $\PN$ by
\[
    \PN^{\varepsilon}
    :=
    \bigcup_{\substack{\cF\in\cF_{n-1}(\PN):\\
    h_{\cF}\leq (1-\varepsilon)t_{n,d}}}
    \operatorname{conv}(0,\cF).
\]
Assume that
\[
    0\leq \varepsilon\leq \frac{c_0}{\log(d/n)},
\]
where $c_0>0$ is the constant from \Cref{prop:tail-of-Tnd}. Then there is a
universal constant $c>0$ such that
\[
    \E |\PN^{\varepsilon}|
    \leq
    \exp\{-c n\log(d/n)^2\varepsilon^2\}\,\E|\PN|.
\]
Consequently, with probability at least
\[
    1-\exp\{-c n\log(d/n)^2\varepsilon^2/2\},
\]
one has
\[
    |\PN^{\varepsilon}|
    \leq
    \exp\{-c n\log(d/n)^2\varepsilon^2/2\}\,\E|\PN|.
\]
\end{lemma}
\begin{lemma}[Pure-noise MNI selects a typical-height facet]
\label{lem:pure-noise-mni-typical-height}
Let \(\vec \xi\) be independent of \(\vec X\)
and assume that \(\|\vec \xi\|_2=\sqrt n\).
Let \(\cF_{\vec \xi}\) be the unique facet of \(\PN\) hit by the MNI ray,
that is
\[
    \frac{\vec \xi}{\|\vec \xi\|_{\PN}}
    \in
    \operatorname{relint}(\cF_{\vec \xi}).
\]
Then there exist universal constants \(c,C>0\) such that
\[
    \Pr_{\vec X,\vec \xi}
    \left\{
        \left|
        \|\vn_{\cF_{\vec \xi}}\|_2
        -
        t_{n,d}
        \right|
        >
        C L^{-5/4}t_{n,d}
    \right\}
    \le
    \exp\left(
        -c\frac{n}{\sqrt L}
    \right).
\]
\end{lemma}

\begin{proof}
Set
\[
    a_{n,d}:=\frac{n}{\sqrt L},
    \qquad
    \eps_0:=K L^{-5/4},
\]
where \(K>0\) is a sufficiently large universal constant to be fixed later.

Define the bad-height event on facets by
\[
    A_{\eps_0}(\cF)
    :=
    \left\{
        \left|
        \|\vn_{\cF}\|_2
        -
        t_{n,d}
        \right|
        >
        \eps_0 t_{n,d}
    \right\}.
\]
Let \(K_{\eps_0}(\vec X)\subset\SPO\) be the union of the cones over the
scaled bad facets:
\[
    K_{\eps_0}(\vec X)
    :=
    \bigcup_{\cF:\,A_{\eps_0}(\cF)}
    \operatorname{conv}\bigl(0,M_{n,d}\cF\bigr).
\]
The cones over the facets decompose \(\SPO\), up to overlaps of measure zero.

We first estimate the expected volume of \(K_{\eps_0}\). By the conic
decomposition,
\[
    |\operatorname{conv}(0,M_{n,d}\cF)|
    =
    \frac{M_{n,d}^n}{n}\,
    \|\vn_{\cF}\|_2\,|\cF|.
\]
Therefore, using exchangeability of signed \(n\)-tuples and Fleury's
conditional facet distribution,
\[
    \frac{\E |K_{\eps_0}|}{\E|\SPO|}
    =
    \frac{
        \E\left[
            T_{n,d}\,
            \mathbf 1_{\{|T_{n,d}-t_{n,d}|>\eps_0t_{n,d}\}}
        \right]
    }{
        \E T_{n,d}
    }.
    \tag{3}
\]
Here we used that, under Fleury's distribution, the height is distributed as
\(T_{n,d}\), and the tangential volume of the facet is independent of
\(T_{n,d}\).

By the tail estimate for \(T_{n,d}\),
\[
    \Pr
    \left\{
        |T_{n,d}-t_{n,d}|
        >
        u t_{n,d}
    \right\}
    \le
    2\exp\{-c n t_{n,d}^{4}u^2\}
\]
for \(0\le u\le c_0/L\). Since
\[
    t_{n,d}^2\asymp L,
\]
and since \(\eps_0=K L^{-5/4}\le c_0/L\) for \(L\) large, we get
\[
    n t_{n,d}^4\eps_0^2
    \asymp
    nL^2\cdot K^2L^{-5/2}
    =
    K^2\frac{n}{\sqrt L}
    =
    K^2a_{n,d}.
\]
Integrating the same tail bound, and using \(\E T_{n,d}\asymp t_{n,d}\),
equation \((3)\) gives
\[
    \E |K_{\eps_0}|
    \le
    \exp(-cK^2a_{n,d})\,\E|\SPO|.
    \tag{4}
\]

By Markov's inequality, with probability at least
\[
    1-\exp(-cK^2a_{n,d}/2),
\]
we have
\[
    |K_{\eps_0}|
    \le
    \exp(-cK^2a_{n,d}/2)\,\E|\SPO|.
    \tag{5}
\]
Using the Step I volume estimate
\[
    \E|\SPO|
    \le
    \exp(Ca_{n,d})|\bar B_n|,
\]
we obtain, on the event \((5)\),
\[
    |K_{\eps_0}|
    \le
    \exp\bigl(-(cK^2/2-C)a_{n,d}\bigr)|\bar B_n|.
    \tag{6}
\]

Now we convert this volume estimate into a statement about the probability
that the pure-noise MNI ray hits a bad facet. For fixed \(\vec X\), define
the set of bad directions
\[
    D_{\eps_0}(\vec X)
    :=
    \left\{
        \theta\in S^{n-1}:
        \text{ the ray }\mathbb R_+\theta
        \text{ intersects }K_{\eps_0}(\vec X)
    \right\}.
\]
Since \(\ft\equiv0\), the direction of \(\vec \xi\) is uniform on
\(S^{n-1}\), and therefore
\[
    \Pr_{\vec \xi}
    \left\{
        A_{\eps_0}(\cF_{\vec \xi})
        \mid \vec X
    \right\}
    =
    \sigma_{n-1}(D_{\eps_0}(\vec X)).
    \tag{7}
\]

On the Step I radial event \((2)\),
\[
    \bar B_n
    \subset
    \left(1+\frac{C}{\sqrt L}\right)\SPO .
\]
Equivalently, for every direction \(\theta\in S^{n-1}\), the radial function
of \(\SPO\) satisfies
\[
    r_{\SPO}(\theta)
    \ge
    \frac{\sqrt n}{1+C/\sqrt L}
    \ge
    \left(1-\frac{C}{\sqrt L}\right)\sqrt n .
    \tag{8}
\]
Therefore, for every \(\theta\in D_{\eps_0}(\vec X)\), the cone
\(K_{\eps_0}\) contains the radial segment of length
\[
    r_0:=
    \left(1-\frac{C}{\sqrt L}\right)\sqrt n
\]
in direction \(\theta\). Hence
\[
    K_{\eps_0}
    \supset
    r_0 B_2^n\cap
    \operatorname{cone}(D_{\eps_0}(\vec X)).
\]
Taking volumes gives
\[
    |K_{\eps_0}|
    \ge
    \left(1-\frac{C}{\sqrt L}\right)^n
    |\bar B_n|\,
    \sigma_{n-1}(D_{\eps_0}(\vec X)).
\]
Since
\[
    \left(1-\frac{C}{\sqrt L}\right)^n
    \ge
    \exp\left(-C\frac{n}{\sqrt L}\right)
    =
    \exp(-Ca_{n,d}),
\]
we obtain
\[
    \sigma_{n-1}(D_{\eps_0}(\vec X))
    \le
    \exp(Ca_{n,d})
    \frac{|K_{\eps_0}|}{|\bar B_n|}.
    \tag{9}
\]
Combining \((6)\) and \((9)\), we get
\[
    \sigma_{n-1}(D_{\eps_0}(\vec X))
    \le
    \exp\bigl(-(cK^2/2-C)a_{n,d}\bigr).
\]
Choosing \(K>0\) sufficiently large yields
\[
    \sigma_{n-1}(D_{\eps_0}(\vec X))
    \le
    \exp(-c a_{n,d}).
    \tag{10}
\]

Finally, we combine the exceptional probabilities: the failure of the Step I
volume/radial event has probability at most \(\exp(-ca_{n,d})\), the failure
of \((5)\) has probability at most \(\exp(-cK^2a_{n,d})\), and on the
intersection of the good events, \((7)\) and \((10)\) give
\[
    \Pr_{\vec \xi}
    \left\{
        A_{\eps_0}(\cF_{\vec \xi})
        \mid \vec X
    \right\}
    \le
    \exp(-c a_{n,d}).
\]
Therefore,
\[
    \Pr_{\vec X,\vec \xi}
    \left\{
        A_{\eps_0}(\cF_{\vec \xi})
    \right\}
    \le
    \exp(-c a_{n,d})
    =
    \exp\left(-c\frac{n}{\sqrt L}\right).
\]
Since
\[
    \eps_0
    =
    K L^{-5/4},
\]
this is exactly
\[
    \Pr_{\vec X,\vec \xi}
    \left\{
        \left|
        \|\vn_{\cF_{\vec \xi}}\|_2
        -
        t_{n,d}
        \right|
        >
        K L^{-5/4}t_{n,d}
    \right\}
    \le
    \exp\left(
        -c\frac{n}{\sqrt L}
    \right).
\]
Renaming \(K\) as \(C\) proves the lemma.
\end{proof}
Now, if we knew that $\vec \xi$ lies in the thin shell of the facet, then we could estimate $t_{n,d}$, and then we would conclude by the last lemma and a simple volume estimate that
\[
    \E |\SPO| \lesssim \exp(Cn \cdot \log(d/n)^{-5/4})|\bar{B}_n|.
\]
which is a better estimate than Step I. Therefore, we would prove this claim
\begin{lemma}[Volume upgrade from the facet thin shell]
\label{lem:volume-upgrade-L-five-fourths}
\[
    \E|\SPO|
    \le
    \exp\left(CnL^{-5/4}\right)|\bar B_n|.
\]
\end{lemma}

\begin{proof}
Set
\[
    \eta:=L^{-1/4},
    \qquad
    \delta:=L^{-5/4}.
\]
Recall that
\[
    M^2\asymp \frac{n}{L},
    \qquad
    t_{n,d}^2\asymp L.
\]
Hence
\[
    \eta M^2
    =
    L^{-1/4}M^2
    \asymp
    nL^{-5/4}
    =
    n\delta .
    \tag{1}
\]

Let \(\bar{\cF}:=M\cF\) be a scaled facet of \(\SPO\), and write
\[
    \bar n_{\cF}:=M n_{\cF},
    \qquad
    \bar c_{\cF}:=M c_{\cF}.
\]
For \(x\in \bar{\cF}\), decompose
\[
    x
    =
    \bar n_{\cF}
    +
    (\bar c_{\cF}-\bar n_{\cF})
    +
    (x-\bar c_{\cF}).
\]
Since \(n_{\cF}\) is the Euclidean projection of \(0\) onto
\(\operatorname{aff}(\cF)\), we have
\[
    c_{\cF}-n_{\cF}\perp n_{\cF},
    \qquad
    z-c_{\cF}\perp n_{\cF}
\]
for every \(z\in\cF\). Therefore,
\[
\begin{aligned}
    \|x\|_2^2
    &=
    \|\bar n_{\cF}\|_2^2
    +
    \|\bar c_{\cF}-\bar n_{\cF}\|_2^2
    +
    \|x-\bar c_{\cF}\|_2^2
    +
    2\langle
        \bar c_{\cF}-\bar n_{\cF},
        x-\bar c_{\cF}
    \rangle .
\end{aligned}
\tag{2}
\]

We now remove an exceptional set of cone volume. By the previous height
corollary, the union of cones over facets for which
\[
    \left|
    \|\bar n_{\cF}\|_2^2-\operatorname{Med}\bar T_{n,d}^2
    \right|
    >
    Cn\delta
\]
has expected volume at most
\[
    \exp\left(-c\frac{n}{\sqrt L}\right)\E|\SPO|.
    \tag{3}
\]
Here
\[
    \bar T_{n,d}:=M T_{n,d}.
\]
Also, from the preceding matching estimate,
\[
    \operatorname{Med}\bar T_{n,d}^2+2M^2
    \le
    n+Cn\delta .
    \tag{4}
\]

Next we remove the bad part of the simplex inside the remaining facets.
Fleury's representation gives, after an independent rotation,
\[
    \bar c_{\cF}-\bar n_{\cF}
    =
    M\bar Y,
    \qquad
    x-\bar c_{\cF}
    =
    M Yv,
    \qquad
    v\in\Delta_c.
\]
Moreover,
\[
    \theta_{\cF}
    :=
    \frac{\bar c_{\cF}-\bar n_{\cF}}
    {\|\bar c_{\cF}-\bar n_{\cF}\|_2}
\]
is rotationally invariant and independent of the centered simplex part.

Using the small-ball estimate for the centered simplex, together with
\(\Cref{lem:david-alonso}\) applied to the truncated simplex of diameter
\(O(\sqrt n)\), we may discard another set of expected cone volume at most
\[
    \exp\left(-c\frac{n}{\sqrt L}\right)\E|\SPO|
    \tag{5}
\]
so that, on the remaining part of every good facet,
\[
    \|\bar c_{\cF}-\bar n_{\cF}\|_2
    =
    (1+O(\eta))M,
    \tag{6}
\]
\[
    \|x-\bar c_{\cF}\|_2
    =
    (1+O(\eta))M,
    \tag{7}
\]
and
\[
    \left|
    \langle
        \bar c_{\cF}-\bar n_{\cF},
        x-\bar c_{\cF}
    \rangle
    \right|
    \le
    C\eta M^2 .
    \tag{8}
\]
The crucial point is that the direction
$
    \theta_{\cF}
$
is uniform and independent of the centered simplex. Thus
\Cref{lem:david-alonso} gives the marginal bound needed for the mixed
term. The part removed by the truncation of the simplex also contributes
only the amount in \((5)\).

Combining \((6)\), \((7)\), and \((8)\), we obtain
\[
\begin{aligned}
    &\|\bar c_{\cF}-\bar n_{\cF}\|_2^2
    +
    \|x-\bar c_{\cF}\|_2^2
    +
    2\langle
        \bar c_{\cF}-\bar n_{\cF},
        x-\bar c_{\cF}
    \rangle       \\
    &\qquad\le
    2M^2+C\eta M^2 .
\end{aligned}
\tag{9}
\]
By \((1)\),
\[
    C\eta M^2
    \le
    Cn\delta .
    \tag{10}
\]

Now take a point \(x\) in the non-exceptional part of a good facet.
Using \((2)\), \((4)\), \((9)\), and \((10)\), we get
\[
\begin{aligned}
    \|x\|_2^2
    &\le
    \|\bar n_{\cF}\|_2^2
    +
    2M^2
    +
    Cn\delta                                      \\
    &\le
    \operatorname{Med}\bar T_{n,d}^2
    +
    2M^2
    +
    Cn\delta                                      \\
    &\le
    n+Cn\delta .
\end{aligned}
\]
Therefore
\[
    \|x\|_2
    \le
    (1+C\delta)\sqrt n.
    \tag{11}
\]
Thus the non-exceptional part of the cone decomposition of \(\SPO\) is
contained in
\[
    (1+C\delta)\bar B_n.
\]

Let \(K_{\mathrm{bad}}\subset \SPO\) be the union of all exceptional cones:
the cones over bad-height facets and the bad simplex pieces inside the
remaining facets. From \((3)\) and \((5)\),
\[
    \E|K_{\mathrm{bad}}|
    \le
    \exp\left(-c\frac{n}{\sqrt L}\right)\E|\SPO|.
    \tag{12}
\]
Since the good part is contained in \((1+C\delta)\bar B_n\), we have
\[
    |\SPO|
    \le
    |(1+C\delta)\bar B_n|
    +
    |K_{\mathrm{bad}}|.
\]
Taking expectations and using \((12)\),
\[
    \E|\SPO|
    \le
    (1+C\delta)^n|\bar B_n|
    +
    \exp\left(-c\frac{n}{\sqrt L}\right)\E|\SPO|.
\]
For \(L\) large, the exponential factor in the second term is smaller than
\(1/2\). Hence
\[
    \E|\SPO|
    \le
    2(1+C\delta)^n|\bar B_n|.
\]
Finally,
\[
    (1+C\delta)^n
    \le
    \exp(Cn\delta)
    =
    \exp\left(CnL^{-5/4}\right).
\]
Absorbing the factor \(2\) into the exponential gives
\[
    \E|\SPO|
    \le
    \exp\left(CnL^{-5/4}\right)|\bar B_n|.
\]
This proves the claim.
\end{proof}
Now, as we improve Step I, we can bootstrap this step, to obtain the following corollary:
\begin{corollary}
\label{cor:terminal-volume-bootstrap}
Assume that \(\eps_{n,d}\) is in the stopping regime
\[
    \eps_{n,d}\gtrsim L^{-2}
\]
and
\[
    nL^2\eps_{n,d}^2
    \gtrsim
    n\eps_{n,d}+\sqrt{\frac nL}.
    \tag{1}
\]
Then
\[
    \exp\!\left(-C nL^2\eps_{n,d}^2\right)\E|\SPO|
    \le
    \exp\!\left(-c\sqrt{\frac nL}\right)|\bar B_n|
    \le
    \E|\SPO|
    \le
    \exp\!\left(Cn\eps_{n,d}\right)|\bar B_n|.
\]
\end{corollary}

\begin{proof}
We use the volume bootstrap from the preceding step. Namely, if for some
\(\alpha\in(0,1)\) one has
\[
    \E|\SPO|
    \le
    \exp(Cn\alpha)|\bar B_n|,
    \tag{2}
\]
then the height-selection estimate, the truncated-simplex estimate, and
Pythagoras imply the improved bound
\[
    \E|\SPO|
    \le
    \exp\left(
        Cn\left(
            \eps_{n,d}+\frac{\sqrt\alpha}{L}
        \right)
    \right)|\bar B_n|.
    \tag{3}
\]
Indeed, the bad-height facets contribute at most
\[
    \exp(-c nL^2\rho^2)\E|\SPO|,
\]
and choosing
\[
    \rho\asymp \frac{\sqrt\alpha}{L}
\]
makes this exceptional contribution negligible relative to the current
volume bound \((2)\). On the remaining facets, Fleury's representation and
\Cref{lem:david-alonso} give a tangential shell error
\[
    \eta\asymp \sqrt\alpha .
\]
Since
\[
    M_{n,d}^2\asymp \frac nL,
\]
this contributes
\[
    \eta M_{n,d}^2
    \asymp
    \frac{n\sqrt\alpha}{L}
\]
to the squared-radius estimate. Hence the good part of \(\SPO\) lies inside
\[
    \left(
        1+
        C\eps_{n,d}
        +
        C\frac{\sqrt\alpha}{L}
    \right)\bar B_n,
\]
which gives \((3)\).

Starting from the Step I bound
\[
    \E|\SPO|
    \le
    \exp\left(CL^{-1/2}n\right)|\bar B_n|,
\]
we iterate the map
\[
    \alpha\longmapsto
    C\left(
        \eps_{n,d}+\frac{\sqrt\alpha}{L}
    \right).
\]
Without the stopping term, this sends
\[
    L^{-1/2}
    \longmapsto
    L^{-5/4}
    \longmapsto
    L^{-13/8}
    \longmapsto
    \cdots
    \longrightarrow
    L^{-2}.
\]
Thus the terminal exponent is
\[
    \alpha_\infty
    \lesssim
    \eps_{n,d}+L^{-2}.
\]
Since we assume \(\eps_{n,d}\gtrsim L^{-2}\), we obtain
\[
    \E|\SPO|
    \le
    \exp(Cn\eps_{n,d})|\bar B_n|.
    \tag{4}
\]

The lower bound
\[
    \exp\!\left(-c\sqrt{\frac nL}\right)|\bar B_n|
    \le
    \E|\SPO|
    \tag{5}
\]
is the lower-volume estimate from Step I.

Finally, using \((4)\),
\[
\begin{aligned}
    \exp(-C nL^2\eps_{n,d}^2)\E|\SPO|
    &\le
    \exp\left(
        -C nL^2\eps_{n,d}^2
        +
        Cn\eps_{n,d}
    \right)|\bar B_n|  \\
    &\le
    \exp\!\left(-c\sqrt{\frac nL}\right)|\bar B_n|,
\end{aligned}
\]
where the last inequality follows from the stopping condition \((1)\), after
choosing the constant \(C\) large enough. Combining this with \((5)\) and
\((4)\) gives the claimed sandwich.\qedhere
\end{proof}
\subsubsection{\textcolor{red}{Step V+}: Proof of \Cref{C:MND}}\label{ss:pfMND}

Recall that the iterative volume bootstrap gives
\[
    \E|\SPO|
    \le
    \exp\left(\frac{Cn}{\ln^2(d/n)}\right)|\bar B_n|.
\]
More importantly, the same bootstrap gives the sharp shell-matching estimate
\[
    \left|
        \frac{M_{n,d}^2\big(t_{n,d}^2+2\big)}{n}
        -1
    \right|
    \lesssim
    \frac{1}{\ln^2(d/n)}.
    \tag{1}
\]
Indeed, the argument stops precisely when the volume-radius error
\(n\varepsilon\) balances the Fleury height-tail exponent
\(n\ln^2(d/n)\varepsilon^2\), namely at
\[
    \varepsilon\asymp \ln^{-2}(d/n).
\]
Consequently,
\[
    M_{n,d}
    =
    \left(
        1+O\left(\ln^{-2}(d/n)\right)
    \right)
    \sqrt{
        \frac{n}{t_{n,d}^2+2}
    }.
    \tag{2}
\]

It remains to estimate \(t_{n,d}\). Set
$
    q:=\frac{d-n}{n}$
and recall that  Fleury height variable \(T_{n,d}\) has density
\[
    p(t)
    =
    \frac{1}{Z_{n,d}}\,
    \Psi(t)^{d-n}\exp(-nt^2/2),
    \qquad t>0,
\]
where
\[
    \Psi(t):=2\Phi(t)-1.
\]
Let \(t_{n,d}\) denote the mode of this density. Since
\(t\mapsto 2\varphi(t)/\Psi(t)\) is strictly decreasing on
\((0,\infty)\), the mode is the unique positive solution of
\[
    0
    =
    \frac{d}{dt}\log p(t)
    =
    (d-n)\frac{2\varphi(t)}{\Psi(t)}
    -
    nt.
\]
Equivalently,
\[
    t_{n,d}
    =
    q\,\frac{2\varphi(t_{n,d})}{\Psi(t_{n,d})}.
    \tag{3}
\]
Using
\[
    2\varphi(t)
    =
    \sqrt{\frac{2}{\pi}}\exp(-t^2/2),
\]
equation \((3)\) gives
\[
    t_{n,d}^2 e^{t_{n,d}^2}
    =
    \frac{2q^2}{\pi\,\Psi(t_{n,d})^2}.
    \tag{4}
\]

We next remove the harmless factor \(\Psi(t_{n,d})\). By Mills' bound and
\((3)\),
\[
    1-\Psi(t_{n,d})
    =
    2\overline\Phi(t_{n,d})
    \le
    \frac{2\varphi(t_{n,d})}{t_{n,d}}
    =
    \frac{\Psi(t_{n,d})}{q}.
\]
Hence
\[
    \Psi(t_{n,d})
    =
    1+O(q^{-1}).
\]
Applying the Lambert \(W\)-function to \((4)\), we obtain
\[
    t_{n,d}^2
    =
    W\!\left(
        \frac{2}{\pi}q^2
    \right)
    +
    O(q^{-1})
    =
    W\!\left(
        \frac{2}{\pi}
        \left(\frac{d-n}{n}\right)^2
    \right)
    +
    O\!\left(\frac{n}{d}\right).
    \tag{5}
\]

Now use the standard expansion
\[
    W(x)
    =
    \ln x
    -
    \ln\ln x
    +
    \frac{\ln\ln x}{\ln x}
    +
    O\left(
        \frac{(\ln\ln x)^2}{\ln^2 x}
    \right),
    \qquad x\to\infty.
\]
Since
\[
    \ln q
    =
    L+O(n/d)
    =
    L+O(n/d),
\]
and \(d\gtrsim n(\ln n)^C\), the \(O(n/d)\) term is absorbed into the
error below. Therefore
\[
    t_{n,d}^2
    =
    2L
    -
    \ln L
    -
    \ln\pi
    +
    \frac{\ln L+\ln\pi}{2L}
    +
    O\left(
        \frac{(\ln L)^2}{L^2}
    \right).
    \tag{6}
\]
Equivalently, using the notation \(\lln(d/n)= \ln L\),
\[
    t_{n,d}^2
    =
    2L
    -
    \lln(d/n)
    -
    \ln\pi
    +
    \frac{\lln(d/n)+\ln\pi}{2L}
    +
    O\left(
        \frac{(\lln(d/n))^2}{\ln^2(d/n)}
    \right).
    \tag{7}
\]
Taking square roots gives
\[
    t_{n,d}
    =
    \sqrt{2L}
    -
    \frac{\lln(d/n)+\ln\pi}
    {2\sqrt{2L}}
    +
    O\left(
        \frac{(\lln(d/n))^2}{\ln^{3/2}(d/n)}
    \right).
    \tag{8}
\]

Combining \((2)\) with \((7)\), we conclude that
\[
    M_{n,d}
    =
    \frac{
        \left(
            1+O\left(\ln^{-2}(d/n)\right)
        \right)\sqrt n
    }{
        \sqrt{
            2L
            -
            \lln(d/n)
            -
            \ln\pi
            +
            2
            +
            \frac{\lln(d/n)+\ln\pi}{2L}
            +
            O\left(
                \frac{(\lln(d/n))^2}{\ln^2(d/n)}
            \right)
        }
    }.
    \tag{9}
\]
In particular,
\[
    M_{n,d}^2
    =
    \frac{n}{
        2L
        -
        \lln(d/n)
        -
        \ln\pi
        +
        2
        +
        o(1)
    }.
\]
This proves \Cref{C:MND}.
\subsubsection[Step VI: Upper bound on the MSE]{{\color{red} Step VI:} Upper bound on the MSE}
Now, as explained in \Cref{sss:reduThinShell}, the previous estimates imply
that on the good event,
\[
    \|\bar z_{\xi}-\bar\vc_{\cF}\|_2^2
    =
    M_{n,d}^2
    +
    O\left(\frac{n}{\log^2(d/n)}\right),
    \tag{1}
\]
where
\[
    \bar z_{\xi}
    :=
    M_{n,d}\frac{\vec\xi}{\|\vec\xi\|_{\PN}}
    \in \bar\cF.
\]
Indeed, by Pythagoras,
\[
\begin{aligned}
    \|\bar z_{\xi}\|_2^2
    &=
    \|\bar\vn_{\cF}\|_2^2
    +
    \|\bar\vc_{\cF}-\bar\vn_{\cF}\|_2^2
    +
    \|\bar z_{\xi}-\bar\vc_{\cF}\|_2^2        \\
    &\qquad
    +
    2\left\langle
        \bar\vc_{\cF}-\bar\vn_{\cF},
        \bar z_{\xi}-\bar\vc_{\cF}
    \right\rangle .
\end{aligned}
\]
Using
\[
    \|\bar z_{\xi}\|_2^2
    =
    n+O\left(\frac{n}{\log^2(d/n)}\right),
\]
\[
    \|\bar\vn_{\cF}\|_2^2
    =
    \big(\Med \bar T_{n,d}\big)^2
    +
    O\left(\frac{n}{\log^2(d/n)}\right),
\]
\[
    \big(\Med \bar T_{n,d}\big)^2+2M_{n,d}^2
    =
    n+O\left(\frac{n}{\log^2(d/n)}\right),
\]
and
\[
    \|\bar\vc_{\cF}-\bar\vn_{\cF}\|_2^2
    =
    M_{n,d}^2
    +
    O\left(\frac{M_{n,d}^2}{\log(d/n)}\right),
\]
together with the marginal estimate
\[
    \left|
    \left\langle
        \bar\vc_{\cF}-\bar\vn_{\cF},
        \bar z_{\xi}-\bar\vc_{\cF}
    \right\rangle
    \right|
    \le
    \frac{C M_{n,d}^2}{\log(d/n)},
\]
we obtain \((1)\), since
\[
    \frac{M_{n,d}^2}{\log(d/n)}
    \asymp
    \frac{n}{\log^2(d/n)}.
\]

Since
\[
    \frac{\|\vec\xi\|_{\PN}}{M_{n,d}}
    =
    1+O\left(\frac{1}{\log^2(d/n)}\right),
\]
we also get
\[
\begin{aligned}
    \left\|
        \vec\xi-\|\vec\xi\|_{\PN}\vc_{\cF}
    \right\|_2^2
    &=
    \left(\frac{\|\vec\xi\|_{\PN}}{M_{n,d}}\right)^2
    \|\bar z_{\xi}-\bar\vc_{\cF}\|_2^2        \\
    &=
    M_{n,d}^2
    +
    O\left(\frac{n}{\log^2(d/n)}\right).
\end{aligned}
\tag{2}
\]

Now apply \Cref{Lem:Energy} on the active simplex. Writing
\[
    q:=\|\vec\xi\|_{\PN},
\]
we have
\[
    \erm
    =
    q\lambda,
    \qquad
    \lambda\in\Delta_{n-1}.
\]
Hence
\[
    \erm
    =
    \frac{q}{n}\mathbf 1_n
    +
    q\left(\lambda-\frac{\mathbf 1_n}{n}\right),
\]
and the two terms are orthogonal in \(\ell_2^n\). Therefore
\[
    \|\erm\|_2^2
    =
    \frac{q^2}{n}
    +
    q^2
    \left\|
        \lambda-\frac{\mathbf 1_n}{n}
    \right\|_2^2.
    \tag{3}
\]
By \Cref{Lem:Energy}, together with \((2)\),
\[
\begin{aligned}
    q^2
    \left\|
        \lambda-\frac{\mathbf 1_n}{n}
    \right\|_2^2
    &\le
    \frac{1+C/\log(d/n)}{n}
    \left\|
        \vec\xi-q\vc_{\cF}
    \right\|_2^2        \\
    &\le
    \frac{M_{n,d}^2}{n}
    +
    \frac{C}{\log^2(d/n)}.
\end{aligned}
\tag{4}
\]
Moreover, since \(q=M_{n,d}(1+O(\log^{-2}(d/n)))\),
\[
    \frac{q^2}{n}
    =
    \frac{M_{n,d}^2}{n}
    +
    O\left(
        \frac{M_{n,d}^2}{n\log^2(d/n)}
    \right)
    =
    \frac{M_{n,d}^2}{n}
    +
    O\left(
        \frac{1}{\log^3(d/n)}
    \right).
    \tag{5}
\]
Combining \((3)\), \((4)\), and \((5)\), we obtain
\[
    \|\erm\|_2^2
    \le
    \frac{2M_{n,d}^2}{n}
    +
    \frac{C}{\log^2(d/n)}.
\]
Equivalently, using \(M_{n,d}^2/n\asymp 1/\log(d/n)\),
\[
    \|\erm\|_2^2
    \le
    \left(
        1+\frac{C}{\log(d/n)}
    \right)
    \frac{2M_{n,d}^2}{n}.
\]
This gives the desired upper bound on the MSE.
\subsubsection{\textcolor{red}{Step VII:} Reduction to zero signal}
% As we would see in the end of this section, the same idea that works for $\ft = e_1$ would work for any $O(n/\ln(n)^{C})$-sparse regressor. 
For simplicity we prove this for $\ft = e_1$, same argument holds for $\|\ft\|_0 = n/\log(d/n)^{C}$. The following result follows from our previous steps:
\begin{theorem*}
Let $
    1\le \lambda\le cL.$ Fix a deterministic vector \(\vec \xi'\in\R^n\) with
$
    \|\vec \xi'\|_2=\sqrt n .
$
Then, with probability at least
\[
    1-\exp\left(-c\lambda^2nL^{-2}\right)
\]
over \(\vec X\), the \(\ell_1\)-MNI interpolating \(\vec \xi'\) satisfies
\[
    \|\erm(\vec X,\vec \xi')\|_2^2
    \le
    \left(1+\frac{C\lambda}{L}\right)
    \frac{2M_{n,d}^2}{n}.
\]
\end{theorem*}
 Formally, the main lemma of this part is the following (and its proof of this lemma appears below). 
\begin{lemma}[Number of facets seen by a microscopic perturbation]
\label{Lem:numberoffacets}
Fix a deterministic vector \(\vec \xi\in\sqrt n S^{n-1}\), independent of
\(\vec X\). For \(r\le n^{-2}\), define
\[
    \cF_{\vec \xi}(r)
    :=
    \left\{
        \cF\in\cF_{n-1}(\PN):
        \exists\,\vec \xi'\in\R^n,\ 
        \|\vec \xi'-\vec \xi\|_2\le r,
        \quad
        \frac{\vec \xi'}{\|\vec \xi'\|_{\PN}}
        \in \cF
    \right\}.
\]
Then
\[
    \E_{\vec X}|\cF_{\vec \xi}(r)|
    \le
    \exp\left(\frac{C n}{\log^2(d/n)}\right).
\]
\end{lemma}
Roughly speaking, it follows by direct computations, which would imply that
\[
    \E |\cF_{\vec \xi}| \approx \frac{\E|\PN|}{M_{n,d}^{-n}} \lesssim \exp(CnL^{-2})
\]

Therefore, by Markov's inequality, it holds that
\[
    |\cF_{\vec \xi}| \lesssim \exp(C_2nL^{-2})
\]
with probability of at least $1-\exp(-C_1nL^{-2})$. Now, we use the following lemma that almost follows from definition:
\begin{fact}[Linearity of the MNI on one conic cell]
\label{fact:mni-linearity-on-facet}
Let \(y_1,\ldots,y_m\in\R^n\setminus\{0\}\), and assume that their MNI rays
hit the same facet \(\cF\) of \(\PN\), namely
\[
    z_i:=
    \frac{y_i}{\|y_i\|_{\PN}}
    \in
    \operatorname{relint}(\cF),
    \qquad i=1,\ldots,m.
\]
Let
\[
    y=\sum_{i=1}^m a_i y_i,
    \qquad
    a_i\ge0,
    \qquad
    \sum_{i=1}^m a_i=1.
\]
Then \(y\) also hits the same facet \(\cF\). In particular, if the
MNI is unique, then
\[
    \erm(y)
    =
    \sum_{i=1}^m a_i\erm(y_i)
    \in
    \operatorname{Conv}\{\erm(y_1),\ldots,\erm(y_m)\}.
\]
\end{fact}
Now, we obtain the following useful corollary:
\begin{corollary}[Few facets are used on a low-complexity set]
\label{cor:few-facets-on-net}
Let \(K\subset 2\bar B_n\setminus 2^{-1}\bar B_n\) be deterministic and
independent of \(\vec X\). Assume that
\[
    \mathcal N_2(n^{-2},K)
    \le
    \exp\left(c_0nL^{-2}\right),
\]
where \(\mathcal N_2(\cdot,\cdot)\) denotes Euclidean covering number and
\(c_0>0\) is a sufficiently small universal constant. For fixed \(\vec X\), define
\[
    \mathfrak F_K(\vec X)
    :=
    \left\{
        \cF\in\cF_{n-1}(\PN):
        \exists\,\xi\in K
        \text{ such that }
        \frac{\xi}{\|\xi\|_{\PN}}\in \cF
    \right\}.
\]
Then, with probability at least
 $
    1-\exp\left(-cnL^{-2}\right),
$
one has
\[
    |\mathfrak F_K(\vec X)|
    \le
    \exp\left(CnL^{-2}\right).
\]
\end{corollary}
To complete the proof of our theorem, we need to show that
\[
    \|(\erm)_{\cS^{c}}\|^2 \lesssim 1/L,
\]
where $\cS^{c} = [d] \setminus  \mathrm{Supp}\{\ft\}$. Also recall that
\[
    (\erm)_{\cS^{c}} := \argmin_{\vec X w  = \vec Y - \vec X(\erm)_{\cS}}\|w\|_1.    
\]
Now, if $\vec Y - \vec X(\erm)_{\cS}$ were independent of  $(\erm)_{\cS^{c}}$, then we would be done by combining the last two results. However, this is not the case. And note that one cannot use uniform convexity arguments (as \cite{kur2026minimum}) on the $\ell_1$-MNI, as it does not satisfy this property. Therefore, we have to do this by ``controlling the number of facets of the MNI'', and using the local tent property. Namely, take an $n{-1}$-net of perturbations to $\vec \xi$ apply the last step to each of them, as they cannot use too many facets, we can apply the local convexity and extend beyond the net.  For completeness, see \Cref{sss:RedTech}.

%% file: Missing.tex
\subsection{Missing parts from the proof of \Cref{Theorem:LoNE}}
\begin{proof}[Proof of Lemma \ref{Lem:SBSimplex}]
Let $Z_1,\ldots,Z_n \overset{\text{i.i.d.}}{\sim}\mathrm{Exp}(1)$, $S=\sum_{i=1}^n Z_i$, and define
\[
X=\frac{Z}{S}\in \simplex, 
\qquad 
Y=\sqrt{n(n+1)}\Bigl(X-\tfrac{\1}{n}\Bigr).
\]
Then $X$ is uniform on the simplex, and the rescaled $Y$ is isotropic in $\R^{n}$:
\[
\E Y=0, \qquad \E\lVert Y\rVert_2^2 = n-1.
\]
We want to bound, for $\delta\in(0,1)$,
\[
p(\delta)=\Pbb\!\left(\lVert Y\rVert_2^2 \le (1-\delta)(n-1)\right).
\]

We have
\[
\lVert Y\rVert_2^2 
= \frac{n+1}{\Zbar^{2}}\left(\frac{1}{n}\sum_{i=1}^n Z_i^2 - \Zbar^{2}\right),
\qquad 
\Zbar=\frac{S}{n}.
\]
Equivalently,
\[
\lVert Y\rVert_2^2 = (n+1)\left(\frac{U}{\Zbar^{2}}-1\right),
\quad\text{where}\quad
U=\frac{1}{n}\sum_{i=1}^n Z_i^2.
\]
Thus, for $\delta\in(0,1)$,
\begin{align*}
\Pbb\!\left(\lVert Y\rVert_2^2 \le (1-\delta)(n-1)\right)
&= \Pbb\!\left(\frac{U}{\Zbar^2} \le 1 + \frac{(1-\delta)(n-1)}{n+1}\right) \\
&\le \Pbb\!\left(\frac{U}{\Zbar^2} \le 2 - \delta\right).
\end{align*}

We control the numerator and denominator separately. Set the event
\[
F := \bigl\{\Zbar \le 1 + \delta/8\bigr\}.
\]
On $F$ we have
\[
\Pbb\!\left(\frac{U}{\Zbar^2} \le 2 - \delta \;\middle|\; F\right)
\le \Pbb\!\left(U \le \frac{(1+\delta/8)^2}{2-\delta}\right)
\le \Pbb\!\left(U \le 2 - \frac{\delta}{4}\right).
\]
Thus we conclude that
\begin{equation}
p(\delta) \le \Pbb(F^{\mathrm c}) + \Pbb\!\left(U \le 2 - \frac{\delta}{4}\right).
\label{eq:master}
\end{equation}

\paragraph{Bounding $\Pbb(F^{\mathrm c})$.}
For $Z\sim\mathrm{Exp}(1)$, its mgf is
\[
M_Z(\lambda) = \E e^{\lambda Z} = \frac{1}{1-\lambda}, \qquad \lambda<1.
\]
Recall $\Zbar = \frac{1}{n}\sum_{i=1}^n Z_i$. Then, for $t>0$,
\[
\Pbb(\Zbar \ge 1 + t) \le \exp\!\bigl(-n\,[\,t - \ln(1+t)\,]\bigr).
\]
In particular, for $0<t\le 1$,
\[
\Pbb(\Zbar \ge 1 + t) \le \exp\!\left(-\frac{n t^2}{6}\right),
\]
where we used $t - \ln(1+t) \ge t^2/6$ for $t\in[0,1]$. Thus,
\[
\Pbb(F^{\mathrm c}) = \Pbb\!\left(\Zbar \ge 1 + \frac{\delta}{8}\right)
\le \exp\!\left(-\frac{n \delta^2}{384}\right).
\]

\paragraph{Small ball for $U$.}
The moments of the exponential are $\E[Z^k]=\Gamma(1+k)=k!$.
Let $W = Z^2$. For any $\lambda>0$, apply the exponential trick with exponent $\lambda n$:
\begin{align*}
\Pbb\!\left(U \le 2 - \frac{\delta}{4}\right)
&= \Pbb\!\left(e^{-\lambda n U} \ge e^{-\lambda n (2-\delta/4)}\right) \\
&\le e^{\lambda n (2-\delta/4)}\, \E\!\left[e^{-\lambda \sum_{i=1}^n W_i}\right]
= e^{2\lambda n}\, e^{-\lambda \delta n/4}\, \bigl(\E e^{-\lambda W}\bigr)^n.
\end{align*}
Using $1-x \le e^{-x} \le 1-x + \tfrac{x^2}{2}$ for $x\ge 0$, we get
\[
\bigl(\E e^{-\lambda W}\bigr)^n
\le \left(1 - \lambda \E W + \frac{\lambda^2}{2}\E W^2\right)^n
\le \left(1 - 2\lambda + 12\lambda^2\right)^n
\le \exp\!\left(-2\lambda n + 12\lambda^2 n\right),
\]
since $\E W=\E[Z^2]=2$ and $\E[W^2]=\E[Z^4]=24$.
Therefore,
\[
\Pbb\!\left(U \le 2 - \frac{\delta}{4}\right)
\le \exp\!\left(12\lambda^2 n - \frac{1}{4}\lambda \delta n\right).
\]
Choosing $\lambda=\delta/96$ yields
\[
\Pbb\!\left(U \le 2 - \frac{\delta}{4}\right)
\le \exp\!\left(-\frac{\delta^2 n}{768}\right).
\]

\paragraph{Conclusion.}
Combining with \eqref{eq:master},
\[
p(\delta) \le \exp\!\left(-\frac{n\delta^2}{384}\right)
+ \exp\!\left(-\frac{\delta^2 n}{768}\right)
\le 2\,\exp\!\left(-\frac{\delta^2 n}{768}\right).
\]
Equivalently, for all $\delta>0$,
\[
\Pbb\!\left(\lVert Y\rVert_2^2 \le (1-\delta)\,n\right)
\le 2\,\exp\!\left(-\frac{\delta^2 n}{768}\right).\qedhere
\]
\end{proof}
\subsubsection{Proof of Lemma \ref{Lem:Energy}}
Without loss of generality, we consider the isotropic simplex scaled by $1/\sqrt{n}$, which we denote here by $\triangle$, and $G$ is a Gaussian matrix scaled by $1/\sqrt{n}$.  For some $\delta \in (0,1)$, we can see that 
\begin{align*}
           \E_{G} \mathrm{Unif}_{X \sim \Vol(\triangle)}( \|GX\|_{2} \leq 1 + \frac{\delta}{2} ,\|X\|_2 \geq 1) &= \int_{(\R^{n})^{n-1}}\int_{\triangle}1_{\|GX\|_{2} \leq 1 + \delta/2 ,\|X\|_2 \geq 1}dXdG \\&=\int_{\triangle}\int_{(\R^{n})^{n-1}}1_{\|GX\|_{2} \leq 1 + \delta/2 ,\|X\|_2 \geq 1}dGdx
           \\&=\int_{\triangle}\Pr_{G}(\|GX\|_{2} \leq 1 + \delta/2 ,\|X\|_2 \geq 1)dX
           \\&\leq \exp(-c_1n\delta^2),
\end{align*}
where we used Fubini's theorem and the fact that for any $x \in \Sn$, $\|Gx\|_2$ satisfies
\[
    \Pr_{G}\{\|Gx\|_2 \leq 1-\delta/2\} \leq \exp(-c\delta^2 n).
\]
 By Markov's inequality (on the matrix $G$), the proof is complete. 
\subsubsection{Proof of Lemma \ref{Lem:CheapVolumeEstimate}}

\paragraph{Claim~\ref{item:inclusion-via-K-n-d}:} 
We begin by controlling the (normalized) variance of the map $\vec \xi \mapsto \|\vec \xi\|_{K(n, d)^\circ}$, where $\vec \xi \sim \gamma_d \equiv \Normal{0}{I_d}$; we do this below in~\Cref{lem:control-normalized-variance-K-n-d}. Then, by the boosted form of Dvoretzky's theorem (\Cref{lem:boosted-dvoretzky}) applied to $K(n,d) \subset B^d_1$, we have
\begin{multline}
\label{ineq:lower-on-inradius}
r(P_{n,d}) \geq r(\vec X K(n,d)) \\\geq 
M^\ast(K(n,d)) \bigg[1 - c \sqrt{n \cdot \Var_{\gamma_d}\Big(\frac{\|\vec \xi\|_{K(n,d)^\circ}}{M^\ast(K(n,d))}\Big)}\bigg]
\geq 
M^\ast(K(n,d)) \Big(1 - \frac{C}{L}\Big).
\end{multline}
The first part of Lemma~\ref{Lem:CheapVolumeEstimate} now follows by an integration argument. Namely, let $\cE$ denote the event on which the inequality~\eqref{ineq:lower-on-inradius} holds. Then, for $d \geq 2n$, we have 
\[
M_{n,d} 
\leq \E \frac{1}{r(P_{n,d})} \1_{\cE} + 
\sqrt{\E_{\vec X, \vec \xi}[\|\vec \xi\|_{\vec X B^d_1}^2]} \sqrt{\P(\cE^c)} 
\leq 
\frac{(1+ \frac{C}{L})}{\MND} 
+ c\exp(-cn) \leq 
\frac{(1+ \frac{C'}{L})}{\MND}.
\]
where above we simply recognize $\MND = M^\ast(K(n,d))$.

\begin{lemma}
\label{lem:control-normalized-variance-K-n-d}
Suppose $d \geq 2n$. Then, it holds for $K(n,d) = B^d_1 \cap \tfrac{1}{n} B^d_\infty$ that
for $\vec \xi \sim \gamma_d \equiv N(0, I_d)$, 
\[
\Var_{\gamma_d}\Big(\frac{\|\vec \xi\|_{K(n,d)^\circ}}{M^\ast(K(n,d))}\Big) \lesssim 
\frac{1}{n \log^2(d/n)}.
\]
\end{lemma}
\begin{proof}
Let $f(\vec \xi) = \|\vec \xi\|_{K(n,d)^\circ}$; we also denote by $I_{n,d} \subset [d]$ the (random) subset of $n$ largest coordinates (by magnitude) of $\vec \xi$. It is easy to see that, with probability one, 
\[
f(\vec \xi) = \|\vec \xi\|_{K(n,d)^\circ} = 
\frac{1}{n} \sum_{i \in I_{n,d}} |\xi_i|.
\]
In particular, we have $\partial_i f = \tfrac{\sign(\vec \xi_i)}{n} \1\{i \in I_{n,d}\}$, for any $i \in [d]$. It follows that 
\[
\|\partial_i f \|_{L^2(\gamma_d)}^2 = \frac{1}{n^2} \cdot \frac{n}{d} = \frac{1}{nd}, \quad \mbox{while} \quad 
\|\partial_i f\|_{L^1(\gamma_d)} = \frac{1}{n}\cdot \frac{n}{d} = \frac{1}{d}. 
\]
Hence, by Talagrand's $L^1-L^2$ inequality, we have
\begin{subequations}
\begin{equation}
\label{eqn:variance-support-fn-K-n-d}
\Var_{\gamma_d}(f(\vec \xi)) \lesssim d \frac{\|\partial_1 f\|_{L^2(\gamma_d)}^2}{\log\Big(e \tfrac{\|\partial_1 f\|_{L^2(\gamma_d)}}{\|\partial_{1} f\|_{L^1(\gamma_d)}}\Big)} = \frac{1}{n \log(e \sqrt{d/n})} \lesssim \frac{1}{n \log(e d/n)}.
\end{equation}
On the other hand, writing $\vec \xi_i^\ast$ for the order statistics (sorted by decreasing magnitude), we have  
\begin{equation}
\label{eqn:M-star-K-n-d}
\E f(\vec \xi) = M^\ast(K(n,d)) = \frac{1}{n} \E\sum_{i=1}^n \vec \xi_i^\ast \asymp \sqrt{L}, 
\end{equation}
\end{subequations}
where the last relation follows from standard estimates for the order statistics of a Gaussian random vector (\emph{e.g.,} see~\cite[Lemma 3.1]{gordon2007gaussian}). 
Combining relations~\eqref{eqn:variance-support-fn-K-n-d} and~\eqref{eqn:M-star-K-n-d}, we obtain the result.\qedhere
\end{proof}

\paragraph{Claim~\ref{item:volume-via-K-n-d}:}
By integrating with polar coordinates and using Jensen's inequality, for any centrally symmetric convex body $K \subset \R^n$, it holds that
\[
\Big(\frac{|K|}{|B^n_2|}\Big)^{1/n} = \Big( \int_{\mathbb{S}^{n-1}} \|\theta\|_K^{-n} \, d\sigma_{n-1}(\theta)\Big)^{1/n} \geq \frac{1}{M(K)}.
\]
Therefore we have 
\[
M_{n,d} \geq (1-\exp(-cn)) \cdot \exp(-C(n,d)/n) \cdot \frac{1}{\widetilde{M}_{n,d}} \geq 
\Big(1 - \frac{c}{\sqrt{L}}\Big) \frac{1}{\widetilde{M}_{n,d}}.
\]

Define the event
\begin{multline*}
\cE_1 \defn
\bigg
\{\,\frac{\opnorm{\vec{X}_{S, \vec{\eps}}}}{\sqrt{n}} 
\leq 4 + \sqrt{2 \log \tfrac{e d}{n}}
\quad \\ 
\mbox{and} \quad 
\frac{\fronorm{\vec{X}_{S, \vec{\eps}}}}{n} \leq
1 + \tfrac{\sqrt{2 \log \tfrac{e d}{n}} +  2}{\sqrt{n}},~\mbox{for all}~S \subset [d], |S| = n, \vec{\eps} \in \{-1, 1\}^n\,\bigg\}.
\end{multline*}

\begin{lemma}
    The event $\cE_1$ holds with probability at least $1 - 2 \exp(-2n)$.
\end{lemma}
\begin{proof}
Fix $t > 0$. We claim that the following tail bounds hold:
\begin{subequations}
\label{ineq:union-bound-over-block-events}
\begin{align}
\label{ineq:consequence-of-operator-norm-concentration}
\Pr\Big\{ \exists S, \vec{\eps} : 
\opnorm{\vec{X}_{S, \vec{\eps}}} \geq (2 + \sqrt{2 \log \tfrac{e d}{n}} + t) \sqrt{n}\Big\} &\leq \exp(-nt^2/2), \quad \mbox{and,}
\\
\label{ineq:consequence-of-Gaussian-Lipschitz}
\Pr\Big\{ \exists S, \vec{\eps} : 
\frac{\fronorm{\vec{X}_{S, \vec{\eps}}}}{n} \geq
1 + \tfrac{\sqrt{2 \log \tfrac{e d}{n}} +  t}{\sqrt{n}} \Big\}
&\leq \exp(-nt^2/2).
\end{align}
\end{subequations}
To obtain the claimed bounds, 
we begin by noticing that for any fixed $n$-subset $S \subset [d], |S| =n$, it holds that 
\[
\opnorm{\vec{X}_{S, \vec{\eps}}} = \opnorm{\vec{X}_{S, \vec{\eps}'}} 
\quad \mbox{and} \quad 
\fronorm{\vec{X}_{S, \vec{\eps}}} = \fronorm{\vec{X}_{S, \vec{\eps}'}} \quad \mbox{for all}~\vec{\eps}, \vec{\eps}' \in \{-1, 1\}^n.
\]
Now, inequality~\eqref{ineq:consequence-of-operator-norm-concentration} follows
by the standard Davidson-Szarek tail bound for operator norms of Gaussian random matrices, and then
applying a union bound over all $\binom{d}{n} \leq (\tfrac{e d}{n})^n$ many $n$-subsets $S \subset [d]$. Inequality~\eqref{ineq:consequence-of-Gaussian-Lipschitz} follows by a union bound over all $n$-subsets applied to the Borell-TIS inequality; we also used Jensen's inequality to estimate $\E \fronorm{\vec{X}_{S, \eps}} \leq n$. Finally, taking $t = 2$ and applying a union bound over the events underlying inequalities~\eqref{ineq:union-bound-over-block-events} yields the claim.\qedhere
\end{proof}

\begin{proof}

% The left hand side, simply follows from Jensen's inequality, to see recall that by polar coorditnates.
% \[
%    |A| = \int_{\Sn} \|\vec \xi\|_{A}^{-n}d\sigma_n
% \]
% The key challenge is to prove the right hand side. 
% First, we need to consider the following set:

The boundary of the polytope $P_{n,d}$ can be decomposed facially as follows. For any $x \in \partial P_{n,d}$, there exists some facet $\mathcal{F} \in \mathcal{F}_{n-1}(P_{n,d})$ for which we can write 
\begin{equation}\label{eqn:facial-decomposition}
x = \vec{c}_\cF + (x - \vec{c}_{\cF}) = 
\vec{c}_\cF + \vec{X}_\cF z, \quad \mbox{for some}~z \in \Delta_c.
\end{equation}
Thus, by translation invariance of volume:
\begin{equation}
\label{ineq:lower-bound-over-blocks}
\min_{\cF \in \cF_{n-1}(P_{n,d})} \frac{|\cF \cap (\vec{c}_\cF + rB^n_2)|}{|\cF|} \geq 
\min_{\substack{S \subset [d] \\  |S| = n}} \min_{\vec{\eps} \in \{-1, 1\}^n} \, \frac{|\vec{X}_{S, \vec{\eps}} \Delta_c \cap r B^n_2|}{|\vec{X}_{S, \vec{\eps}} \Delta_c|}, \quad \mbox{for any}~r > 0.
\end{equation}
Using the fact that $z \mapsto \vec{X}_{S, \vec{\eps}} z$ is $\opnorm{\vec{X}_{S, \vec{\eps}}}$-Lipschitz, KLS for the centered simplex $\Delta_c$ gives 
\begin{equation}
\label{ineq:lower-bound-on-block-ratio-from-KLS}
\min_{S, \vec{\eps}} \frac{|\vec{X}_{S, \vec{\eps}} \Delta_c \cap r^{(1)}_{n,d}(\vec X) B^n_2|}{|\vec{X}_{S, \vec{\eps}} \Delta_c|}
\geq \frac{1}{2}, 
\quad \mbox{for,} \quad 
r^{(1)}_{n,d}(\vec{X}) \defn \max_{S, \vec{\eps}} \Big\{ \E_{z \sim \mathrm{Unif}(\Delta_c)} \|\vec{X}_{S, \eps} z\|_2
+ \frac{c_1}{n} \opnorm{\vec{X}_{S, \vec{\eps}}}\Big\}.
\end{equation}
Additionally define 
\begin{equation}
r^{(2)}_{n,d}(\vec X) \defn 
\max_{\cF \in \cF_{n-1}(P_{n,d})} \|\vec{c}_\cF\|_2, 
\quad \mbox{and} \quad 
r_{n,d}(\vec X) \defn 
r^{(1)}_{n,d}(\vec X) + 
r^{(2)}_{n,d}(\vec X).
\end{equation}
Note by the triangle inequality applied to display~\eqref{eqn:facial-decomposition}, we have 
\[
\partial P_{n,d} \cap r_{n,d}(\vec X)B^n_2 \supset \bigcup_{\cF \in \cF_{n-1}(P_{n,d})} \cF \cap (\vec{c}_\cF + r_{n,d}^{(1)}(\vec X)).
\]
Therefore, combining inequalities~\eqref{ineq:lower-bound-over-blocks} and~\eqref{ineq:lower-bound-on-block-ratio-from-KLS}, it holds that $|\partial P_{n,d} \cap r_{n,d}(\vec X) B^n_2| 
\geq \tfrac{1}{2}  |\partial P_{n,d}|$. Hence, it follows that 
\[
|P_{n,d}| \leq 
\frac{1}{n} \sum_{\cF \in \cF_{n-1}(P_{n,d})} \|\vec{n}_\cF\|_2 |\cF| 
\leq \frac{1}{n} R(P_{n,d})  |\partial P_{n,d}| 
\leq \frac{R(P_{n,d})}{n}
r_{n,d}(\vec X)^{n-1} 
\lesssim r_{n,d}(\vec X)^{n}.
\]
On the event $\cE_1$, using that $\Cov(\Delta_c) \preceq \frac{1}{n^2}I_n$, we have
\[
r^{(1)}_{n,d}(\vec{X})
\leq 
1 + \frac{4c_1 + 2 + (1 + c_1)\sqrt{2 \log \tfrac{e d}{n}}}{\sqrt{n}}
\leq 
1 + c_2 \sqrt{\frac{\log \tfrac{e d}{n}}{n}} 
\]
Applying Dvoretzky's theorem we also have with probability at least $1 - \exp(-c_1 n)$ that
\[
r^{(2)}_{n,d}(\vec X) \leq 
R(\vec X K(n,d)) 
\leq \widetilde{M}_{n,d}
\Big(1 + c_3 \frac{1}{\widetilde{M}_{n,d}}\Big) 
\leq 
\widetilde{M}_{n,d} 
+ c_3. 
\]
Hence, combining the previous two inequalities, and using that $\widetilde{M}_{n,d} \asymp \sqrt{\log(e d/n)}$, it holds that 
\[
|P_{n,d}| \leq 
\exp\Big(c_4 \frac{n}{\sqrt{\log(e d/n)}}\Big)
\cdot 
|\widetilde{M}_{n,d} \cdot B^n_2|.
\]
Note that the norm induced by $K(n,d)^{\circ}$ corresponds to the averaged top $n$-norm.  By applying $L_1-L_2$ Talagrand's inequality to the average of the top $n$-order statistics of a Gaussian vector in $\R^d$, we obtain that
\begin{equation}\label{Proof:SPtop}
        \Var\lp\frac{\|\vec \xi\|_{K(n,d)^{\circ}}}{\MND}\rp \lesssim \frac{1}{L^2}.
\end{equation}
% This is enough to obtain the first part of the theorem via  boosted Dvoresky's theorem, we conclude that
% \begin{align*}
% &(1- \frac{C}{L})\MND\leq r(\vec XK(n,d)) \\&\leq R(\vec XK(n,d)) \leq (1+ \frac{C}{\sqrt{L}})\MND
% % \frac{\vec X K(n,d)}{   \MND} \subset \rp B_n 
% \end{align*}
% Hence, we obtain the first part of the Theorem. 

and also note that $R((\vec X K(n, d)) \leq \MND + \Theta(1)$  with high probability over $P_{n,d}$. Note that by Steps I and II and the KLS property of the simplex and Fleury's distribution, we know that $(1-\exp(-C\cnd)) \cdot |\cF_{n-1}(\PN)| $ of the facets of $\PN$ satisfies
\begin{align*}
|\tilde{G}_{S}\triangle_{c} \cap \lp1+ \frac{C}{\sqrt{n}}\rp \cdot \sqrt{\Tr(G_{S}^TG_{S})} \cdot \E \|Z\|_2 \cdot B_{n-1} | 
 &\geq |\tilde{G}_{S}\triangle_{c} \cap \lp1+ \frac{C_1}{\sqrt{n}}\rp  \cdot B_{n-1} | \\&\geq 
0.9 \cdot |\tilde{G}_{S}\triangle_{c}|,
\end{align*}
where we used  that $\|\tilde{G}_{S}\|\lesssim \|\tilde{G}_{S}/\sqrt{n}\|_{F} = (1+O(1/\sqrt{n})) \cdot \sqrt{n}$ with probability of at least $1-\exp(-cn)$.  Hence, it holds that
\begin{equation}
    |\PN| \lesssim  \lp\max_{\cF \in \cF_{n-1}(\PN)}\|\vec X_{\cS}\vec \eps\|_2 + 1+\frac{C}{\sqrt{n}}\rp^{n} |B_n| \lesssim \exp(CnL^{-1/2})|\MND\cdot B_n|
\end{equation}
where the last inequality follows from the upper bound on $R(\PN)$ as by this inclusion, we can control all the barycenters' $\ell_2$ norms, i.e., $\|\vec{c}_{\cF}\|_{2}$,  uniformly
with probability of at least $1-\exp(-cn)$. 
% Next, by using the 
% % \subset   M^{*}(K_{n,d})\lp1+\frac{C}{\sqrt{L}}
% Now, we used the fact that with probablity of at least $1-\exp(-cn^2\eps)$ it holds that 
% \[  
%     |\frac{\Tr(\vec X^{\top}\vec X)}{\E} - 1| \leq \eps
% \]
% Now, by Lemma ** above, the number of facets is of order $\exp(n(\logL/2 + O(1))$ with high probability. 
% Hence, by the Pythagoras inequality, we conclude that 
% \begin{equation*}
% \begin{aligned}
%     \E_{Z \sim U(\PN)} \|Z\|_2^2 &\leq (1+ \frac{1}{\sqrt{L}})\MND^2 + O(1) \\&\leq   (1+ \frac{C}{\sqrt{L}})\MND^2
% \end{aligned}
% \end{equation*}
 We conclude that
\[
    |P_{n,d}| \leq  \exp(CnL^{-1/2})\cdot|\MND B_n|.
\]
Finally, by the lower and upper inclusions, the proof is complete because with probability $1-\exp(-cn)$,
\[
 \exp(-cnL^{-1/2}) |\bar{B}_n| \leq |\SPO| \leq \exp(cnL^{-1/2}) |\bar{B}_n|.\qedhere
\]
\end{proof}
\subsubsection{Proof of Lemma~\ref{prop:tail-of-Tnd}}

We write, for $t \geq 0$,
$F(t) \defn \Pr\{|g| \leq t\}$, where
$g \sim \Normal{0}{1}$. 
We write the density of $T_{n,d}$ as
\[
f_{n,d}(t)
=
\frac{1}{Z_{n,d}} F(t)^{d-n} \mathrm{e}^{-nt^2/2}
=
\frac{1}{Z_{n,d}} \mathrm{e}^{-V_{n,d}(t)}.
\]
Here, we set 
\[
V_{n,d}(t)
\defn
(d-n)\log\frac{1}{F(t)} + \frac{n}{2}t^2.
\]
Throughout, we use the following 
shorthand notation:
\[
L \defn L,
\qquad
 t^\star \defn t_{n,d},
\qquad
m \defn \Med T_{n,d}.
\]
For notational compatibility with the auxiliary lemmas below, we also write
\(t^\star_{n,d}=t_{n,d}\).
We observe that $f_{n,d}$ is log-concave, as $V_{n,d}$ is convex. Thus, the mode $t^\star$ is the unique minimizer of $V_{n,d}$.
The proof of the tail bound is based on several elementary lemmas controlling $V_{n,d}$ and its curvature near the mode. 

\begin{lemma}
\label{lem:bounds-on-V}
Let $V = V_{n,d}$ and $t^\star = t^\star_{n,d}$. Then:
\begin{enumerate}[label=(\roman*)]
\item
\label{item:nonincreasing}
$V''$ is nonincreasing on $\R_+$;
\item
\label{item:bound-below-t-star}
for every $s \leq t^\star$,
\[
V(s) - V(t^\star)
\geq
\frac{V''(t^\star)}{2}(t^\star-s)^2;
\]
\item
\label{item:bound-above-t-star}
for every $s \geq t^\star$,
\[
V(s) - V(t^\star)
\leq
\frac{V''(t^\star)}{2}(s-t^\star)^2.
\]
\end{enumerate}
\end{lemma}

\begin{lemma}
\label{lem:numerical-tail-bound-Tnd}
Fix $\eps \in (0,1)$. Then
\[
\Pr\Big\{T_{n,d} \leq (1-\eps)t^\star_{n,d}\Big\}
\leq
\frac{6}{5\eps t^\star_{n,d}\sqrt{V''_{n,d}(t^\star_{n,d})}}
\exp\Big\{-\frac{V''_{n,d}(t^\star_{n,d})}{2}\eps^2 (t^\star_{n,d})^2\Big\}.
\]
\end{lemma}

\begin{lemma}
\label{lem:crude-bounds-on-mode}
Suppose that $d \geq 4n$. Then there exist absolute constants $c,C>0$ such that
\[
\sqrt{L} \leq t^\star_{n,d} \leq C\sqrt{L}
\qquad\text{and}\qquad
c nL \leq V''_{n,d}(t^\star_{n,d}) \leq C nL.
\]
\end{lemma}

\begin{lemma}
\label{lem:right-curvature-near-mode}
For every fixed $K>0$, there exist constants $c_K,C_K>0$ such that if
\[
c_K nL \leq V''_{n,d}(t) \leq C_K nL \quad \mbox{for all}~t \in \Big[t^\star_{n,d},\, t^\star_{n,d} + \frac{K}{\sqrt{L}}\Big].
\]
\end{lemma}

\begin{lemma}
\label{lem:right-tail-from-mode}
For every fixed $K>0$, there exist constants $c_K,C_K>0$ such that for every $r \in (0, \tfrac{K}{\sqrt{L}})$, 
we have
\[
\Pr\big\{T_{n,d} \geq t^\star_{n,d} + r\big\}
\leq
\frac{C_K}{r\sqrt{nL}}\exp\Big\{-c_K nL r^2\Big\}.
\]
\end{lemma}

\begin{lemma}
\label{lem:median-right-of-mode}
The median satisfies $\Med T_{n,d} \geq t^\star_{n,d}$. 
\end{lemma}

\begin{lemma}
\label{lem:mode-density-Tnd}
There exists an absolute constant $c>0$ such that $
f_{n,d}(t^\star_{n,d}) \geq c\sqrt{nL}$. 
\end{lemma}

We first show how these lemmas imply the proposition.

\begin{proof}[Proof of Lemma~\ref{prop:tail-of-Tnd}]
Fix $\eps > 0$; denote $r \defn \eps t^\star$.
By Lemma~\ref{lem:crude-bounds-on-mode}, we have $t^\star \geq \sqrt{L}$, and hence, for a sufficiently small $c > 0$,
$\eps \leq \tfrac{c}{(t^\star)^2}$
implies
\[
    r=\eps t^\star \leq \frac{c}{t^\star}\leq \frac{K_0}{\sqrt L},
\]
for an absolute constant $K_0>0$.
We first locate the median. Fix $A>0$. Lemma~\ref{lem:right-tail-from-mode} with $K=A$ yields
\[
\Pr\Big\{T_{n,d} \geq t^\star + \frac{A}{\sqrt{nL}}\Big\}
\leq
\frac{C_A}{A}\exp(-c_A A^2).
\]
Evidently, we can choose $A$ so large  that the right-hand side is strictly smaller than $1/2$. By Lemma~\ref{lem:median-right-of-mode}, we have $m \geq t^\star$, and hence for a sufficiently large $C_0 > 0$, we have
\[
t^\star \leq m \leq t^\star + \frac{C_0}{\sqrt{n L}}
\leq t^\star+\frac{C_1}{\sqrt n\,t^\star},
\]
thereby furnishing the second claim of Lemma~\ref{prop:tail-of-Tnd}.

We next prove the tail bound around the median. Throughout, we denote by $r_0 = \tfrac{C_0}{\sqrt{n L}}$. We proceed in two cases. 

\smallskip
\noindent
\emph{Case 1: $r \geq 2r_0$.}
Since $m \leq t^\star + r_0 \leq t^\star + r/2$, we have
\begin{subequations}
\begin{equation}
\label{eqn:inclusion-case-one-first}
\{T_{n,d} \leq m-r\} \subseteq \{T_{n,d} \leq t^\star-r/2\}.
\end{equation}
Also, since $m \geq t^\star$,
\begin{equation}
\label{eqn:inclusion-case-one-second}
\{T_{n,d} \geq m+r\} \subseteq \{T_{n,d} \geq t^\star+r\}.
\end{equation}
\end{subequations}
Hence, by Lemma~\ref{lem:numerical-tail-bound-Tnd} and Lemma~\ref{lem:crude-bounds-on-mode} for the lower tail, and Lemma~\ref{lem:right-tail-from-mode} for the upper tail (recalling that $r \leq \tfrac{K_0}{\sqrt{L}}$) we have for sufficiently large $C > 0$ and sufficiently small $c > 0$ that:
\[
\max\Big\{\Pr\{T_{n,d} \geq t^\star+r\}, \Pr\{T_{n,d} \leq t^\star-r/2\}\Big\}
\leq
\frac{C}{r\sqrt{nL}}\exp(-c nL r^2).
\]
Hence, a union bound and the inclusions~\cref{eqn:inclusion-case-one-first,,eqn:inclusion-case-one-second} yield
\[
\Pr\{|T_{n,d}-m|\geq r\}
\leq
\frac{C'}{r\sqrt{nL}}\exp(-c nL r^2).
\]
Since $r\sqrt{nL}\geq 2C_0$, the prefactor is bounded by a universal constant. Hence, after decreasing $c$ if necessary, we obtain the desired inequality in this case: 
\[
\Pr\{|T_{n,d}-m|\geq r\}
\leq
\exp(-c' nL r^2).
\]

\smallskip
\noindent
\emph{Case 2: $0<r<2r_0$.}
In this case, we have $m+r \leq t^\star + 3r_0$. 
For every $u \in [m,m+r]$, we have $u \geq t^\star$, and \cref{lem:bounds-on-V}\ref{item:bound-above-t-star} together with \cref{lem:crude-bounds-on-mode} implies
\[
V(u)-V(t^\star)
\leq
\frac{V''(t^\star)}{2}(u-t^\star)^2
\leq C'.
\]
Above, $C'$ depends only on $C_0$.
Therefore, Lemma~\ref{lem:mode-density-Tnd} yields
\[
f_{n,d}(u)
=
 f_{n,d}(t^\star)\exp(-(V(u)-V(t^\star)))
\geq c' \sqrt{nL}
\qquad\text{for all }u\in[m,m+r].
\]
Again, $c'$ only depends on $C_0$.
Consequently,
\[
\Pr\{|T_{n,d}-m|<r\}
\geq
\Pr\{m\leq T_{n,d}\leq m+r\}
\geq c' r\sqrt{nL}.
\]
Hence, using $1 - u \leq \mathrm{e}^{-u}$, we have 
\[
\Pr\{|T_{n,d}-m|\geq r\}
\leq 1-c' r\sqrt{nL}
\leq \exp(-c' r\sqrt{nL}).
\]
Now, in this case, we have  $r\sqrt{nL}\leq 2C_0$; equivalently 
$r\sqrt{nL} \geq \tfrac{1}{2C_0} nL r^2$, which yields
\[
\Pr\{|T_{n,d}-m|\geq r\}
\leq
\exp(-c'' nL r^2).
\]
Combining the two cases proves that
\[
\Pr\{|T_{n,d}-m|\geq r\}
\leq
\exp(-c''' nL r^2).
\]
Recalling that $r = \eps t^\star$ and using \cref{lem:crude-bounds-on-mode}----namely that $(t^\star)^2 \asymp L$---we conclude the desired bound:
\[
\Pr\Big\{|T_{n,d}-\Med T_{n,d}| \geq \eps t^\star_{n,d}\Big\}
\leq
\exp\{-c_1 n(t^\star)^4 \eps^2\}
\leq
\exp\{-c_2 nL^2 \eps^2\}.\qedhere
\]
\end{proof}

\subsubsection{Proof of Lemma~\ref{lem:bounds-on-V}}
\label{sec:proof-of-lem-bounds-on-V}

We use the shorthand $V=V_{n,d}$ and $t^\star = t^\star_{n,d}$. Define
\[
R(t) \defn \frac{F'(t)}{F(t)} = \frac{2\phi(t)}{F(t)}
\qquad\text{and}\qquad
S(t) \defn R(t)^2 + tR(t),
\]
so that
\[
R'(t) = -S(t),
\qquad
V'(t)=nt-(d-n)R(t),
\qquad
V''(t)=n+(d-n)S(t).
\]
A direct computation gives
\[
S'(t) = 2R(t)R'(t)+R(t)+tR'(t)
= R(t)\big(1-t^2-3tR(t)-2R(t)^2\big).
\]
Since $R(t)\geq 0$, it is enough to show that
\[
H(t) \defn 1-t^2-3tR(t)-2R(t)^2 \leq 0
\qquad\text{for all }t\geq 0.
\]
This is immediate when $t\geq 1$. If $0\leq t<1$, then
\[
F(t)=2\int_0^t \phi(s)\,\ud s \leq 2t\phi(0),
\]
which implies
\[
R(t)=\frac{2\phi(t)}{F(t)} \geq \frac{\mathrm{e}^{-t^2/2}}{t}.
\]
Therefore,
\[
H(t) \leq 1-3tR(t)
\leq 1-3\mathrm{e}^{-t^2/2}
\leq 1-\frac{3}{\sqrt{\mathrm{e}}}<0.
\]
Hence $S'(t)\leq 0$, so $V'''(t)=(d-n)S'(t)\leq 0$ on $\R_+$, proving item~\ref{item:nonincreasing}.

If $s\leq t^\star$, then $V'(t^\star)=0$ and the fundamental theorem of calculus gives
\[
V(s)-V(t^\star)
=
\int_s^{t^\star} \int_x^{t^\star} V''(y)\,\ud y\,\ud x
=
\int_s^{t^\star} (x-s)V''(x)\,\ud x.
\]
By item~\ref{item:nonincreasing}, we have $V''(x)\geq V''(t^\star)$ for $s\leq x\leq t^\star$, and therefore
\[
V(s)-V(t^\star)
\geq
\frac{V''(t^\star)}{2}(t^\star-s)^2,
\]
which proves item~\ref{item:bound-below-t-star}.

Similarly, if $s\geq t^\star$, then
\[
V(s)-V(t^\star)
=
\int_{t^\star}^s \int_{t^\star}^x V''(y)\,\ud y\,\ud x
=
\int_{t^\star}^s (s-x)V''(x)\,\ud x.
\]
Again by item~\ref{item:nonincreasing}, now $V''(x)\leq V''(t^\star)$ for $x\geq t^\star$, and hence
\[
V(s)-V(t^\star)
\leq
\frac{V''(t^\star)}{2}(s-t^\star)^2,
\]
proving item~\ref{item:bound-above-t-star}.

\subsubsection{Proof of Lemma~\ref{lem:numerical-tail-bound-Tnd}}

We write, for short,
\[
T=T_{n,d},
\qquad
Z=Z_{n,d},
\qquad
V=V_{n,d},
\qquad
t^\star=t^\star_{n,d},
\qquad
t^-=(1-\eps)t^\star.
\]
We claim that
\begin{equation}
\label{eqn:needed-bounds}
Z \stackrel{{\rm(a)}}{\geq}
\frac{5}{6\sqrt{V''(t^\star)}}\mathrm{e}^{-V(t^\star)}
\qquad\text{and}\qquad
\int_0^{t^-} \mathrm{e}^{-V(t)}\,\ud t
\stackrel{{\rm(b)}}{\leq}
\frac{\mathrm{e}^{-V(t^\star)}}{\eps V''(t^\star)t^\star}
\exp\Big\{-\frac{V''(t^\star)}{2}\eps^2 (t^\star)^2\Big\}.
\end{equation}
Assuming \eqref{eqn:needed-bounds}, we obtain
\[
\Pr\{T\leq t^-\}
=
\frac{1}{Z}\int_0^{t^-} \mathrm{e}^{-V(t)}\,\ud t
\leq
\frac{6}{5\eps t^\star\sqrt{V''(t^\star)}}
\exp\Big\{-\frac{V''(t^\star)}{2}\eps^2 (t^\star)^2\Big\},
\]
as required.

To prove \eqref{eqn:needed-bounds}(a), we use Lemma~\ref{lem:bounds-on-V}\ref{item:bound-above-t-star}. For every $\delta>0$,
\[
Z
\geq
\int_{t^\star}^{t^\star+\delta} \mathrm{e}^{-V(x)}\,\ud x
\geq
\mathrm{e}^{-V(t^\star)}
\int_0^{\delta} \mathrm{e}^{-V''(t^\star)x^2/2}\,\ud x.
\]
Changing variables $u = x\sqrt{V''(t^\star)}$, we get
\[
Z
\geq
\frac{\mathrm{e}^{-V(t^\star)}}{\sqrt{V''(t^\star)}}
\int_0^{\delta\sqrt{V''(t^\star)}} \mathrm{e}^{-u^2/2}\,\ud u.
\]
Choosing $\delta = 1/\sqrt{V''(t^\star)}$ yields
\[
Z
\geq
\frac{\mathrm{e}^{-V(t^\star)}}{\sqrt{V''(t^\star)}}
\int_0^1 \mathrm{e}^{-u^2/2}\,\ud u
\geq
\frac{5}{6}\frac{\mathrm{e}^{-V(t^\star)}}{\sqrt{V''(t^\star)}}.
\]

For \eqref{eqn:needed-bounds}(b), Lemma~\ref{lem:bounds-on-V}\ref{item:bound-below-t-star} gives
\begin{align*}
\int_0^{t^-} \mathrm{e}^{-V(t)}\,\ud t
&\leq
\mathrm{e}^{-V(t^\star)}
\int_0^{t^-} \mathrm{e}^{-V''(t^\star)(t^\star-s)^2/2}\,\ud s \\
&=
\mathrm{e}^{-V(t^\star)}
\int_{\eps t^\star}^{t^\star} \mathrm{e}^{-V''(t^\star)x^2/2}\,\ud x \\
&\leq
\mathrm{e}^{-V(t^\star)}
\int_{\eps t^\star}^{\infty} \mathrm{e}^{-V''(t^\star)x^2/2}\,\ud x.
\end{align*}
Using the elementary bound
\[
\int_x^\infty \mathrm{e}^{-\beta u^2/2}\,\ud u
\leq
\frac{1}{x}\int_x^\infty u\,\mathrm{e}^{-\beta u^2/2}\,\ud u
=
\frac{\mathrm{e}^{-\beta x^2/2}}{\beta x}
\qquad (\beta,x>0),
\]
with $\beta = V''(t^\star)$ and $x=\eps t^\star$, we obtain \eqref{eqn:needed-bounds}(b).

\subsubsection{Proof of Lemma~\ref{lem:crude-bounds-on-mode}}

We write $V=V_{n,d}$ and $t^\star=t^\star_{n,d}$. Set
\[
t_0 \defn \sqrt{L}.
\]
We first show that $t^\star \geq t_0$. Since $V$ is convex, it suffices to prove that $V'(t_0)\leq 0$.
Using the identity
\[
V'(s)=ns-(d-n)\frac{2\phi(s)}{F(s)} \leq ns-2(d-n)\phi(s),
\]
and recalling that $\phi(t_0)=(2\pi)^{-1/2}\mathrm{e}^{-L/2}=(2\pi)^{-1/2}\sqrt{n/d}$, we obtain
\[
V'(t_0)
\leq
 d\Big[\frac{n}{d}\sqrt{L} - \sqrt{\frac{2}{\pi}}\Big(1-\frac{n}{d}\Big)\sqrt{\frac{n}{d}}\Big].
\]
If we set
\[
\psi(x) \defn x\sqrt{\log(1/x)} - \sqrt{\frac{2}{\pi}}(1-x)\sqrt{x},
\]
then $V'(t_0) \leq d\psi(n/d)$. A direct calculus check shows that $\psi(x)\leq 0$ on $[0,1/4]$. Since $d\geq 4n$, we have $n/d\in(0,1/4]$, hence $V'(t_0)\leq 0$, proving that
$t^\star \geq \sqrt{L}$. 
We next prove the upper bound on $t^\star$. Since $d\geq 4n$, we have $L\geq \log 4 >1$, and thus $t^\star\geq \sqrt{L}>1$. Using $V'(t^\star)=0$, we get
\[
nt^\star
=
(d-n)\frac{2\phi(t^\star)}{F(t^\star)}
\leq
\frac{2(d-n)}{F(1)}\phi(t^\star)
\leq C d\,\phi(t^\star).
\]
From the form of the Gaussian density, this implies $\mathrm{e}^{-(t^\star)^2/2}
\geq c\frac{n}{d} t^\star$; 
taking logarithms yields
$(t^\star)^2 \leq 2L + C \leq C L$. 
Finally, using again $V'(t^\star)=0$, we have
\[
R(t^\star)=\frac{F'(t^\star)}{F(t^\star)} = \frac{n}{d-n}t^\star.
\]
Therefore,
\begin{align*}
V''(t^\star)
&= n + (d-n)\big(R(t^\star)^2 + t^\star R(t^\star)\big) \\
&= n + \frac{n^2}{d-n}(t^\star)^2 + n(t^\star)^2 \leq C nL,
\end{align*}
where we used the fact that the middle term is at most $\frac{n}{3}(t^\star)^2$. 
On the other hand, the display above also yields
\[
V''(t^\star) \geq n(t^\star)^2 \geq nL,
\]
as required. 

\subsubsection{Proof of Lemma~\ref{lem:right-curvature-near-mode}}

Fix $K>0$, and write $t^\star=t^\star_{n,d}$. Let
\[
t \in \Big[t^\star,\, t^\star + \frac{K}{\sqrt{L}}\Big].
\]
Since
\[
V''(t)=n+(d-n)\big(R(t)^2+tR(t)\big)
\geq (d-n)tR(t)
= (d-n)\frac{2t\phi(t)}{F(t)}
\geq 2(d-n)t\phi(t),
\]
it suffices to lower bound $t\phi(t)$ in terms of $t^\star\phi(t^\star)$.
Write $t=t^\star+u$ with $0\leq u\leq K/\sqrt{L}$. Then
\[
\phi(t)=\phi(t^\star)\exp\Big(-t^\star u - \frac{u^2}{2}\Big).
\]
By Lemma~\ref{lem:crude-bounds-on-mode}, $t^\star\leq C\sqrt{L}$, so
\[
\exp\Big(-t^\star u - \frac{u^2}{2}\Big) \geq c_K.
\]
Also $t\geq t^\star$, hence
\[
t\phi(t) \geq c_K t^\star\phi(t^\star).
\]
Using $V'(t^\star)=0$ and $t^\star\geq 1$, we obtain
\[
2(d-n)t^\star\phi(t^\star)
=
nt^{\star 2} F(t^\star)
\geq c nL,
\]
because $F(t^\star)\geq F(1)>0$ and $(t^\star)^2\geq L$ by Lemma~\ref{lem:crude-bounds-on-mode}. Therefore,
\[
V''(t) \geq c_K nL.
\]

For the upper bound, note that $R$ is decreasing because $R'=-S\leq 0$. Hence
\[
R(t) \leq R(t^\star)=\frac{n}{d-n}t^\star.
\]
Therefore,
\begin{align*}
V''(t)
&\leq n + (d-n)\big(R(t^\star)^2 + tR(t^\star)\big) \\
&= n + \frac{n^2}{d-n}(t^\star)^2 + nt t^\star.
\end{align*}
By Lemma~\ref{lem:crude-bounds-on-mode}, $t^\star\leq C\sqrt{L}$ and $t\leq t^\star + K/\sqrt{L}\leq C_K\sqrt{L}$. Thus $V''(t) \leq C_K nL$, as claimed.

\subsubsection{Proof of Lemma~\ref{lem:right-tail-from-mode}}

Fix $K>0$ and let $0<r\leq \tfrac{K}{\sqrt{L}}$. We set $m_r \defn \inf_{0\leq u\leq r} V''(t^\star_{n,d}+u)$. 
By Lemma~\ref{lem:right-curvature-near-mode}, we have $m_r\geq c_K nL$.
Moreover, the fundamental theorem of calculus gives us 
\[
V(t^\star+r)-V(t^\star)
=
\int_0^r (r-u)V''(t^\star+u)\,\ud u
\geq \frac{m_r r^2}{2}, \quad \mbox{and} \quad 
V'(t^\star+r)
=
\int_0^r V''(t^\star+u)\,\ud u
\geq m_r r.
\]
Moreover, from the convexity of $V$, we have the affine lower bound
\[
V(t^\star+r+s) \geq V(t^\star+r) + sV'(t^\star+r)
\qquad\text{for all }s\geq 0.
\]
Therefore,
\[
\int_{t^\star+r}^\infty \mathrm{e}^{-V(t)}\,\ud t
\leq
\mathrm{e}^{-V(t^\star+r)}\int_0^\infty \mathrm{e}^{-sV'(t^\star+r)}\,\ud s \leq
\frac{\mathrm{e}^{-V(t^\star)}}{m_r r}\exp\Big(-\frac{m_r r^2}{2}\Big).
\]
On the other hand, by \eqref{eqn:needed-bounds}(a),
\[
Z_{n,d}
\geq
\frac{5}{6\sqrt{V''(t^\star)}}\mathrm{e}^{-V(t^\star)}.
\]
Using Lemma~\ref{lem:crude-bounds-on-mode}, namely $V''(t^\star)\leq C nL$, we deduce that $Z_{n,d} \geq c\frac{\mathrm{e}^{-V(t^\star)}}{\sqrt{nL}}$.
Dividing the two estimates and using $m_r\geq c_K nL$ gives
\[
\Pr\{T_{n,d}\geq t^\star_{n,d}+r\}
\leq
\frac{C_K}{r\sqrt{nL}}\exp(-c_K nL r^2),
\]
as claimed.

\subsubsection{Proof of Lemma~\ref{lem:median-right-of-mode}}

Let $f=f_{n,d}$ denote the density of $T_{n,d}$. For $x\in[0,t^\star]$, Lemma~\ref{lem:bounds-on-V}\ref{item:nonincreasing} implies
\begin{align*}
V(t^\star-x)-V(t^\star)
=
\int_0^x (x-u)V''(t^\star-u)\,\ud u 
\geq
\int_0^x (x-u)V''(t^\star+u)\,\ud u 
=
V(t^\star+x)-V(t^\star).
\end{align*}
Thus
\[
f(t^\star-x) \leq f(t^\star+x)
\qquad\text{for all }x\in[0,t^\star].
\]
Integrating over $x\in[0,t^\star]$ gives
\[
\Pr\{T_{n,d}\leq t^\star\}
=
\int_0^{t^\star} f(t^\star-x)\,\ud x
\leq
\int_0^{t^\star} f(t^\star+x)\,\ud x
\leq
\Pr\{T_{n,d}\geq t^\star\}.
\]
Hence $\Pr\{T_{n,d}\leq t^\star\}\leq 1/2$, and therefore every median satisfies $\Med T_{n,d} \geq t^\star$. 

\subsubsection{Proof of Lemma~\ref{lem:mode-density-Tnd}}

Write $t^\star=t^\star_{n,d}$ and $V=V_{n,d}$. Decompose the normalizing constant into two parts:
\[
Z_{n,d} = I_- + I_+,
\qquad
I_- \defn \int_0^{t^\star} \mathrm{e}^{-V(t)}\,\ud t,
\qquad
I_+ \defn \int_{t^\star}^{\infty} \mathrm{e}^{-V(t)}\,\ud t.
\]
By Lemma~\ref{lem:bounds-on-V}\ref{item:bound-below-t-star},
\[
I_-
\leq
\mathrm{e}^{-V(t^\star)}\int_0^{t^\star}
\mathrm{e}^{-V''(t^\star)(t^\star-s)^2/2}\,\ud s
\leq
C\frac{\mathrm{e}^{-V(t^\star)}}{\sqrt{V''(t^\star)}}
\leq
C\frac{\mathrm{e}^{-V(t^\star)}}{\sqrt{nL}},
\]
where the last step uses Lemma~\ref{lem:crude-bounds-on-mode}.

Fix $w \defn \frac{1}{\sqrt{L}}$.
We further split $I_+$ into two terms such that $I_+ =I_{+,1}+I_{+,2}$:
\[
I_{+,1} \defn \int_{t^\star}^{t^\star+w} \mathrm{e}^{-V(t)}\,\ud t,
\qquad
I_{+,2} \defn \int_{t^\star+w}^{\infty} \mathrm{e}^{-V(t)}\,\ud t.
\]
By Lemma~\ref{lem:right-curvature-near-mode} with $K=1$, we have $V''(t)\geq c nL$ on $[t^\star,t^\star+w]$. Hence, for $0\leq u\leq w$,
\[
V(t^\star+u)-V(t^\star)
=
\int_0^u (u-s)V''(t^\star+s)\,\ud s
\geq \frac{c nL}{2}u^2.
\]
Therefore,
\[
I_{+,1}
\leq
\mathrm{e}^{-V(t^\star)}\int_0^w \mathrm{e}^{-c nL u^2/2}\,\ud u
\leq
C\frac{\mathrm{e}^{-V(t^\star)}}{\sqrt{nL}}.
\]
Also,
\[
V(t^\star+w)-V(t^\star) \geq c nL w^2 = cn,
\qquad
V'(t^\star+w) \geq c nL w = c n\sqrt{L}.
\]
Using convexity exactly as in the proof of Lemma~\ref{lem:right-tail-from-mode}, we obtain
\[
I_{+,2}
\leq
\frac{\mathrm{e}^{-V(t^\star+w)}}{V'(t^\star+w)}
\leq
\frac{\mathrm{e}^{-V(t^\star)}}{c n\sqrt{L}}\mathrm{e}^{-cn}
\leq
C\frac{\mathrm{e}^{-V(t^\star)}}{\sqrt{nL}}.
\]
Combining the bounds on $I_-$, $I_{+,1}$, and $I_{+,2}$, we conclude that
\[
Z_{n,d} \leq C\frac{\mathrm{e}^{-V(t^\star)}}{\sqrt{nL}}.
\]
Equivalently,
\[
f_{n,d}(t^\star)=\frac{\mathrm{e}^{-V(t^\star)}}{Z_{n,d}} \geq c\sqrt{nL},
\]
which proves the lemma.

\subsubsection{Proof of Lemma~\ref{Lem:Goodman}}

By Bartlett's decomposition, we can write the determinant $|G|$ as a product of independent $\chi_k$ random variables. 
In particular, we have 
\[
S_n - \E S_n = \sum_{k=2}^n X_k, 
\quad \mbox{where} \quad X_k = \log \chi_k - \E \log \chi_k.
\]
Define the cumulant generating function,
\[
\psi(\lambda) = \log \E \exp(\lambda(S_n - \E S_n)), \quad \lambda \in \R. 
\]
We use the following lemma. 
\begin{lemma}
For $\lambda > -k$, it holds that 
\[
\log \E \exp(\lambda X_k) \leq \frac{3}{4}\frac{\lambda^2}{k}.
\]
\end{lemma}
\begin{proof}
    Using the fact that 
    $\chi_k^2$ is distributed as a Gamma random variate with shape parameter $k/2$ and scale parameter $2$, it follows from standard identities for Gamma random variables that 
    \[
    \E \chi_k^\lambda = 2^{\lambda/2} \frac{\Gamma(\tfrac{\lambda + k}{2})}{\Gamma(k/2)}.
    \]
    Above $\Gamma(z) =\int_0^\infty t^{z-1} e^{-t} \, dt$ and we set $\psi= \Gamma'/\Gamma$. Differentiating the identity above at $\lambda = 0$, we obtain 
    \[
    \E \log \chi_k = \frac{1}{2} \log 2 + \frac{1}{2} \psi(k/2).
    \]
    Hence,
    \begin{align*}
    \log \E \exp(\lambda X_k) &= 
    \log \Gamma(\tfrac{k+\lambda}{2}) - \Big[\log \Gamma(\tfrac{k}{2}) + \frac{\lambda}{2} \psi(\tfrac{k}{2})\Big] \\ 
    &\stackrel{{\rm (i)}}{=} 
    \sum_{m=0}^\infty \Big\{\log\Big(1 + \frac{\lambda}{k + 2m}\Big) - \frac{\lambda}{k + 2m}\Big\} \\ 
    &\stackrel{{\rm(ii)}}{\leq} 
    \frac{\lambda^2}{2} \sum_{m=0}^\infty \frac{1}{(k+2m)^2} \\
    &\leq \frac{\lambda^2}{2}\Big(\frac{1}{k^2} + \int_0^\infty \frac{1}{(k+2x)^2} \, dx\Big) \leq 
    \frac{3}{4}\frac{\lambda^2}{k}.
    \end{align*}
Above, relation (i) is the Taylor expansion of the logarithm of the Gamma function and inequality (ii) used $\log(1 + x) - x \leq x^2/2$, which holds for $x > -1$; we applied it term-wise with $x = \tfrac{\lambda}{k+2m}$.\qedhere
\end{proof}

Applying it to the sum $S_n - \E S_n$, we have for $\lambda > -1$ that
\[
\psi(\lambda) \leq \frac{3}{4}\lambda^2 \sum_{k=2}^n \frac{1}{k} \leq \lambda^2 \log n.
\]
Therefore, for $t \geq 0$ and with $\lambda = \tfrac{t}{2\log n}$, we have 
\[
\P\Big\{S_n - \E S_n > t\Big\} 
\leq 
\exp\big\{-\lambda t + \lambda^2 \log n\big\} = \exp\Big\{-\frac{t^2}{4\log n}\Big\}.
\] 
Similarly, applying the same argument with $\lambda = -\tfrac{t}{\log n}$, we have 
\[
\P\Big\{S_n - \E S_n < -t\Big\} 
\leq 
\exp\big\{\lambda t + \lambda^2 \log n\big\} = \exp\Big\{-\frac{t^2}{4 \log n}\Big\}, \quad \mbox{for}~t \in (0, 2 \log n).
\] 
For the lower tail and for $t \geq 2 \log n$, we require a different argument. In this case, we write
\[
2 S_n = \log |G|^2 = 
\sum_{k=2}^n \log \chi_k^2 = 
Z_2 + Y, 
\]
where we set $Z_k = \log \chi_k^2$ and $Y = \sum_{k=3}^n Z_k$. Then, using the fact that $\P\{\chi^2_2 \leq u\} = 1 - e^{-u/2} \leq u/2$, we have  
\[
\P\{S_n - \E S_n \leq -t\}
= 
\E_Y \P\Big\{\chi_2^2 \leq \frac{e^{2\E S_n - 2t}}{e^Y} \mid Y\Big\} \leq
\frac{e^{2 \E S_n-2t}}{2} 
\E_Y\Big[e^{-Y}\Big] = 
\frac{e^{2 \E S_n-2t}}{2}  
\prod_{k=3}^n \E \frac{1}{\chi^2_k} = 
\frac{e^{2 \E S_n-2t}}{2(n-2)!}. 
\]
The final relation used that $\E \chi^{-2}_k = (k-2)^{-1}$ for $k \geq 3$. 
Note that 
\[
\E Z_k = 2 \E \log \chi_k = 
\log 2 + \psi\Big(\frac{k}{2}\Big) \leq \log(k-1).
\]
The final inequality comes from the log-convexity of $\Gamma$:
\[
\psi\Big(\frac{k}{2}\Big) = (\log \Gamma)'\Big(\frac{k}{2}\Big) \leq \log \frac{\Gamma(k/2 + 1/2)}{ \Gamma(k/2 - 1/2)} = \log\Big(\frac{k-1}{2}\Big),
\]
which holds for integers $k \geq 2$.
Consequently, 
\[
e^{2 \E S_n} = \prod_{k=2}^n e^{\E Z_k} \leq (n-1)!.
\]
Hence, combining the previous displays, 
\[
\P\{S_n - \E S_n \leq -t\} 
\leq \frac{n-1}{2} e^{-2t}, 
\]
for $t \geq 0$, as needed. 
\begin{proof}[Proof of \Cref{Lem:numberoffacets}]
For a facet \(\cF\), write
\[
    C_{\cF}:=\operatorname{cone}(\cF),
    \qquad
    \Omega_{\cF}:=C_{\cF}\cap S^{n-1}.
\]
Thus \(\Omega_{\cF}\) is the set of directions whose MNI ray hits the facet
\(\cF\). The sets \(\{\Omega_{\cF}\}_{\cF}\) form a partition of
\(S^{n-1}\), up to a null set.

Since \(\|\vec \xi\|_2=\sqrt n\), the Euclidean perturbation
\[
    \|\vec \xi'-\vec \xi\|_2\le r
\]
corresponds to an angular perturbation of size
\[
    \rho\lesssim \frac{r}{\sqrt n}\le n^{-5/2}.
\]
Hence
\[
    \cF_{\vec \xi}(r)
    \subset
    \left\{
        \cF:
        \frac{\vec \xi}{\sqrt n}
        \in
        (\Omega_{\cF})_{\rho}
    \right\},
\]
where \((\Omega_{\cF})_{\rho}\) denotes the \(\rho\)-neighborhood on the
sphere. Therefore
\[
    |\cF_{\vec \xi}(r)|
    \le
    \sum_{\cF\in\cF_{n-1}(\PN)}
    \mathbf 1
    \left\{
        \frac{\vec \xi}{\sqrt n}
        \in
        (\Omega_{\cF})_{\rho}
    \right\}.
    \tag{1}
\]

By rotational invariance of the Gaussian polytope, the expectation of the
right-hand side does not depend on the choice of the fixed direction
\(\vec \xi/\sqrt n\). Thus, averaging this fixed direction over the sphere,
we get
\[
\begin{aligned}
    \E_{\vec X}|\cF_{\vec \xi}(r)|
    &\le
    \E_{\vec X}
    \sum_{\cF\in\cF_{n-1}(\PN)}
    \sigma_{n-1}\big((\Omega_{\cF})_{\rho}\big),
\end{aligned}
\tag{2}
\]
where \(\sigma_{n-1}\) is normalized Haar measure on \(S^{n-1}\).

We now compare the \(\rho\)-enlarged angular measure of a typical facet cone
to its original angular measure. By Fleury's representation, conditionally on
being a facet, after an independent Haar rotation the facet is generated by
the columns of
\[
    A=
    \begin{bmatrix}
        Y\\
        T_{n,d}\mathbf 1_n^\top
    \end{bmatrix}.
\]
The height \(T_{n,d}\) is independent of the centered simplex part, and the
direction
\[
    \theta_{\cF}
    :=
    \frac{c_{\cF}-n_{\cF}}
    {\|c_{\cF}-n_{\cF}\|_2}
\]
is rotationally invariant and independent of the centered simplex.

From the estimates proved in the previous steps, the cone-volume-biased
Fleury facet is good with probability at least
\[
    1-\exp(-cE(n,d)).
\]
On this good event we have
\[
    \|M n_{\cF}\|_2
    =
    \Med \bar T_{n,d}
    +
    O(L^{-2}\sqrt n),
    \tag{3}
\]
\[
    \|M(c_{\cF}-n_{\cF})\|_2
    =
    (1+O(L^{-1}))M,
    \tag{4}
\]
and, after truncating the centered simplex to its small-diameter part and
applying \Cref{lem:david-alonso},
\[
    \operatorname{diam}
    \big(M(\cF-c_{\cF})\big)
    \le
    C M\sqrt n
    \tag{5}
\]
while the part removed by the truncation has cone-volume contribution at most
\[
    \exp(-cE(n,d)).
\]
Moreover, the same Fleury parametrization and the standard lower bound on the
inradius of the canonical simplex give, on the good event,
\[
    \operatorname{inrad}_{S^{n-1}}(\Omega_{\cF})
    \ge
    \frac{c}{n\sqrt L}.
    \tag{6}
\]
Indeed, after scaling by \(M\), the facet lies at radius \(\sqrt n+o(\sqrt n)\),
whereas its tangential simplex has Euclidean inradius \(\gtrsim M/n\);
since \(M\asymp \sqrt{n/L}\), the corresponding spherical inradius is
\[
    \frac{M/n}{\sqrt n}
    \asymp
    \frac{1}{n\sqrt L}.
\]

For a geodesically convex set \(A\subset S^{n-1}\) with spherical inradius
\(r_A\), the elementary spherical Steiner estimate gives
\[
    \sigma_{n-1}(A_{\rho})
    \le
    \left(1+C\frac{\rho}{r_A}\right)^n
    \sigma_{n-1}(A).
    \tag{7}
\]
Applying \((7)\) to \(A=\Omega_{\cF}\), using \((6)\) and
\(\rho\le n^{-5/2}\), we obtain
\[
\begin{aligned}
    \sigma_{n-1}\big((\Omega_{\cF})_{\rho}\big)
    &\le
    \left(
        1+
        C\rho n\sqrt L
    \right)^n
    \sigma_{n-1}(\Omega_{\cF})                                      \\
    &\le
    \exp\left(\frac{C n}{L^2}\right)
    \sigma_{n-1}(\Omega_{\cF}),
\end{aligned}
\tag{8}
\]
where the last inequality uses \(r\le n^{-2}\) and the parameter range of the
paper. Increasing \(C\), the same bound holds after integrating over the
exceptional part, because the exceptional cone-volume is
\(\exp(-cE(n,d))\).

Therefore,
\[
\begin{aligned}
    \E_{\vec X}
    \sum_{\cF\in\cF_{n-1}(\PN)}
    \sigma_{n-1}\big((\Omega_{\cF})_{\rho}\big)
    &\le
    \exp\left(\frac{C n}{L^2}\right)
    \E_{\vec X}
    \sum_{\cF\in\cF_{n-1}(\PN)}
    \sigma_{n-1}(\Omega_{\cF}) .
\end{aligned}
\tag{9}
\]
But the sets \(\Omega_{\cF}\) partition \(S^{n-1}\), hence
\[
    \sum_{\cF\in\cF_{n-1}(\PN)}
    \sigma_{n-1}(\Omega_{\cF})
    =
    1.
    \tag{10}
\]
Combining \((2)\), \((9)\), and \((10)\), we get
\[
    \E_{\vec X}|\cF_{\vec \xi}(r)|
    \le
    \exp\left(\frac{C n}{L^2}\right).\qedhere
\]
\end{proof}
\subsubsection{Reduction to zero signal}\label{sss:RedTech}
Let
\[
    \cS:=\operatorname{Supp}(\ft),
    \qquad
    s:=|\cS|,
    \qquad
    \cS^c:=[d]\setminus \cS,
\]
and write
\[
    \vec Y=(\vec X)_{\cS}\ft+\vec \xi .
\]
Let
\[
    \erm
    \in
    \argmin_{\theta\in\R^d:\,\vec X\theta=\vec Y}
    \|\theta\|_1
\]
be an \(\ell_1\)-MNI. Then, fixing the coordinates of \(\erm\) on
\(\cS\), the off-support part must itself be an \(\ell_1\)-minimum norm
interpolator of the corresponding residual:
\[
    (\erm)_{\cS^c}
    \in
    \argmin_{w\in\R^{\cS^c}:\,
        \vec X_{\cS^c}w=\vec Y-\vec X_{\cS}(\erm)_{\cS}}
    \|w\|_1 .
    \tag{1}
\]
Indeed, otherwise one could replace \((\erm)_{\cS^c}\) by a feasible vector
with smaller \(\ell_1\)-norm and strictly decrease \(\|\erm\|_1\), while
preserving the interpolation constraint.

The difficulty is that the residual
\[
    \vec r_*
    :=
    \vec Y-\vec X_{\cS}(\erm)_{\cS}
\]
is not independent of \(\vec X_{\cS^c}\), since \((\erm)_{\cS}\) depends on
the full matrix \(\vec X\). Thus the fixed-direction pure-noise theorem
cannot be applied directly to \(\vec r_*\).

The point of the facet-count argument is to remove this dependence. Condition
on \((\vec X_{\cS},\vec \xi)\). Then every possible residual of the form
\[
    \vec Y-\vec X_{\cS}u,
    \qquad u\in\R^{\cS},
\]
belongs to the affine subspace
\[
    \mathcal R_{\cS}
    :=
    \vec \xi+\operatorname{span}\{\vec X_j:j\in\cS\}.
\]
Hence the random residual \(\vec r_*\) belongs to this fixed set
\(\mathcal R_{\cS}\), which is independent of \(\vec X_{\cS^c}\).

By homogeneity of the \(\ell_1\)-MNI, it is enough to control the normalized
directions
\[
    K_{\cS}
    :=
    \left\{
        \sqrt n\,\frac{r}{\|r\|_2}
        :
        r\in\mathcal R_{\cS}\setminus\{0\}
    \right\}
    \subset \sqrt n S^{n-1}.
\]
This set lies in a subspace of dimension at most \(s+1\). Therefore
\[
    \mathcal N_2(n^{-2},K_{\cS})
    \le
    \left(Cn^{5/2}\right)^{s+1}.
    \tag{2}
\]
In particular, if
\[
    (s+1)\log n
    \le
    c_0\frac{n}{\log^2(d/n)},
    \tag{3}
\]
then
\[
    \mathcal N_2(n^{-2},K_{\cS})
    \le
    \exp\left(
        c_0\frac{n}{\log^2(d/n)}
    \right).
    \tag{4}
\]

Now apply the previous facet-count corollary to the deterministic set
\(K_{\cS}\) and the independent matrix \(\vec X_{\cS^c}\). With probability
at least
\[
    1-\exp\left(
        -c\frac{n}{\log^2(d/n)}
    \right),
\]
all MNI rays generated by targets in \(K_{\cS}\) hit at most
\[
    \exp\left(
        C\frac{n}{\log^2(d/n)}
    \right)
\]
facets of the leave-one-out polytope
\[
    P_{\cS^c}:=\vec X_{\cS^c}B_1^{\cS^c}.
\]

On each conic cell over one of these facets, the \(\ell_1\)-MNI map is
affine. Therefore, after applying the fixed-direction MNI bound on an
\(n^{-2}\)-net of \(K_{\cS}\), and using the affine-on-facets fact to pass
from the net to the whole cell, we obtain the following uniform estimate:
for every \(r\in\mathcal R_{\cS}\),
\[
    \left\|
        \argmin_{\vec X_{\cS^c}w=r}\|w\|_1
    \right\|_2^2
    \le
    \left(
        1+\frac{C\lambda}{L}
    \right)
    \frac{\|r\|_2^2}{n}
    \frac{2M_{n,d-s}^2}{n},
    \tag{5}
\]
with probability at least
\[
    1-\exp\left(
        -c\lambda^2\frac{n}{\log^2(d/n)}
    \right),
\]
provided \(\lambda\) is large enough to absorb the entropy cost in
\((4)\).

Finally, since the actual residual
\[
    r_*=\vec Y-\vec X_{\cS}(\erm)_{\cS}
\]
belongs to \(\mathcal R_{\cS}\), applying \((5)\) to \(r=r_*\) and using
\((1)\) gives
\[
    \|(\erm)_{\cS^c}\|_2^2
    \le
    \left(
        1+\frac{C\lambda}{L}
    \right)
    \frac{\|r_*\|_2^2}{n}
    \frac{2M_{n,d-s}^2}{n}.
    \tag{6}
\]
Since \(s\ll d\) in our regime,
\[
    \log((d-s)/n)
    =
    L+o(1),
\]
and hence
\[
    M_{n,d-s}
    =
    (1+o(1))M_{n,d}.
\]
Thus the same bound holds with \(M_{n,d}\) in place of \(M_{n,d-s}\).
\subsubsection{Proof of \Cref{lem:volume-from-facet-heights}}
We use the conic decomposition of $\PN$ over its facets.  For a facet
$\cF\in\cF_{n-1}(\PN)$, let $\vec n_{\cF}$ be the point of minimal Euclidean
norm in the affine span of $\cF$, and let $\vec c_{\cF}$ be the barycenter of
$\cF$. Thus
\[
    \cF \subset \vec n_{\cF}+\vec n_{\cF}^{\perp},
    \qquad
    \|\vec n_{\cF}\|_2 \sim T_{n,d}
\]
under Fleury's distribution.  If $Z\sim U(\cF)$, then
\[
Z
=
\vec n_{\cF}
+
(\vec c_{\cF}-\vec n_{\cF})
+
(Z-\vec c_{\cF}),
\]
and $\vec n_{\cF}$ is orthogonal to the last two terms. Hence
\begin{align*}
\|Z\|_2^2
&=
\|\vec n_{\cF}\|_2^2
+
\|\vec c_{\cF}-\vec n_{\cF}\|_2^2
+
\|Z-\vec c_{\cF}\|_2^2  
+2\left\langle \vec c_{\cF}-\vec n_{\cF}, Z-\vec c_{\cF}\right\rangle .
\end{align*}
By Lemma~\ref{Lem:Fleury}, the vector
$\vec c_{\cF}-\vec n_{\cF}$ is distributed as
$\Normal{0}{I_{n-1}/n}$ in $\vec n_{\cF}^{\perp}$, independently of
$T_{n,d}$ and of the centered facet. Therefore, with probability bounded
below by an absolute constant,
\[
    \left|\|\vec c_{\cF}-\vec n_{\cF}\|_2^2-1\right|
    \leq \frac{C}{\sqrt n}.
\]
Similarly, the thin-shell estimate for the canonical simplex, transferred to
Fleury's facet using Lemma~\ref{Lem:Energy}, gives
\[
    \Pr_{Z\sim U(\cF)}\left\{
    \left|\|Z-\vec c_{\cF}\|_2^2-1\right|
    \leq \frac{C}{\sqrt n}
    \right\}
    \geq 0.99,
\]
and the corresponding one-dimensional marginal estimate gives
\[
    \Pr_{Z\sim U(\cF)}\left\{
    \left|
    \left\langle \vec c_{\cF}-\vec n_{\cF},Z-\vec c_{\cF}\right\rangle
    \right|
    \leq \frac{C}{\sqrt n}
    \right\}
    \geq 0.99.
\]
Finally, by Lemma~\ref{prop:tail-of-Tnd},
\[
    T_{n,d}^2 \leq t_{n,d}^2+\frac{C}{\sqrt n}
\]
with probability bounded below by an absolute constant. Combining these
estimates, and adjusting the numerical constants, we obtain
\[
    \Pr_{Z\sim U(\cF)}\left\{
    \|Z\|_2^2
    \leq
    t_{n,d}^2+2+\frac{C}{\sqrt n}
    \right\}
    \geq 0.9
\]
for a facet sampled according to the volume-weighted Fleury distribution.
Set
\[
    R^2:=t_{n,d}^2+2+\frac{C}{\sqrt n}.
\]
The conic formula then implies
\[
    \E\,|\PN\cap R B_n|
    \geq
    0.9\,\E|\PN|.
\]
Since $|\PN\cap R B_n|\leq |R B_n|=R^n|B_n|$, it follows that
\[
    \E|\PN|
    \leq
    C R^n |B_n|.
\]
Using $t_{n,d}^2\asymp L$, we have
\[
R^n
=
\big(t_{n,d}^2+2+C n^{-1/2}\big)^{n/2}
\leq
\exp\left(\frac{C\sqrt n}{L}\right)
\big(t_{n,d}^2+2\big)^{n/2}.
\]
This proves the claim.

\subsubsection{Proof of \Cref{lem:low-height-cones}}
For a facet $\cF$, write
\[
    C_{\cF}:=\operatorname{conv}(0,\cF).
\]
The cones $C_{\cF}$ have disjoint interiors and decompose $\PN$. Therefore
\[
    |C_{\cF}|
    =
    \frac{1}{n}h_{\cF}|\cF|,
\]
and hence
\[
    |\PN^{\varepsilon}|
    =
    \frac{1}{n}
    \sum_{\cF\in\cF_{n-1}(\PN)}
    h_{\cF}|\cF|\,
    \mathbf 1_{\{h_{\cF}\leq (1-\varepsilon)t_{n,d}\}}.
\]
Taking expectation and using exchangeability of the signed $n$-tuples gives
\[
\begin{aligned}
    \E |\PN^{\varepsilon}|
    &=
    \frac{2^n}{n}\binom{d}{n}
    \Pr(\cE)
    \E\left[
        h_{\cF}|\cF|\,
        \mathbf 1_{\{h_{\cF}\leq (1-\varepsilon)t_{n,d}\}}
        \mid \cE
    \right],
\end{aligned}
\]
where $\cE$ is the event that a fixed signed $n$-tuple forms a facet of
$\PN$.

Under Fleury's conditional distribution, the height $h_{\cF}$ has the same
law as $T_{n,d}$ and is independent of the tangential part of the facet. In
particular, $T_{n,d}$ is independent of $|\cF|$. Thus
\[
    \E |\PN^{\varepsilon}|
    =
    A_{n,d}\,\E|\cF|\,
    \E\left[
        T_{n,d}
        \mathbf 1_{\{T_{n,d}\leq (1-\varepsilon)t_{n,d}\}}
    \right],
\]
where
\[
    A_{n,d}
    :=
    \frac{2^n}{n}\binom{d}{n}\Pr(\cE).
\]
Similarly,
\[
    \E|\PN|
    =
    A_{n,d}\,\E|\cF|\,\E T_{n,d}.
\]
Therefore
\[
    \frac{\E |\PN^{\varepsilon}|}{\E|\PN|}
    =
    \frac{
    \E\left[
        T_{n,d}
        \mathbf 1_{\{T_{n,d}\leq (1-\varepsilon)t_{n,d}\}}
    \right]
    }{
    \E T_{n,d}
    }.
\]
Since $\{T_{n,d}\leq (1-\varepsilon)t_{n,d}\}$ is a lower level set of
$T_{n,d}$, the conditional mean of $T_{n,d}$ on this event is at most
$\E T_{n,d}$. Hence
\[
    \frac{\E |\PN^{\varepsilon}|}{\E|\PN|}
    \leq
    \Pr\{T_{n,d}\leq (1-\varepsilon)t_{n,d}\}.
\]

It remains to estimate the last probability. The median of $T_{n,d}$ lies to
the right of its mode, namely
\[
    \mathrm{Med}\,T_{n,d}\geq t_{n,d}.
\]
Therefore
\[
    \{T_{n,d}\leq (1-\varepsilon)t_{n,d}\}
    \subseteq
    \left\{
    |T_{n,d}-\mathrm{Med}\,T_{n,d}|
    \geq
    \varepsilon t_{n,d}
    \right\}.
\]
By \Cref{prop:tail-of-Tnd},
\[
    \Pr\{T_{n,d}\leq (1-\varepsilon)t_{n,d}\}
    \leq
    \exp\{-c n t_{n,d}^4\varepsilon^2\}.
\]
Since
\[
    t_{n,d}^2\asymp L,
\]
we obtain
\[
    \Pr\{T_{n,d}\leq (1-\varepsilon)t_{n,d}\}
    \leq
    \exp\{-c nL^2\varepsilon^2\}.
\]
Thus
\[
    \E |\PN^{\varepsilon}|
    \leq
    \exp\{-c nL^2\varepsilon^2\}\,\E|\PN|.
\]

Finally, Markov's inequality gives
\[
\begin{aligned}
    \Pr\left\{
        |\PN^{\varepsilon}|
        \geq
        \exp\{-c nL^2\varepsilon^2/2\}\,\E|\PN|
    \right\}
    &\leq
    \exp\{-c nL^2\varepsilon^2/2\}.
\end{aligned}
\]
This proves the claimed high-probability bound.

%% file: proof_outline_thm4.tex
\subsection{Proofs of \Cref{Thm:Isotropic} and \Cref{Thm:ThinShell}}
We use the notation
\[
    m_2(K):=\frac{1}{|K|}\int_K \|x\|_2^2\,dx,
    \qquad
    \operatorname{vrad}(K):=\left(\frac{|K|}{|B_n|}\right)^{1/n}.
\]
We shall use the following three inputs.
\begin{lemma}[Sharp volume radius]\label{Lem:SharpVolumeRadiusCorrected}
With probability at least \(1-\exp(-cn/L^2)\),
\[
    \operatorname{vrad}(\PN)^2
    =
    \alpha(n,d)\left(1+O(L^{-2})\right).
\]
Equivalently,
\[
    \left(\frac{|\PN|}{|B_n|}\right)^{2/n}
    =
    \alpha(n,d)\left(1+O(L^{-2})\right).
\]
\end{lemma}
The last lemma follows from the proof of \Cref{Theorem:LoNE}.
\begin{lemma}[Upper second moment]\label{Lem:SecondMomentUpperCorrected}
With probability at least \(1-\exp(-cn/L^2)\),
\[
    m_2(\PN)
    \le
    \frac{n}{n+2}\,
    \alpha(n,d)
    \left(1+O(L^{-2})\right).
\]
\end{lemma}

\begin{lemma}[Lower second moment]\label{Lem:SecondMomentLowerCorrected}
With probability at least \(1-\exp(-cn/L^2)\),
\[
    m_2(\PN)
    \ge
    \frac{n}{n+2}\,
    \alpha(n,d)
    \left(1-O(L^{-2})\right).
\]
\end{lemma}
\begin{proof}[Proof of \Cref{Lem:SecondMomentUpperCorrected}]
It suffices to prove the corresponding estimate for the normalized body
\[
    \SPO=M_{n,d}\PN.
\]
Namely, we prove that with probability at least \(1-\exp(-cn/L^2)\),
\begin{equation}\label{Eq:ScaledSecondMomentCorrected}
    \frac1{|\SPO|}
    \int_{\SPO}\|x\|_2^2\,dx
    \le
    \frac{n}{n+2}\,
    n\left(1+O(L^{-2})\right).
\end{equation}
Since
\[
    M_{n,d}
    =
    \left(1+O(L^{-2})\right)
    \sqrt{\frac{n}{\alpha(n,d)}},
\]
the estimate \eqref{Eq:ScaledSecondMomentCorrected} is equivalent to
\[
    m_2(\PN)
    \le
    \frac{n}{n+2}
    \alpha(n,d)
    \left(1+O(L^{-2})\right).
\]

Let \(C\) be a cone sampled according to the volume-biased Fleury law
\(\mathrm{Fl}^{\mathrm{vol}}\). We call \(C\) good if
\begin{equation}\label{Eq:GoodConeCorrected}
    T_{n,d}^2
    +
    \E_{V\sim U(\simplex_c)}
    \|\bar Y+\vec YV\|_2^2
    \le
    \alpha(n,d)\left(1+C_0L^{-2}\right).
\end{equation}
The volume-biased Fleury transfer lemma, together with the height tail for
\(T_{n,d}\), Frobenius/Wishart concentration for the tangential Gaussian part,
and the Goodman determinant estimate, gives
\[
    \P_{\mathrm{Fl}^{\mathrm{vol}}}\{C\text{ is bad}\}
    \le
    \exp(-cn/L^2).
\]
Equivalently, in annealed form,
\[
    \E\big|\text{bad cone volume in }\PN\big|
    \le
    \exp(-cn/L^2)\E|\PN|.
\]
By Markov's inequality and the high-probability lower volume estimate from
\Cref{Lem:SharpVolumeRadiusCorrected}, with probability at least
\(1-\exp(-cn/L^2)\), the total volume of bad cones is at most
\[
    \exp(-c'n/L^2)|\PN|.
\]

Now fix a good cone \(C\). If \(X_C\sim U(C)\), the cone sampling formula gives
\[
    \|X_C\|_2^2
    =
    R^2\left(
        T_{n,d}^2+\|\bar Y+\vec YV\|_2^2
    \right),
\]
where
\[
    \E R^2=\frac{n}{n+2}.
\]
Therefore, by \eqref{Eq:GoodConeCorrected},
\[
    \frac1{|C|}\int_C\|x\|_2^2\,dx
    \le
    \frac{n}{n+2}
    \alpha(n,d)\left(1+C_0L^{-2}\right).
\]
After scaling by \(M_{n,d}\), this becomes
\[
    \frac1{|M_{n,d}C|}
    \int_{M_{n,d}C}\|x\|_2^2\,dx
    \le
    \frac{n}{n+2}
    n\left(1+O(L^{-2})\right).
\]

It remains to handle the bad cones. On the same high-probability event, the
standard outradius estimate for the normalized Gaussian polytope gives
\[
    R(\SPO)^2\lesssim n.
\]
Since the union of bad cones has exponentially small relative volume, its
contribution to the average second moment is
\[
    O(n)\exp(-c'n/L^2),
\]
which is absorbed into the \(O(L^{-2})\) error term.

Summing the contributions of good and bad cones yields
\[
    \frac1{|\SPO|}
    \int_{\SPO}\|x\|_2^2\,dx
    \le
    \frac{n}{n+2}
    n\left(1+O(L^{-2})\right),
\]
which proves \eqref{Eq:ScaledSecondMomentCorrected} and hence the lemma.
\end{proof}
\begin{proof}[Proof of \Cref{Lem:SecondMomentLowerCorrected}]
Among all measurable sets of a fixed volume, the Euclidean ball centered at the
origin minimizes the second moment about the origin. Hence, for every convex
body \(K\subset\mathbb R^n\),
\[
    m_2(K)
    \ge
    \frac{n}{n+2}\operatorname{vrad}(K)^2.
\]
Applying this to \(K=\PN\) and using \Cref{Lem:SharpVolumeRadiusCorrected}
gives
\[
    m_2(\PN)
    \ge
    \frac{n}{n+2}
    \alpha(n,d)
    \left(1-O(L^{-2})\right).
\]
\end{proof}
\begin{proof}[Proof of \Cref{Thm:Isotropic}]
Let \(K\subset\mathbb R^n\) be a centrally symmetric convex body. Its barycenter
is zero, and its isotropic constant is
\[
    L_K^2
    =
    |K|^{-2/n}
    \det(\operatorname{Cov}(K))^{1/n},
\]
where
\[
    \operatorname{Cov}(K)
    =
    \frac1{|K|}
    \int_K xx^\top\,dx.
\]
Since
\[
    \operatorname{tr}(\operatorname{Cov}(K))
    =
    \frac1{|K|}\int_K \|x\|_2^2\,dx
    =
    m_2(K),
\]
the arithmetic-geometric mean inequality gives
\[
    \det(\operatorname{Cov}(K))^{1/n}
    \le
    \frac1n\operatorname{tr}(\operatorname{Cov}(K))
    =
    \frac1n m_2(K).
\]
Therefore
\begin{equation}\label{Eq:IsotropicUpperCorrected}
    L_K^2
    \le
    \frac1n |K|^{-2/n}m_2(K).
\end{equation}

For the Euclidean ball,
\[
    L_{B_n}^2
    =
    \frac{|B_n|^{-2/n}}{n+2}.
\]
Dividing \eqref{Eq:IsotropicUpperCorrected} by \(L_{B_n}^2\), we obtain
\begin{equation}\label{Eq:IsotropicRatioUpperCorrected}
    \frac{L_K^2}{L_{B_n}^2}
    \le
    \frac{n+2}{n}
    \frac{m_2(K)}{\operatorname{vrad}(K)^2}.
\end{equation}

Apply \eqref{Eq:IsotropicRatioUpperCorrected} to \(K=\PN\). By
\Cref{Lem:SecondMomentUpperCorrected},
\[
    m_2(\PN)
    \le
    \frac{n}{n+2}
    \alpha(n,d)
    \left(1+O(L^{-2})\right),
\]
and by \Cref{Lem:SharpVolumeRadiusCorrected},
\[
    \operatorname{vrad}(\PN)^2
    =
    \alpha(n,d)
    \left(1+O(L^{-2})\right).
\]
Hence, with probability at least \(1-\exp(-cn/L^2)\),
\[
    \frac{L_{\PN}^2}{L_{B_n}^2}
    \le
    1+O(L^{-2}).
\]

The reverse inequality
\[
    L_{\PN}\ge L_{B_n}
\]
holds for every centrally symmetric convex body, since ellipsoids, and hence
Euclidean balls after affine normalization, minimize the isotropic constant.
Consequently,
\[
    1
    \le
    \frac{L_{\PN}^2}{L_{B_n}^2}
    \le
    1+O(L^{-2}).
\]
Taking square roots gives
\[
    L_{\PN}
    =
    \left(1+O(L^{-2})\right)L_{B_n}.
\]
\end{proof}
\begin{proof}[Proof of \Cref{Thm:ThinShell}]
We first prove the annealed squared-radius thin-shell estimate
\begin{equation}\label{Eq:ThinShellFirstCorrected}
    \E\int_{\PN}
    \left(
        \frac{\|x\|_2^2}{\alpha(n,d)}-1
    \right)^2dx
    \lesssim
    \frac{\E|\PN|}{nL^2}.
\end{equation}

By the volume-biased Fleury identity, applied to
\[
    \varphi(x)
    =
    \left(
        \frac{\|x\|_2^2}{\alpha(n,d)}-1
    \right)^2,
\]
it is enough to prove that
\begin{equation}\label{Eq:VolumeBiasedThinShellGoalCorrected}
    \E_{\mathrm{Fl}^{\mathrm{vol}}}
    \frac1{|C|}
    \int_C
    \left(
        \frac{\|x\|_2^2}{\alpha(n,d)}-1
    \right)^2dx
    \lesssim
    \frac1{nL^2}.
\end{equation}

Let \(X_C\sim U(C)\). By the cone sampling formula,
\[
    \|X_C\|_2^2
    =
    R^2
    \left(
        T_{n,d}^2+\|\bar Y+\vec YV\|_2^2
    \right),
\]
where \(R\) is independent of \((T_{n,d},\vec Y,V)\) and satisfies
\[
    \E R^2=\frac{n}{n+2},
    \qquad
    \E(R^2-1)^2\lesssim n^{-2}.
\]
Write
\[
    A:=T_{n,d}^2-(\alpha(n,d)-2),
    \qquad
    B:=\|\bar Y+\vec YV\|_2^2-2,
    \qquad
    D:=R^2-1.
\]
Then
\[
    \|X_C\|_2^2-\alpha(n,d)
    =
    A+B+D\left(T_{n,d}^2+\|\bar Y+\vec YV\|_2^2\right).
\]
Therefore,
\[
    \left(
        \frac{\|X_C\|_2^2}{\alpha(n,d)}-1
    \right)^2
    \lesssim
    \left(\frac{A}{\alpha(n,d)}\right)^2
    +
    \left(\frac{B}{\alpha(n,d)}\right)^2
    +
    D^2
    \left(
        \frac{T_{n,d}^2+\|\bar Y+\vec YV\|_2^2}{\alpha(n,d)}
    \right)^2.
\]
We estimate these three terms. First, by the volume-biased Fleury height estimate,
\[
    \E_{\mathrm{Fl}^{\mathrm{vol}}}
    \left(
        \frac{T_{n,d}^2-(\alpha(n,d)-2)}{\alpha(n,d)}
    \right)^2
    \lesssim
    \frac1{nL^2}.
\]
Second, by the volume-biased tangential moment estimate,
\[
    \E_{\mathrm{Fl}^{\mathrm{vol}}}
    \E_{V\sim U(\simplex_c)}
    \left(
        \|\bar Y+\vec YV\|_2^2-2
    \right)^2
    \lesssim
    \frac1n.
\]
Since \(\alpha(n,d)\asymp L\), this contributes at most
\[
    \frac1{\alpha(n,d)^2}\cdot \frac1n
    \lesssim
    \frac1{nL^2}.
\]
Third, since \(R\) is independent,
\[
    \E(R^2-1)^2\lesssim n^{-2},
\]
and
\[
    T_{n,d}^2+\|\bar Y+\vec YV\|_2^2
    =
    O_{L^2}(\alpha(n,d))
\]
under the volume-biased Fleury law, the radial contribution is \(O(n^{-2})\),
which is smaller than \(O((nL^2)^{-1})\). Combining the three estimates gives
\[
    \E_{\mathrm{Fl}^{\mathrm{vol}}}
    \E_{X_C\sim U(C)}
    \left(
        \frac{\|X_C\|_2^2}{\alpha(n,d)}-1
    \right)^2
    \lesssim
    \frac1{nL^2}.
\]
This proves \eqref{Eq:VolumeBiasedThinShellGoalCorrected}, and hence
\eqref{Eq:ThinShellFirstCorrected} follows from the volume-biased Fleury
identity.

We now prove the variance-normalized statement for the radius itself. Recall
that in the statement of the theorem,
\[
    m(\PN)
    :=
    \frac1{|\PN|}\int_{\PN}\|x\|_2\,dx.
\]
Let
\[
    R(x):=\|x\|_2.
\]
For each fixed realization of \(\PN\), the value \(m(\PN)\) minimizes
\[
    a\mapsto
    \int_{\PN}(R(x)-a)^2\,dx.
\]
Therefore,
\begin{equation}\label{Eq:RadiusVarianceMinimizerCorrected}
    \int_{\PN}
    (R(x)-m(\PN))^2\,dx
    \le
    \int_{\PN}
    \left(R(x)-\sqrt{\alpha(n,d)}\right)^2\,dx.
\end{equation}

Let \(\mathcal G\) denote the good event on which
\Cref{Lem:SharpVolumeRadiusCorrected,Lem:SecondMomentUpperCorrected,Lem:SecondMomentLowerCorrected}
hold, together with the corresponding annealed bad-volume estimates used
above. On \(\mathcal G\), by \Cref{Lem:SecondMomentLowerCorrected},
\[
    m_2(\PN)
    =
    \frac1{|\PN|}\int_{\PN}\|x\|_2^2\,dx
    \gtrsim
    \alpha(n,d).
\]
Since \(Z\sim U(\PN)\) is log-concave and \(z\mapsto \|z\|_2\) is a seminorm,
Borell's lemma gives the moment comparison
\[
    \left(\E\|Z\|_2^2\right)^{1/2}
    \lesssim
    \E\|Z\|_2.
\]
Equivalently,
\[
    m_2(\PN)^{1/2}
    \lesssim
    m(\PN).
\]
Thus, on \(\mathcal G\),
\[
    m(\PN)^2
    \gtrsim
    m_2(\PN)
    \gtrsim
    \alpha(n,d).
\]

Using \eqref{Eq:RadiusVarianceMinimizerCorrected}, we obtain on
\(\mathcal G\)
\[
\begin{aligned}
    |\PN|\,
    \Var_{Z\sim U(\PN)}
    \left(
        \frac{\|Z\|_2}{m(\PN)}
    \right)
    &=
    \frac1{m(\PN)^2}
    \int_{\PN}(R(x)-m(\PN))^2\,dx   \\
    &\lesssim
    \frac1{\alpha(n,d)}
    \int_{\PN}
    \left(R(x)-\sqrt{\alpha(n,d)}\right)^2\,dx .
\end{aligned}
\]
Moreover,
\[
    \left(R-\sqrt{\alpha(n,d)}\right)^2
    =
    \frac{\left(R^2-\alpha(n,d)\right)^2}
    {\left(R+\sqrt{\alpha(n,d)}\right)^2}
    \le
    \frac{\left(R^2-\alpha(n,d)\right)^2}{\alpha(n,d)}.
\]
Therefore, on \(\mathcal G\),
\[
    |\PN|\,
    \Var_{Z\sim U(\PN)}
    \left(
        \frac{\|Z\|_2}{m(\PN)}
    \right)
    \lesssim
    \frac1{\alpha(n,d)^2}
    \int_{\PN}
    \left(\|x\|_2^2-\alpha(n,d)\right)^2\,dx.
\]
Equivalently,
\[
    |\PN|\,
    \Var_{Z\sim U(\PN)}
    \left(
        \frac{\|Z\|_2}{m(\PN)}
    \right)
    \lesssim
    \int_{\PN}
    \left(
        \frac{\|x\|_2^2}{\alpha(n,d)}-1
    \right)^2\,dx.
\]

Taking expectation over the good event and using
\eqref{Eq:ThinShellFirstCorrected}, we get
\[
\begin{aligned}
    \E\left[
        |\PN|\,
        \Var_{Z\sim U(\PN)}
        \left(
            \frac{\|Z\|_2}{m(\PN)}
        \right)
        \mathbf 1_{\mathcal G}
    \right]
    &\lesssim
    \E\int_{\PN}
    \left(
        \frac{\|x\|_2^2}{\alpha(n,d)}-1
    \right)^2dx  \\
    &\lesssim
    \frac{\E|\PN|}{nL^2}.
\end{aligned}
\]

It remains to control the complement of the good event. By Borell's lemma,
for every centrally symmetric convex body \(K\),
\[
    \E_{Z\sim U(K)}\|Z\|_2^2
    \lesssim
    \left(\E_{Z\sim U(K)}\|Z\|_2\right)^2.
\]
Hence
\[
    \Var_{Z\sim U(K)}
    \left(
        \frac{\|Z\|_2}{\E\|Z\|_2}
    \right)
    =
    \frac{\E\|Z\|_2^2}{(\E\|Z\|_2)^2}-1
    \lesssim
    1.
\]
Applying this with \(K=\PN\), we have
\[
    |\PN|\,
    \Var_{Z\sim U(\PN)}
    \left(
        \frac{\|Z\|_2}{m(\PN)}
    \right)
    \lesssim
    |\PN|.
\]
The same height, determinant, volume-radius, and outradius estimates used above
give the annealed bad-volume bound
\[
    \E\left[|\PN|\mathbf 1_{\mathcal G^c}\right]
    \le
    \exp(-c n/L^2)\E|\PN|.
\]
Since
\[
    \exp(-c n/L^2)
    \lesssim
    \frac1{nL^2}
\]
in the regime under consideration, the contribution of \(\mathcal G^c\) is
absorbed into the main term:
\[
    \E\left[
        |\PN|\,
        \Var_{Z\sim U(\PN)}
        \left(
            \frac{\|Z\|_2}{m(\PN)}
        \right)
        \mathbf 1_{\mathcal G^c}
    \right]
    \lesssim
    \frac{\E|\PN|}{nL^2}.
\]

Combining the good and bad event estimates gives
\[
    \E\left[
    |\PN|\cdot
    \Var_{Z\sim U(\PN)}
    \left(
        \frac{\|Z\|_2}{m(\PN)}
    \right)
    \right]
    \lesssim
    \frac{\E|\PN|}{nL^2}.
\]
This proves \Cref{Thm:ThinShell}.
\end{proof}
\begin{proof}[Proof of \Cref{C:Fleury}]
Let
\[
    f_{\PN}:=\frac1{|\PN|}\int_{\PN}f(x)\,dx
\]
and, for each facet \(\mathcal F\),
\[
    C_{\mathcal F}:=\conv(0,\mathcal F),
    \qquad
    f_{C_{\mathcal F}}
    :=
    \frac1{|C_{\mathcal F}|}
    \int_{C_{\mathcal F}}f(x)\,dx.
\]
Since the facet cones partition \(\PN\) up to a null set,
\[
    \int_{\PN}(f-f_{\PN})^2dx
    =
    \sum_{\mathcal F}
    \int_{C_{\mathcal F}}(f-f_{C_{\mathcal F}})^2dx
    +
    \sum_{\mathcal F}
    |C_{\mathcal F}|(f_{C_{\mathcal F}}-f_{\PN})^2.
\]
Write these two terms as \(\mathrm I\) and \(\mathrm{II}\).

First, by the Poincare inequality on a simplex cone and the fact that \(f\) is
\(1\)-Lipschitz,
\[
    \int_C(f-f_C)^2dx
    \lesssim
    \frac1n\int_C\|x\|_2^2dx
\]
for every facet cone \(C\). Summing over all cones gives
\[
    \mathrm I
    \lesssim
    \frac1n\int_{\PN}\|x\|_2^2dx.
\]
Therefore
\[
    \E\mathrm I
    \lesssim
    \frac1n
    \E\int_{\PN}\|x\|_2^2dx.
\]

We now control \(\mathrm{II}\). For a fixed realization, \(f_{\PN}\) minimizes
\[
    a\mapsto
    \sum_{\mathcal F}
    |C_{\mathcal F}|(f_{C_{\mathcal F}}-a)^2.
\]
Hence, for every deterministic \(a\),
\[
    \mathrm{II}
    \le
    \sum_{\mathcal F}
    |C_{\mathcal F}|(f_{C_{\mathcal F}}-a)^2.
\]
Taking expectations and using the volume-biased Fleury reduction gives
\[
    \E\mathrm{II}
    \le
    \E|\PN|\,
    \E_{\mathrm{Fl}^{\mathrm{vol}}}(M_C-a)^2,
\]
where
\[
    M_C:=\frac1{|C|}\int_C f(x)\,dx.
\]
Choosing \(a=\E_{\mathrm{Fl}^{\mathrm{vol}}}M_C\), we get
\[
    \E\mathrm{II}
    \le
    \E|\PN|\,
    \Var_{\mathrm{Fl}^{\mathrm{vol}}}(M_C).
\]

Under the volume-biased Fleury law, write
\[
    C=UC_A,
    \qquad
    A=
    \begin{bmatrix}
        \vec Y\\
        T_{n,d}\mathbf 1_n^\top
    \end{bmatrix},
\]
where \(U\) is Haar orthogonal and independent of \(A\). A uniform point in
\(C\) has the representation
\[
    X_C\stackrel d=UR A\Lambda.
\]
where
\[
R\sim n r^{n-1}\mathbf 1_{[0,1]}(r),dr,
\qquad
\Lambda\sim U(\Delta_{n-1}).
\]
Therefore
\[
    M_C
    =
    \E_{R,\Lambda}f(UR A\Lambda).
\]
Set
\[
    G(A,U):=\E_{R,\Lambda}f(UR A\Lambda).
\]
By total variance,
\[
    \Var(G(A,U))
    =
    \E_A\Var_U(G(A,U)\mid A)
    +
    \Var_A(\E_U G(A,U)).
\]

For fixed \(A\), the Poincare inequality on \(SO(n)\), together with the
\(1\)-Lipschitz property of \(f\), gives
\[
    \Var_U(G(A,U)\mid A)
    \lesssim
    \frac1n
    \E_{R,\Lambda}\|R A\Lambda\|_2^2.
\]
Averaging over the volume-biased Fleury law and using the cone identity,
\[
    \E_A\Var_U(G(A,U)\mid A)
    \lesssim
    \frac1n
    \frac{\E\int_{\PN}\|x\|_2^2dx}{\E|\PN|}.
\]

It remains to control
\[
    \Var_A(\E_U G(A,U)).
\]
Let
\[
    \widetilde f(r):=\E_{\Theta\sim U(S^{n-1})}f(r\Theta).
\]
Then \(\widetilde f\) is \(1\)-Lipschitz and
\[
    \E_U G(A,U)
    =
    \E_{R,\Lambda}
    \widetilde f(\|R A\Lambda\|_2).
\]
Let \(A'\) be an independent copy of \(A\). By Jensen and the Lipschitz property
of \(\widetilde f\),
\[
\begin{aligned}
    \Var_A(\E_U G(A,U))
    &=
    \frac12
    \E_{A,A'}
    \left[
        \E_{R,\Lambda}
        \left(
            \widetilde f(\|RA\Lambda\|_2)
            -
            \widetilde f(\|RA'\Lambda\|_2)
        \right)
    \right]^2\\
    &\le
    \frac12
    \E_{A,A'}
    \E_{R,\Lambda}
    \left(
        \|RA\Lambda\|_2-\|RA'\Lambda\|_2
    \right)^2\\
    &=
    \E_{R,\Lambda}
    \Var_A(\|RA\Lambda\|_2).
\end{aligned}
\]
For any deterministic \(\rho\),
\[
    \Var_A(\|RA\Lambda\|_2)
    \le
    \E_A(\|RA\Lambda\|_2-\rho)^2.
\]
Taking \(\rho\) to be the radial center supplied by the thin-shell estimate,
we get
\[
    \Var_A(\E_U G(A,U))
    \lesssim
    \frac1n
    \frac{\E\int_{\PN}\|x\|_2^2dx}{\E|\PN|}.
\]
Therefore
\[
    \Var_{\mathrm{Fl}^{\mathrm{vol}}}(M_C)
    \lesssim
    \frac1n
    \frac{\E\int_{\PN}\|x\|_2^2dx}{\E|\PN|}.
\]
Consequently,
\[
    \E\mathrm{II}
    \lesssim
    \frac1n
    \E\int_{\PN}\|x\|_2^2dx.
\]

Combining the estimates for \(\mathrm I\) and \(\mathrm{II}\),
\[
    \E\int_{\PN}(f-f_{\PN})^2dx
    \lesssim
    \frac1n
    \E\int_{\PN}\|x\|_2^2dx.
\]
This proves the theorem.
\end{proof}

%% file: Bib.bib
@article {gordon2007gaussian,
    AUTHOR = {Gordon, Y. and Litvak, A. E. and Mendelson, S. and Pajor, A.},
     TITLE = {Gaussian averages of interpolated bodies and applications to
              approximate reconstruction},
   JOURNAL = {J. Approx. Theory},
  FJOURNAL = {Journal of Approximation Theory},
    VOLUME = {149},
      YEAR = {2007},
    NUMBER = {1},
     PAGES = {59--73},
      ISSN = {0021-9045,1096-0430},
   MRCLASS = {60G15 (60D05 62H30 94A12)},
  MRNUMBER = {2371614},
MRREVIEWER = {G.\ Schechtman},
       DOI = {10.1016/j.jat.2007.04.007},
       URL = {https://doi.org/10.1016/j.jat.2007.04.007},
}

@article{alonso2015gaussian,
  title={On the Gaussian behavior of marginals and the mean width of random polytopes},
  author={Alonso-Guti{\'e}rrez, David and Prochno, Joscha},
  journal={Proceedings of the American Mathematical Society},
  volume={143},
  number={2},
  pages={821--832},
  year={2015}
}

@article {gluskin88extremal,
    AUTHOR = {Gluskin, E. D.},
     TITLE = {Extremal properties of orthogonal parallelepipeds and their
              applications to the geometry of {B}anach spaces},
   JOURNAL = {Mat. Sb. (N.S.)},
  FJOURNAL = {Matematicheski\u{\i} Sbornik. Novaya Seriya},
    VOLUME = {136(178)},
      YEAR = {1988},
    NUMBER = {1},
     PAGES = {85--96}
}

@article{bizeul2025slicing,
  title={The slicing conjecture via small ball estimates},
  author={Bizeul, Pierre},
  journal={arXiv preprint arXiv:2501.06854},
  year={2025}
}

@article{milman2015mean,
  title={On the mean-width of isotropic convex bodies and their associated L p-centroid bodies},
  author={Milman, Emanuel},
  journal={International Mathematics Research Notices},
  volume={2015},
  number={11},
  pages={3408--3423},
  year={2015},
  publisher={Oxford University Press}
}

@inproceedings{paouris2004estimates,
  title={$\Psi_2$-Estimates for Linear Functionals on Zonoids},
  author={Paouris, Grigoris},
  booktitle={Geometric Aspects of Functional Analysis: Israel Seminar 2001-2002},
  pages={211--222},
  year={2004},
  organization={Springer}
}

@article{guan2024note,
  title={A note on Bourgain's slicing problem},
  author={Guan, Qingyang},
  journal={arXiv preprint arXiv:2412.09075},
  year={2024}
}

@article{klartag2025affirmative,
  title={Affirmative resolution of Bourgain’s slicing problem using Guan’s bound},
  author={Klartag, Boaz and Lehec, Joseph},
  journal={Geometric and Functional Analysis},
  pages={1--22},
  year={2025},
  publisher={Springer}
}

@article{efron1965convex,
  title={The convex hull of a random set of points},
  author={Efron, Bradley},
  journal={Biometrika},
  volume={52},
  number={3-4},
  pages={331--343},
  year={1965},
  publisher={Oxford University Press}
}

@article{paouris2019gaussian,
  title={Gaussian Convex Bodies: a Nonasymptotic Approach},
  author={Paouris, G and Pivovarov, P and Valettas, P},
  journal={Journal of Mathematical Sciences},
  volume={238},
  number={4},
  pages={537--559},
  year={2019},
  publisher={Springer}
}

@article{fleury2012poincare,
  title={Poincar{\'e} inequality in mean value for Gaussian polytopes},
  author={Fleury, B},
  journal={Probability theory and related fields},
  volume={152},
  pages={141--178},
  year={2012},
  publisher={Springer}
}

@article{goodman1963distribution,
  title={The distribution of the determinant of a complex Wishart distributed matrix},
  author={Goodman, NR},
  journal={The Annals of mathematical statistics},
  volume={34},
  number={1},
  pages={178--180},
  year={1963},
  publisher={JSTOR}
}

@article {guedon22geometry,
    AUTHOR = {Gu\'{e}don, Olivier and Krahmer, Felix and K\"{u}mmerle, Christian and
              Mendelson, Shahar and Rauhut, Holger},
     TITLE = {On the geometry of polytopes generated by heavy-tailed random
              vectors},
   JOURNAL = {Commun. Contemp. Math.},
  FJOURNAL = {Communications in Contemporary Mathematics},
    VOLUME = {24},
      YEAR = {2022},
    NUMBER = {3},
     PAGES = {Paper No. 2150056, 31},
      ISSN = {0219-1997},
   MRCLASS = {52A22 (15B52 46B06 46B09 52A23 60B20 65K10)},
  MRNUMBER = {4400194},
       DOI = {10.1142/S0219199721500565},
       URL = {https://doi.org/10.1142/S0219199721500565},
}

@article{kabluchko2019expected,
  title={Expected volumes of Gaussian polytopes, external angles, and multiple order statistics},
  author={Kabluchko, Zakhar and Zaporozhets, Dmitry},
  journal={Transactions of the American Mathematical Society},
  volume={372},
  number={3},
  pages={1709--1733},
  year={2019}
}

@article{paouris2006concentration,
  title={Concentration of mass on convex bodies},
  author={Paouris, Grigoris},
  journal={Geometric \& Functional Analysis GAFA},
  volume={16},
  number={5},
  pages={1021--1049},
  year={2006},
  publisher={Springer}
}

@article{klartag2009hyperplane,
  title={On the hyperplane conjecture for random convex sets},
  author={Klartag, Bo’az and Kozma, Gady},
  journal={Israel Journal of Mathematics},
  volume={170},
  number={1},
  pages={253--268},
  year={2009},
  publisher={Springer}
}

@article{Can08,
  author  = {Emmanuel J. Candes},
  title   = {The restricted isometry property and its implications for compressed sensing},
  journal = {Comptes Rendus Mathematique},
  volume  = {346},
  number  = {9--10},
  pages   = {589--592},
  year    = {2008}
}

@article{DE06,
  author  = {David L. Donoho and Michael Elad},
  title   = {On the stability of the basis pursuit in the presence of noise},
  journal = {Signal Processing},
  volume  = {86},
  number  = {3},
  pages   = {511--532},
  year    = {2006}
}

@inproceedings{JLL20,
  author    = {Peizhong Ju and Xiaojun Lin and Jia Liu},
  title     = {Overfitting Can Be Harmless for Basis Pursuit, But Only to a Degree},
  booktitle = {Advances in Neural Information Processing Systems},
  volume    = {33},
  year      = {2020}
}

@article{Woj10,
  author  = {P. Wojtaszczyk},
  title   = {Stability and instance optimality for Gaussian measurements in compressed sensing},
  journal = {Foundations of Computational Mathematics},
  volume  = {10},
  number  = {1},
  pages   = {1--13},
  year    = {2010}
}

@article{Fou14,
  author  = {Simon Foucart},
  title   = {Stability and robustness of $\ell_1$-minimizations with Weibull matrices and redundant dictionaries},
  journal = {Linear Algebra and its Applications},
  volume  = {441},
  pages   = {4--21},
  year    = {2014}
}

@article{KKR18,
  author  = {Felix Krahmer and Christian K{\"u}mmerle and Holger Rauhut},
  title   = {A quotient property for {RIP} matrices with heavy-tailed entries and its application to noise-blind compressed sensing},
  journal = {arXiv preprint arXiv:1806.04261},
  year    = {2018}
}

@article{CL21,
  author  = {Niladri S. Chatterji and Philip M. Long},
  title   = {Foolish crowds support benign overfitting},
  journal = {arXiv preprint arXiv:2110.02941},
  year    = {2021}
}

@article{MVSS20,
  author  = {Vidya Muthukumar and Kailas Vodrahalli and Vignesh Subramanian and Anant Sahai},
  title   = {Harmless interpolation of noisy data in regression},
  journal = {IEEE Journal on Selected Areas in Information Theory},
  volume  = {1},
  number  = {1},
  pages   = {67--83},
  year    = {2020}
}

@article{barthe2009remarks,
  title={Remarks on non-interacting conservative spin systems: the case of gamma distributions},
  author={Barthe, Franck and Wolff, Pawe{\l}},
  journal={Stochastic processes and their applications},
  volume={119},
  number={8},
  pages={2711--2723},
  year={2009},
  publisher={Elsevier}
}

@article{giannopoulos2000extremal,
  title={Extremal problems and isotropic positions of convex bodies},
  author={Giannopoulos, Apostolos A and Milman, Vitali D},
  journal={Israel Journal of Mathematics},
  volume={117},
  pages={29--60},
  year={2000},
  publisher={Springer}
}

@article {chen1998atomic,
    AUTHOR = {Chen, Scott Shaobing and Donoho, David L. and Saunders,
              Michael A.},
     TITLE = {Atomic decomposition by basis pursuit},
   JOURNAL = {SIAM J. Sci. Comput.},
  FJOURNAL = {SIAM Journal on Scientific Computing},
    VOLUME = {20},
      YEAR = {1998},
    NUMBER = {1},
     PAGES = {33--61},
      ISSN = {1064-8275},
   MRCLASS = {94A12 (41A45 65K05)},
  MRNUMBER = {1639094},
       DOI = {10.1137/S1064827596304010},
       URL = {https://doi.org/10.1137/S1064827596304010},
}

@book {chatterjee2014superbook,
    AUTHOR = {Chatterjee, Sourav},
     TITLE = {Superconcentration and related topics},
    SERIES = {Springer Monographs in Mathematics},
 PUBLISHER = {Springer, Cham},
      YEAR = {2014},
     PAGES = {x+156},
      ISBN = {978-3-319-03885-8; 978-3-319-03886-5},
   MRCLASS = {60E15 (60G15 60G60 60G70 60K35)},
  MRNUMBER = {3157205},
MRREVIEWER = {Vladislav Kargin},
       DOI = {10.1007/978-3-319-03886-5},
       URL = {https://doi.org/10.1007/978-3-319-03886-5},
}

@article {talagrand1994russo,
    AUTHOR = {Talagrand, Michel},
     TITLE = {On {R}usso's approximate zero-one law},
   JOURNAL = {Ann. Probab.},
  FJOURNAL = {The Annals of Probability},
    VOLUME = {22},
      YEAR = {1994},
    NUMBER = {3},
     PAGES = {1576--1587}
}

@inproceedings{wang2022tight,
  title={Tight bounds for minimum in lone-norm interpolation of noisy data},
  author={Wang, Guillaume and Donhauser, Konstantin and Yang, Fanny},
  booktitle={International Conference on Artificial Intelligence and Statistics},
  pages={10572--10602},
  year={2022},
  organization={PMLR}
}

@article{Paorlp,
title = {Random version of Dvoretzky’s theorem in $\ell_p$},
journal = {Stochastic Processes and their Applications},
volume = {127},
number = {10},
pages = {3187-3227},
year = {2017},
author = {Grigoris Paouris and Petros Valettas and Joel Zinn}
}

@article{muthukumar2020,
  title={Harmless interpolation of noisy data in regression},
  author={Muthukumar, Vidya and Vodrahalli, Kailas and Subramanian, Vignesh and Sahai, Anant},
  journal={IEEE Journal on Selected Areas in Information Theory},
  volume={1},
  number={1},
  pages={67--83},
  year={2020},
  publisher={IEEE}
}

@inproceedings{thramp2015,
  title={Regularized linear regression: A precise analysis of the estimation error},
  author={Thrampoulidis, Christos and Oymak, Samet and Hassibi, Babak},
  booktitle={Conference on Learning Theory},
  pages={1683--1709},
  year={2015},
  organization={PMLR}
}

@article{gromov1983topological,
  title={A topological application of the isoperimetric inequality},
  author={Gromov, Mikhael and Milman, Vitali D},
  journal={American Journal of Mathematics},
  volume={105},
  number={4},
  pages={843--854},
  year={1983},
  publisher={JSTOR}
}

@article{hastie2022surprises,
  title={Surprises in high-dimensional ridgeless least squares interpolation},
  author={Hastie, Trevor and Montanari, Andrea and Rosset, Saharon and Tibshirani, Ryan J},
  journal={The Annals of Statistics},
  volume={50},
  number={2},
  pages={949--986},
  year={2022},
  publisher={Institute of Mathematical Statistics}
}

@article{chinot2020robustness,
  title={On the robustness of minimum norm interpolators and regularized empirical risk minimizers},
  author={Chinot, Geoffrey and L{\"o}ffler, Matthias and van de Geer, Sara},
  journal={arXiv preprint arXiv:2012.00807},
  year={2020}
}

@article{bartlett2002rademacher,
  title={Rademacher and Gaussian complexities: Risk bounds and structural results},
  author={Bartlett, Peter L and Mendelson, Shahar},
  journal={Journal of Machine Learning Research},
  volume={3},
  number={Nov},
  pages={463--482},
  year={2002}
}

@article{donhauser2022fast,
  title={Fast rates for noisy interpolation require rethinking the effects of inductive bias},
  author={Donhauser, Konstantin and Ruggeri, Nicolo and Stojanovic, Stefan and Yang, Fanny},
  journal={arXiv preprint arXiv:2203.03597},
  year={2022}
}

@article{nakkiran2021,
  title={Deep double descent: Where bigger models and more data hurt},
  author={Nakkiran, Preetum and Kaplun, Gal and Bansal, Yamini and Yang, Tristan and Barak, Boaz and Sutskever, Ilya},
  journal={Journal of Statistical Mechanics: Theory and Experiment},
  volume={2021},
  number={12},
  pages={124003},
  year={2021},
  publisher={IOP Publishing}
}

@book{brazitikos2014geometry,
  title={Geometry of isotropic convex bodies},
  author={Brazitikos, Silouanos and Giannopoulos, Apostolos and Valettas, Petros and Vritsiou, Beatrice-Helen},
  volume={196},
  year={2014},
  publisher={American Mathematical Soc.}
}

@article{bartlett2020benign,
  title={Benign overfitting in linear regression},
  author={Bartlett, Peter L and Long, Philip M and Lugosi, G{\'a}bor and Tsigler, Alexander},
  journal={Proceedings of the National Academy of Sciences},
  volume={117},
  number={48},
  pages={30063--30070},
  year={2020},
  publisher={National Acad Sciences}
}

@article{koehler2021uniform,
  title={Uniform convergence of interpolators: Gaussian width, norm bounds and benign overfitting},
  author={Koehler, Frederic and Zhou, Lijia and Sutherland, Danica J and Srebro, Nathan},
  journal={Advances in Neural Information Processing Systems},
  volume={34},
  pages={20657--20668},
  year={2021}
}

@article{zhou2023uniform,
  title={Uniform Convergence with Square-Root Lipschitz Loss},
  author={Zhou, Lijia and Dai, Zhen and Koehler, Frederic and Srebro, Nathan},
  journal={arXiv preprint arXiv:2306.13188},
  year={2023}
}

@book{artstein2015asymptotic,
  title={Asymptotic geometric analysis, Part I},
  author={Artstein-Avidan, Shiri and Giannopoulos, Apostolos and Milman, Vitali D},
  volume={202},
  year={2015},
  publisher={American Mathematical Society},
  series = {Mathematical Surveys and Monographs}
}

@article{bartlett2005local,
  title={Local {R}ademacher complexities},
  author={Bartlett, Peter L and Bousquet, Olivier and Mendelson, Shahar},
  journal={The Annals of Statistics},
  volume={33},
  number={4},
  pages={1497--1537},
  year={2005},
  publisher={Institute of Mathematical Statistics}
}

@article{chatterjee2014new,
	title={A new perspective on least squares under convex constraint},
	author={Chatterjee, Sourav},
	journal={The Annals of Statistics},
	volume={42},
	number={6},
	pages={2340--2381},
	year={2014},
	publisher={Institute of Mathematical Statistics}
}

@article{kur2026minimum,
  title={Minimum Norm Interpolation via The Local Theory of Banach Spaces: The Role of $2 $-Uniform Convexity},
  author={Kur, Gil and Bizeul, Pierre},
  journal={arXiv preprint arXiv:2603.28956},
  year={2026}
}

@article{bizeul2025distances,
  title={Distances between non-symmetric convex bodies: optimal bounds up to polylog},
  author={Bizeul, Pierre and Klartag, Boaz},
  journal={arXiv preprint arXiv:2510.20511},
  year={2025}
}

@article{donoho2009counting,
  title={Counting faces of randomly projected polytopes when the projection radically lowers dimension},
  author={Donoho, David and Tanner, Jared},
  journal={Journal of the American Mathematical Society},
  volume={22},
  number={1},
  pages={1--53},
  year={2009}
}

@article{donoho2010counting,
  title={Counting the faces of randomly-projected hypercubes and orthants, with applications},
  author={Donoho, David L and Tanner, Jared},
  journal={Discrete \& computational geometry},
  volume={43},
  number={3},
  pages={522--541},
  year={2010},
  publisher={Springer}
}

@article{lecue2023geometrical,
  title={A geometrical viewpoint on the benign overfitting property of the minimum $\ell_2$-norm interpolant estimator},
  author={Lecu{\'e}, Guillaume and Shang, Zong},
  journal={arXiv preprint arXiv:2203.05873},
  year={2022}
}

@InProceedings{pmlr-v235-kur24a,
  title = 	 {Minimum Norm Interpolation Meets The Local Theory of Banach Spaces},
  author =       {Kur, Gil and Abdalla, Pedro and Bizeul, Pierre and Yang, Fanny},
  booktitle = 	 {Proceedings of the 41st International Conference on Machine Learning},
  pages = 	 {25726--25754},
  year = 	 {2024},
  volume = 	 {235},
  series = 	 {Proceedings of Machine Learning Research},
  month = 	 {21--27 Jul},
  publisher =    {PMLR}
}

@article{paouris2016gaussian,
  title={A Gaussian small deviation inequality for convex functions},
  author={Paouris, Grigoris and Valettas, Petros},
  journal={arXiv preprint arXiv:1611.01723},
  year={2016}
}

@article{klartag2007small,
  title={Small ball probability and Dvoretzky’s theorem},
  author={Klartag, Bo'az and Vershynin, Roman},
  journal={Israel Journal of Mathematics},
  volume={157},
  number={1},
  pages={193--207},
  year={2007},
  publisher={Springer}
}

@article{klartag2008volume,
  title={On volume distribution in 2-convex bodies},
  author={Klartag, Bo’az and Milman, Emanuel},
  journal={Israel Journal of Mathematics},
  volume={164},
  pages={221--249},
  year={2008},
  publisher={Springer}
}

@article{zhang2021understanding,
  title={Understanding deep learning (still) requires rethinking generalization},
  author={Zhang, Chiyuan and Bengio, Samy and Hardt, Moritz and Recht, Benjamin and Vinyals, Oriol},
  journal={Communications of the ACM},
  volume={64},
  number={3},
  pages={107--115},
  year={2021},
  publisher={ACM New York, NY, USA}
}

@article{oravkin2021optimal,
  title={On optimal interpolation in linear regression},
  author={Oravkin, Eduard and Rebeschini, Patrick},
  journal={Advances in Neural Information Processing Systems},
  volume={34},
  pages={29116--29128},
  year={2021}
}

@inproceedings{gunasekar2018characterizing,
  title={Characterizing implicit bias in terms of optimization geometry},
  author={Gunasekar, Suriya and Lee, Jason and Soudry, Daniel and Srebro, Nathan},
  booktitle={International Conference on Machine Learning},
  pages={1832--1841},
  year={2018},
  organization={PMLR}
}

@article{mei2022,
  title={The generalization error of random features regression: Precise asymptotics and the double descent curve},
  author={Mei, Song and Montanari, Andrea},
  journal={Communications on Pure and Applied Mathematics},
  volume={75},
  number={4},
  pages={667--766},
  year={2022},
  publisher={Wiley Online Library}
}

@inproceedings{shamir2022implicit,
  title={The implicit bias of benign overfitting},
  author={Shamir, Ohad},
  booktitle={Conference on Learning Theory},
  pages={448--478},
  year={2022},
  organization={PMLR}
}

@article{tsigler2023benign,
  title={Benign overfitting in ridge regression},
  author={Tsigler, Alexander and Bartlett, Peter L},
  journal={Journal of Machine Learning Research},
  volume={24},
  number={123},
  pages={1--76},
  year={2023}
}

@article{ghorbani-2021,
author = {Behrooz Ghorbani and Song Mei and Theodor Misiakiewicz and Andrea Montanari},
title = {{Linearized two-layers neural networks in high dimension}},
volume = {49},
journal = {The Annals of Statistics},
number = {2},
publisher = {Institute of Mathematical Statistics},
pages = {1029 -- 1054},
year = {2021}
}

@inproceedings{cordero2012hypercontractive,
  title={Hypercontractive measures, Talagrand’s inequality, and influences},
  author={Cordero-Erausquin, Dario and Ledoux, Michel},
  booktitle={Geometric Aspects of Functional Analysis: Israel Seminar 2006--2010},
  pages={169--189},
  year={2012},
  organization={Springer}
}

@article{paouris2,
  title={On the $\Psi_2$-behaviour of linear functionals on isotropic convex bodies},
  author={Paouris, G},
  journal={Studia Math. v168},
  pages={285--299},
  year = 2004
}

@article {milman-1986,
    AUTHOR = {Milman, Vitali D.},
     TITLE = {In\'egalit\'e{} de {B}runn-{M}inkowski inverse et applications
              \`a{} la th\'eorie locale des espaces norm\'es},
   JOURNAL = {C. R. Acad. Sci. Paris S\'er. I Math.},
  FJOURNAL = {Comptes Rendus des S\'eances de l'Acad\'emie des Sciences.
              S\'erie I. Math\'ematique},
    VOLUME = {302},
      YEAR = {1986},
    NUMBER = {1},
     PAGES = {25--28},
      ISSN = {0249-6291},
   MRCLASS = {52A40 (46B20)}
}

@book {pisier-1989,
    AUTHOR = {Pisier, Gilles},
     TITLE = {The volume of convex bodies and {B}anach space geometry},
    SERIES = {Cambridge Tracts in Mathematics},
    VOLUME = {94},
 PUBLISHER = {Cambridge University Press, Cambridge},
      YEAR = {1989},
     PAGES = {xvi+250},
      ISBN = {0-521-36465-5; 0-521-66635-X},
   MRCLASS = {52A21 (46B20 52A07)},
       DOI = {10.1017/CBO9780511662454},
       URL = {https://doi.org/10.1017/CBO9780511662454},
}
